\documentclass{article}
\usepackage{geometry}
\usepackage{amsmath,amsthm,amssymb,mathtools,microtype,needspace}
\usepackage{hyperref}
\hypersetup{
    pdftitle={Weyl's law and Pólya's conjecture for the Vladimirov-Taibleson operator},
    pdfauthor={Yaojia Sun},
    pdfsubject={Spectral geometry of the Vladimirov-Taibleson operator},
    pdfkeywords={Vladimirov-Taibleson operator, Weyl's law, Pólya's conjecture, p-adic spectral geometry}
}

\theoremstyle{plain}
\newtheorem{theorem}{Theorem}[section]
\newtheorem{lemma}[theorem]{Lemma}
\newtheorem{corollary}[theorem]{Corollary}
\newtheorem{prop}[theorem]{Proposition}

\theoremstyle{definition}
\newtheorem{definition}{Definition}[section]
\newtheorem{example}{Example}[section]

\newtheoremstyle{boldremark}
  {\topsep}{\topsep}
  {\normalfont}{}
  {\bfseries}{.}{0.5em}{}
\theoremstyle{boldremark}
\newtheorem{remark}[theorem]{Remark}

\numberwithin{equation}{section}

\title{Weyl's law and Pólya's conjecture for the 

Vladimirov-Taibleson operator}
\author{Yaojia Sun}
\date{}

\begin{document}

\maketitle

\begin{abstract}
    This paper studies some fundamental problems in spectral geometry in the $p$-adic setting. By viewing the Vladimirov-Taibleson operator $D^{\alpha}$ as the $p$-adic counterpart of the fractional Laplacian $(-\Delta)^{\frac{\alpha}{2}}$ in the Archimedean setting, we prove Weyl's law for the Dirichlet operator and establish it for the Neumann operator outside an exceptional set with zero Lebesgue asymptotic density. For the Dirichlet operator, we derive the sharp estimate for the remainder and prove that the Weyl-Berry conjecture fails. Furthermore, we show that Pólya's conjecture fails in general, and we give geometric necessary and sufficient conditions for it to hold.
\end{abstract}

\section{Introduction}
In 1912, Weyl \cite{MR1511670} showed that the Dirichlet and Neumann eigenvalue counting functions for the Laplacian $\Delta$ on a bounded open domain $\Omega\subseteq\mathbb{R}^n$ obey the following asymptotic behavior:
\begin{align*}
    N_{\mathcal{D}}(\lambda),N_{\mathcal{N}}(\lambda)\sim\frac{\omega_n}{(2\pi)^n}|\Omega|\lambda^{\frac{n}{2}},\ \lambda\to+\infty.
\end{align*}
One year later \cite{MR1580880}, he conjectured that when $\partial\Omega$ is sufficiently smooth, the asymptotic expansion should contain a subleading term determined by the $(n-1)$-dimensional measure of the boundary:
\begin{align*}
    N_{\mathcal{D}}(\lambda)\sim\frac{\omega_n}{(2\pi)^n}|\Omega|\lambda^{\frac{n}{2}}-C_{n-1}|\partial\Omega|\lambda^{\frac{n-1}{2}},\ \lambda\to+\infty;\\
    N_{\mathcal{N}}(\lambda)\sim\frac{\omega_n}{(2\pi)^n}|\Omega|\lambda^{\frac{n}{2}}+C_{n-1}|\partial\Omega|\lambda^{\frac{n-1}{2}},\ \lambda\to+\infty.
\end{align*}
In 1980, Ivriĭ \cite{MR564330} and Melrose \cite{MR573438} independently proved Weyl's conjecture for manifolds with boundary. Their theorem states that on a domain or manifold with smooth boundary, the two-term Weyl asymptotic is valid, provided that the set of periodic billiard trajectories in the dynamical system has Liouville measure zero on the unit cotangent bundle.

In 1979, Berry proposed \cite{MR556688} a generalization of Weyl's conjecture to domains with fractal boundaries. He conjectured that
\begin{align*}
    N_{\mathcal{D}}(\lambda)\sim\frac{\omega_n}{(2\pi)^n}|\Omega|\lambda^{\frac{n}{2}}-C_{n,d}\mathcal{H}^d(\partial\Omega)\lambda^{\frac{d}{2}},\ \lambda\to+\infty.
\end{align*}
Here $d\in(n-1,n)$. Later, Brossard and Carmona \cite{MR834484} proved that Berry’s original conjecture based on the Hausdorff measure is false. Inspired by their work, Lapidus \cite{MR994168} proposed the modified Weyl–Berry conjecture by replacing the Hausdorff measure by the Minkowski content:
\begin{align*}
    N_{\mathcal{D}}(\lambda)\sim\frac{\omega_n}{(2\pi)^n}|\Omega|\lambda^{\frac{n}{2}}-C_{n,d}\mathcal{M}^d(\partial\Omega)\lambda^{\frac{d}{2}},\ \lambda\to+\infty.
\end{align*}
He showed that the conjecture holds when $n=1$. Two years later, Lapidus and Pomerance \cite{MR1189091} gave a rigorous formulation of the problem for fractal drums, and in particular for one-dimensional fractal strings. However, in 1996, Lapidus and Pomerance \cite{MR1356166} showed that the conjecture fails when $n\geq2$.

In 1954, Pólya \cite{MR66321} proposed the following global inequality:
\begin{align*}
    N_{\mathcal{D}}(\lambda)\leq\frac{\omega_n}{(2\pi)^n}|\Omega|\lambda^{\frac{n}{2}}\leq N_{\mathcal{N}}(\lambda),\ \forall\lambda>0.
\end{align*}
Later he proved that this conjecture holds for tiling domains \cite{MR129219}. In a recent breakthrough, Filonov, Levitin, Polterovich and Sher \cite{MR4635832} showed that Pólya's conjecture holds for Euclidean balls: in the Dirichlet case for all dimensions, and in the Neumann case for the two-dimensional disk. Subsequently, the same authors proved Pólya's conjecture for Dirichlet eigenvalues on arbitrary planar annuli \cite{MR5028853}. However, for the fractional Laplacian $(-\Delta)^s$, Kwaśnicki, Laugesen and Siudeja \cite{MR3900781} showed that Pólya's conjecture fails.

In the non-Archimedean setting, there is no canonical Laplacian on $\mathbb{Q}_p$ or its higher-dimensional analogue $\mathbb{Q}_p^n$. Inspired by the harmonic analysis on local fields developed by Taibleson \cite{MR226394}, Vladimirov \cite{MR971464} developed a systematic theory of $p$-adic generalized functions and fractional derivatives $D^{\alpha}$. Subsequently, he determined the Dirichlet eigenvalue on the ball \cite{MR1092528}, which naturally leads to Weyl's law on the ball in $\mathbb{Q}_p$. Later Kochubei \cite{MR1209031} extended these results to compact open domains in $\mathbb{Q}_p$. In the same year, Haran \cite{MR1252936} established the $p$-adic Weyl quantization and a global symbol calculus, obtaining a $p$-adic phase-space Weyl theorem under sufficient smoothness conditions.

In 2017, Chacón-Cortés and Zúñiga \cite{MR3708861,MR3793137} studied Weyl's law for Dirichlet eigenvalues of a ball in $\mathbb{Q}_p^n$ using heat trace and spectral zeta function. However, the spectral zeta function has multiple poles on the critical line $\Re(s)=\frac{n}{\alpha}$, which prevents a direct application of the classical Ikehara Tauberian theorem to derive the asymptotic behavior of the eigenvalue counting function. This shows that some classical methods in Archimedean spaces do not apply directly in the $p$-adic setting. They conjectured that
\begin{align*}
    N_{D_{\mathbb{Z}_p^n,\mathcal{D}}^{\alpha}}(\lambda)\sim C\lambda^{\frac{n}{\alpha}},\ \lambda\to+\infty.
\end{align*}

In this paper, we study these spectral-geometric problems in the $p$-adic setting. For a bounded domain $\Omega\subseteq\mathbb{Q}_p^n$, consider the Dirichlet Vladimirov-Taibleson operator $D_{\Omega,\mathcal{D}}^{\alpha}$ (see Definition \ref{def:Dirichlet_Vladimirov-Taibleson_operator}) and the Neumann Vladimirov-Taibleson operator $D_{\Omega,\mathcal{N}}^{\alpha}$ (see Definition \ref{def:Neumann_Vladimirov-Taibleson_operator}) with the corresponding eigenvalue counting functions $N_{D_{\Omega,\mathcal{D}}^{\alpha}}(\lambda)$ and $N_{D_{\Omega,\mathcal{N}}^{\alpha}}(\lambda)$. Then our first result is:

\begin{theorem}[Weyl's law for the Vladimirov-Taibleson operator]
    Let $\Omega\subseteq\mathbb{Q}_p^n$ be a bounded, open, Jordan measurable set (i.e. $\mathcal{H}^n(\partial\Omega)=0$). Then the Dirichlet operator satisfies Weyl's law:
    \begin{align*}
        N_{D_{\Omega,\mathcal{D}}^{\alpha}}(\lambda)\sim|\Omega|p^{n[\frac{1}{\alpha}\log_p\lambda]},\ \lambda\to+\infty.
    \end{align*}
    Furthermore, write $\Omega$ as the disjoint union of its maximal balls
    $B_{L_i}(\boldsymbol a_i)$ and let
    \begin{align*}
    \Gamma_i=\frac{p^\alpha-1}{1-p^{-\alpha-n}}
    \int_{\Omega^c}
    \frac{d\boldsymbol y}
    {\|\boldsymbol x-\boldsymbol y\|_p^{n+\alpha}},
    \qquad \boldsymbol x\in B_{L_i}(\boldsymbol a_i).
    \end{align*}
    Suppose $\Gamma_{\max}:=\sup_{i\geq1}\Gamma_i<+\infty$.
    Define $E=\displaystyle\bigcup_{j=1}^{\infty}[p^{\alpha j}-\Gamma_{\max},p^{\alpha j})$. Then we have
    \begin{align*}
        \lim_{\substack{\lambda\to+\infty\\ \lambda\notin E}}\frac{N_{D_{\Omega,\mathcal{N}}^{\alpha}}(\lambda)}{|\Omega|p^{n[\frac{1}{\alpha}\log_p\lambda]}}=1,
        \qquad
        \lim_{R\to\infty}\frac{|E\cap[0,R]|}{R}=0.
    \end{align*}
    Thus the Neumann operator $D_{\Omega,\mathcal{N}}^{\alpha}$ satisfies Weyl's law outside an exceptional set $E$ of zero Lebesgue asymptotic density.
\end{theorem}
\begin{remark}
    Unlike Weyl's law in the Archimedean setting, the leading Weyl term is a step function. This is because the values of the $p$-adic norm are discrete.
\end{remark}

After establishing Weyl's law, we next estimate the remainder and consider the $p$-adic analogue of the Weyl--Berry conjecture. Define
\begin{align*}
    &\mathcal{B}_K=\{B_K(\boldsymbol{a})|\ B_K(\boldsymbol{a})\cap\Omega\neq\emptyset,\ B_K(\boldsymbol{a})\cap\Omega^c\neq\emptyset\};\\
    &\mathcal{U}_K=\{\boldsymbol{x}\in\mathbb{Q}_p^n|\ \operatorname{dist}(\boldsymbol{x},\Omega)\leq p^{-K},\ \operatorname{dist}(\boldsymbol{x},\Omega^c)\leq p^{-K}\}.
\end{align*}
Our second main result is as follows:

\begin{theorem}[Remainder estimate and Weyl-Berry conjecture for the Dirichlet operator]
    Let $\Omega\subseteq\mathbb{Q}_p^n$ be a bounded, open, Jordan measurable set. Then for $\lambda\in[p^{\alpha K},p^{\alpha(K+1)})$, we have
    \begin{align*}
        -|\mathcal{U}_K\cap\Omega^c|p^{nK}\leq|\Omega|p^{n[\frac{1}{\alpha}\log_p\lambda]}-N_{D_{\Omega,\mathcal{D}}^{\alpha}}(\lambda)\leq|\mathcal{U}_K\cap\Omega|p^{nK}.
    \end{align*}
    Furthermore, suppose that $\dim_{M}(\partial\Omega)=d$ and that $B_K(\boldsymbol{a})\in\mathcal{B}_K\Rightarrow B_K(\boldsymbol{a})\cap\partial\Omega\neq\emptyset$ holds for $K\geq L$. Define
    \begin{align*}
        C(\lambda)=\frac{|\Omega|p^{n[\frac{1}{\alpha}\log_p\lambda]}-N_{D_{\Omega,\mathcal{D}}^{\alpha}}(\lambda)}{p^{d[\frac{1}{\alpha}\log_p\lambda]}}.
    \end{align*}
    Then
    \begin{align*}
        -\overline{\mathcal{M}}_{\text{ex}}^{d}(\partial\Omega)\leq\varliminf_{\lambda\to+\infty}C(\lambda)
        \leq\varlimsup_{\lambda\to+\infty}C(\lambda)\leq\overline{\mathcal{M}}_{\text{in}}^{d}(\partial\Omega).
    \end{align*}
    The upper and lower bounds for $C(\lambda)$ are sharp for a family of domains $\Omega$.
\end{theorem}
\begin{remark}
    Since the upper and lower bounds for $C(\lambda)$ are sharp for a family of domains $\Omega$ in every dimension $n$, the classical Weyl--Berry conjecture fails in the $p$-adic setting. This is another difference from the Archimedean setting.
\end{remark}

We next consider the $p$-adic version of Pólya's conjecture. Define orthogonal projection $P_K$ by
\begin{align*}
    \widehat{P_Ku}(\boldsymbol{\xi})=\widehat{u}(\boldsymbol{\xi})\boldsymbol{1}_{p^{-K}\mathbb{Z}_p^n}(\boldsymbol{\xi}).
\end{align*}
Let $\Delta_j=P_j-P_{j-1}$. Write $\Omega=\displaystyle\bigsqcup_{i=1}^{\infty}B_{L_i}(\boldsymbol{a}_i)$ as a disjoint union of maximal balls $\{B_{L_i}(\boldsymbol{a}_i)\}_{i=1}^{\infty}$, and denote
\begin{align*}
    &V=\overline{\operatorname{span}\{\boldsymbol{1}_{B_{L_i}(\boldsymbol{a}_i)}\}_{i=1}^{\infty}}^{L^2(\Omega)};\\
    &V_{=i}=\operatorname{span}\{\boldsymbol{1}_{B_{L_j}(\boldsymbol{a}_j)}|\ L_j=i\};\\
    &n_i=\dim V_{=i};\\
    &V_{\leq K}=\bigoplus_{i\leq K} V_{=i};\\
    &V_{>K}=\overline{\bigoplus_{i\geq K+1} V_{=i}}^{L^2(\Omega)}=V\ominus V_{\leq K};\\
    &B_{\leq K}=\{B_{L_i}(\boldsymbol{a}_i)|\ \boldsymbol{1}_{B_{L_i}(\boldsymbol{a}_i)}\in V_{\leq K}\};\\
    &B_{>K}=\{B_{L_i}(\boldsymbol{a}_i)|\ \boldsymbol{1}_{B_{L_i}(\boldsymbol{a}_i)}\in V_{>K}\};\\
    &\Omega_{\leq K}=\bigcup_{L_i\leq K}B_{L_i}(\boldsymbol{a}_i);\\
    &\Omega_{>K}=\bigcup_{L_i\geq K+1}B_{L_i}(\boldsymbol{a}_i).
\end{align*}
For each $f\in P_KV$, define the quadratic form
\begin{align*}
    \mathcal{W}_K[f]=\sum_{j\leq K}(\lambda_j-\lambda_{K+1})\|\Delta_jf\|_{L^2(\mathbb{Q}_p^n)}^2,
\end{align*}
where $\lambda_j=p^{\alpha j}$. With respect to the normalized basis, the matrix representing the quadratic form $\mathcal{W}_K$ can be written as
\begin{align*}
    \boldsymbol{W}_K=\begin{pmatrix}
        \boldsymbol{W}_{00} & \boldsymbol{W}_{01} \\
        \boldsymbol{W}_{10} & \boldsymbol{W}_{11}
    \end{pmatrix}.
\end{align*}
If $\sum\limits_{i\leq K}n_i=0$, we set $\boldsymbol{W}_{00},\boldsymbol{W}_{01},\boldsymbol{W}_{10}=\emptyset$. Let
\begin{align*}
    \boldsymbol{\kappa}_K
    =\operatorname{diag}(\kappa_{B_K}\mid B_K\in\mathcal{B}_K).
\end{align*}
Here $\kappa_{B_K}=\lim\limits_{r\to+\infty}\Lambda_{K,r}(B_K)$,
where $\Lambda_{K,r}(B_K)$ is obtained from the following finite inner
approximation. Starting from $B_K$, split its sub-ball $B_j$ into the $p^n$ sub-balls $B_{j+1}\subseteq B_j$, stopping at level $K+r$. Set
\begin{equation*}
\Lambda_{K,r}(B_j)=
\begin{cases}
0,& B_j\subseteq\Omega,\ K<j\leq K+r;\\
+\infty,& B_j\cap\Omega=\emptyset,\ K\leq j\leq K+r\text{ or }B_{j}\in\mathcal{B}_{j},\ j=K+r;\\
\frac{p^n}{\sum\limits_{B_{j+1}\subseteq B_j}\frac{1}{w_j+\Lambda_{K,r}(B_{j+1})}}-w_j,& B_j\in\mathcal{B}_j,\ K\leq j< K+r,
\end{cases}
\end{equation*}
where $w_j=\lambda_{j+1}-\lambda_{K+1}$. The reciprocal conventions at
$0$ and $+\infty$ are specified in Section~6. The sequence
$\Lambda_{K,r}(B_K)$ is nonincreasing and converges to the strict local
coefficient $\kappa_{B_K}$. Let
\begin{equation*}
    \boldsymbol{R}_K=
    \begin{cases}
        \boldsymbol{W}_{11}+\boldsymbol{\kappa}_K,
        &\sum\limits_{i\leq K}n_i=0,\\[1mm]
        \boldsymbol{W}_{11}+\boldsymbol{\kappa}_K-\boldsymbol{W}_{10}\boldsymbol{W}_{00}^{-1}\boldsymbol{W}_{01},
        &\sum\limits_{i\leq K}n_i>0.
    \end{cases}
\end{equation*}
Our third main result is the following characterization:

\begin{theorem}[Pólya's conjecture for the Dirichlet operator]
    Let $\Omega\subseteq\mathbb{Q}_p^n$ be a bounded, open, Jordan measurable set. Then Pólya's conjecture holds if and only if
    \begin{align*}
        n_{-}(\boldsymbol{R}_K)\leq\left[|\Omega_{>K}|p^{nK}\right],\ \forall K\in\mathbb{Z}.
    \end{align*}
    In particular, Pólya's conjecture fails for compact open domains.
    Suppose more generally that there exists $K_0\in\mathbb Z$ such that,
    for every $K\geq K_0$ and every $B_K\in\mathcal{B}_K$, there is a
    maximal ball $B_{K+1}\subseteq B_K\cap\Omega$. Then Pólya's conjecture
    holds if and only if
    \begin{align*}
    &\#\mathcal{B}_K=|\Omega_{>K}|p^{nK},&&K\geq K_0,\\
    &n_-(\boldsymbol R_K)\leq
    \left[|\Omega_{>K}|p^{nK}\right],&&K<K_0.
    \end{align*}
    The strictly positive Dirichlet spectral bottom makes the second
    condition automatic for all sufficiently negative $K$. Hence only
    finitely many coarse scales require a separate verification.
\end{theorem}

\section{Preliminaries}
\subsection{\texorpdfstring{The $p$-adic field}{The p-adic field}}
It is well known that the completion of $\mathbb{Q}$ under the Archimedean absolute value $|\cdot|_{\infty}$ is $\mathbb{R}$. In 1916, Ostrowski showed that every non-trivial absolute value $|\cdot|$ on $\mathbb{Q}$ is equivalent to the Archimedean absolute value $|\cdot|_{\infty}$ or the non-Archimedean absolute value $|\cdot|_{p}$ for some prime $p$ defined by
\begin{align*}
    \begin{cases}
        |x|_{\infty}=\max\{x,-x\};\\
        |x|_{p}=p^{-v},
    \end{cases}
\end{align*}
where $p$ is a prime number and $x=p^{v}\frac{a}{b}$ with $p\nmid ab$ \cite{MR1555153}. Hence we have the normalized condition
\begin{align*}
    \prod_{\substack{p\text{ prime}\\p\leq\infty}}|x|_p=1.
\end{align*}
The completion of $\mathbb{Q}$ under $|\cdot|_{p}$ is the $p$-adic field $\mathbb{Q}_p$.

We define $v_p(x)=-\log_p|x|_{p}$ and $v_p(0)=+\infty$, so that $|x|_{p}=p^{-v_p(x)}$. For each $x\in\mathbb{Q}_p$, we have the $p$-adic expansion
\begin{align*}
    x=\sum_{k=v_p(x)}^{+\infty}a_kp^k,
\end{align*}
where $a_{v_p(x)}\in\mathbb{F}_p^{\times}$ and $a_{k}\in\mathbb{F}_p$ for $k>v_p(x)$.

The ring of $p$-adic integers $\mathbb{Z}_p$ is defined by
\begin{align*}
    \mathbb{Z}_p:=\{x\in\mathbb{Q}_p|\ v_p(x)\geq0\}=\{x\in\mathbb{Q}_p|\ |x|_p\leq1\},
\end{align*}
which can be viewed as the unit ball in $p$-adic field. The group of units of $\mathbb{Z}_p$ is defined by
\begin{align*}
    \mathbb{Z}_p^*:=\{x\in\mathbb{Q}_p|\ |x|_p=1\}.
\end{align*}

We equip $\mathbb{Q}_p^n$ with the norm
\begin{align*}
    \|\boldsymbol{x}\|_p=\max_{1\leq i\leq n}\{|x_i|_p\},\ \boldsymbol{x}=(x_1,\dots,x_n)\in\mathbb{Q}_p^n,
\end{align*}
so that $v_p(\boldsymbol{x})=\min\limits_{1\leq i\leq n}\{v_p(x_i)\}$. It is straightforward to verify that $\mathbb{Q}_p^n$ is complete under the norm $\|\cdot\|_p$.

\begin{definition}
    The additive character $\chi_{\xi}^p$ on $\mathbb{Q}_p$ is a continuous group homomorphism from $\mathbb{Q}_p$ to $S^1$ defined by
    \begin{align*}
        \chi_{\xi}^p(x)=e^{-2\pi i\{x\xi\}_p},
    \end{align*}
    where $x,\xi\in\mathbb{Q}_p$ since the Pontryagin dual of $\mathbb{Q}_p$ is naturally isomorphic to itself. Here $\{\cdot\}_p$ denotes the $p$-adic fractional part defined by
    \begin{align*}
        \{x\}_p=\begin{cases}
            0,&v_p(x)\geq0;\\
            \displaystyle\sum_{k=v_p(x)}^{-1}a_kp^k,&v_p(x)<0.
        \end{cases}
    \end{align*}
\end{definition}
\begin{remark}
Note that the additive character on $\mathbb{R}$ is $\chi_{\xi}^{\infty}(x)=e^{2\pi ix\xi}$, so for the additive character, we also have the normalized condition
    \begin{align*}
    \prod_{\substack{p\text{ prime}\\p\leq\infty}}\chi_{\xi}^{p}(x)=1.
\end{align*}
For convenience, we define $\chi_p(x):=\overline{\chi_{1}^{p}}(x)=e^{2\pi i\{x\}_p}$.
\end{remark}

\subsection{Basic function spaces and Fourier transform}
We say a function $f:\mathbb{Q}_p^n\to\mathbb{C}$ is locally constant if, for every $\boldsymbol{x}\in\mathbb{Q}_p^n$, there exists a neighborhood $\boldsymbol{x}+p^l\mathbb{Z}_p^n$ for some integer $l$ depending on $\boldsymbol{x}$ such that $f(\boldsymbol{y})=f(\boldsymbol{z})$ for all $\boldsymbol{y},\boldsymbol{z}\in\boldsymbol{x}+p^l\mathbb{Z}_p^n$. Therefore, we introduce the Bruhat-Schwartz space:
\begin{align*}
     \mathcal{S}(\mathbb{Q}_p^n):=\{f|\ f\text{ is locally constant with compact support}\},
\end{align*}
which can be seen as the non-Archimedean analogue of $C_c^{\infty}(\mathbb{R}^n)$ in the Archimedean setting. Therefore, for each $f\in\mathcal{S}(\mathbb{Q}_p^n)$, we can define the Fourier transform as
\begin{align*}
    \hat{f}(\boldsymbol{\xi}):=\int_{\mathbb{Q}_p^n}f(\boldsymbol{x})\overline{\chi_{\boldsymbol{\xi}}(\boldsymbol{x})}d\boldsymbol{x}=\int_{\mathbb{Q}_p^n}f(\boldsymbol{x})\chi_{p}(\boldsymbol{\xi}\cdot\boldsymbol{x})d\boldsymbol{x},
\end{align*}
where $d\boldsymbol{x}=dx_1\cdots dx_n$ and $dx_i$ is the Haar measure on $\mathbb{Q}_p$ normalized by $|\mathbb{Z}_p|=1$. One can check that the Fourier transform is an automorphism of $\mathcal{S}(\mathbb{Q}_p^n)$.

For later use, we introduce the following spaces:
\begin{definition}
    The Lizorkin space of the second kind in $p$-adic field is defined by
    \begin{align*}
     \mathcal{S}_{2}(\mathbb{Q}_p^n):=\{f\in\mathcal{S}(\mathbb{Q}_p^n)|\ \hat{f}(\boldsymbol{0})=0\}=\{f\in\mathcal{S}(\mathbb{Q}_p^n)|\ \int_{\mathbb{Q}_p^n}f(\boldsymbol{x})d\boldsymbol{x}=0\},
\end{align*}
which is dense in $L^2(\mathbb{Q}_p^n)$.
\end{definition}

\begin{remark}
In Archimedean field, the Lizorkin space of the second kind is defined by
\begin{align*}
    \mathcal{S}_{2}(\mathbb{R}^n):=\{f\in\mathcal{S}(\mathbb{R}^n)|\
    \partial^{\boldsymbol{\alpha}}\hat{f}(\boldsymbol{0})=0,\ \forall\boldsymbol{\alpha}\in\mathbb{N}^n\}=\left\{f\in\mathcal{S}(\mathbb{R}^n)|\int_{\mathbb{R}^n}\boldsymbol{x}^{\boldsymbol{\alpha}}f(\boldsymbol{x})d\boldsymbol{x}=0,\ \forall\boldsymbol{\alpha}\in\mathbb{N}^n\right\}.
\end{align*}
Unlike the Archimedean setting, we do not need to impose these additional moment conditions, because $f$ is locally constant.
\end{remark}

The fractional Sobolev space is defined by
\begin{align*}
    H^s(\mathbb{Q}_p^n):=\{f|\ \|f\|_{H^s(\mathbb{Q}_p^n)}^2=\int_{\mathbb{Q}_p^n}\max\{1,\|\boldsymbol{\xi}\|_p\}^{2s}|\hat{f}(\boldsymbol{\xi})|^2d\boldsymbol{\xi}<+\infty\}.
\end{align*}
This norm is analogous to its Archimedean counterpart.

For a domain $\Omega\subseteq\mathbb{Q}_p^n$, we define $\mathcal{S}(\Omega)$ and $\mathcal{S}_{2}(\Omega)$ by
\begin{align*}
     \mathcal{S}(\Omega):=\{f\in\mathcal{S}(\mathbb{Q}_p^n)|\ \operatorname{supp}f\subseteq\Omega\},
\end{align*}
and
\begin{align*}
    \mathcal{S}_{2}(\Omega):=\{f\in\mathcal{S}_{2}(\mathbb{Q}_p^n)|\ \operatorname{supp}f\subseteq\Omega\}.
\end{align*}
We also set
\begin{align*}
    \mathcal{S}(\mathbb{Q}_p^n)|_\Omega:=\{f|_\Omega|\ f\in\mathcal{S}(\mathbb{Q}_p^n)\}.
\end{align*}
Define the mean-zero subspace of $L^p(\Omega)$ by
\begin{align*}
    L_0^p(\Omega):=\{f\in L^p(\Omega)|\ \int_{\Omega}f(\boldsymbol{x})d\boldsymbol{x}=0\}.
\end{align*}
For a domain $\Omega\subseteq\mathbb{Q}_p^n$ and $s\geq0$, define
\begin{align*}
    H_0^s(\Omega):=\overline{\mathcal{S}(\Omega)}^{H^s(\mathbb{Q}_p^n)}.
\end{align*}
We also retain the zero-extension space used for comparison:
\begin{align*}
    \widetilde{H}_0^s(\Omega)
    :=\{u\in H^s(\mathbb Q_p^n):u=0\text{ a.e. on }\Omega^c\}.
\end{align*}
We identify functions in these full-space zero-extension spaces with their
restrictions to $\Omega$ when no confusion can arise. For $s>0$, the intrinsic
fractional Sobolev space is
\begin{align*}
    H^s(\Omega):=\left\{f\in L^2(\Omega):
    \iint_{\Omega\times\Omega}
    \frac{|f(\boldsymbol{x})-f(\boldsymbol{y})|^2}
    {\|\boldsymbol{x}-\boldsymbol{y}\|_p^{n+2s}}
    d\boldsymbol{x}d\boldsymbol{y}<+\infty\right\},
\end{align*}
and $H^0(\Omega)=L^2(\Omega)$.

\begin{prop}
For every open set $\Omega\subseteq\mathbb{Q}_p^n$ and $s\geq0$, every
element of $H_0^s(\Omega)$ vanishes almost everywhere on $\Omega^c$, and its
restriction belongs to $H^s(\Omega)$. Equivalently,
\begin{align*}
    H_0^s(\Omega)\subseteq\widetilde H_0^s(\Omega).
\end{align*}
If $\Omega$ is compact and open, then
\begin{align*}
    H_0^s(\Omega)=\widetilde H_0^s(\Omega),
\end{align*}
and zero extension identifies these spaces with $H^s(\Omega)$, with
equivalent norms.
\end{prop}
\begin{proof}
The vanishing assertion follows by taking the $L^2$ limit of functions in
$\mathcal S(\Omega)$, and membership of the restriction in $H^s(\Omega)$
follows by restricting the full-space Gagliardo integral to
$\Omega\times\Omega$. Suppose that $\Omega$ is compact
and open. It is a finite union of balls of a common level $L$. For $K\geq L$,
convolution with $p^{nK}\boldsymbol 1_{p^K\mathbb Z_p^n}$ preserves the
zero-extension condition, produces a Bruhat--Schwartz function supported in
$\Omega$, and is the Fourier multiplier
    $\boldsymbol 1_{\{\|\boldsymbol\xi\|_p\leq p^K\}}$. Dominated convergence
    therefore gives convergence in $H^s(\mathbb Q_p^n)$ and proves that every
    $H^s(\mathbb Q_p^n)$ function vanishing almost everywhere on $\Omega^c$
    belongs to $H_0^s(\Omega)$.

Finally, the distance between a compact open set and its complement is
positive. If $f\in H^s(\Omega)$, the full-space Gagliardo seminorm of its zero
extension equals the intrinsic seminorm plus
\begin{align*}
2\int_\Omega |f(\boldsymbol{x})|^2
\left(\int_{\Omega^c}
\frac{d\boldsymbol{y}}
{\|\boldsymbol{x}-\boldsymbol{y}\|_p^{n+2s}}\right)d\boldsymbol{x}.
\end{align*}
The inner integral is uniformly bounded on $\Omega$. The equivalence between
the Fourier and Gagliardo norms proves the remaining assertion. The case
$s=0$ is immediate.
\end{proof}

\begin{remark}
For a general open set, the inclusion
$H_0^s(\Omega)\subseteq\widetilde H_0^s(\Omega)$ may be strict. Therefore, the
strict closure of $\mathcal S(\Omega)$ and the almost-everywhere
zero-extension condition must not be identified without an additional
regularity argument. Throughout this paper, the Dirichlet form domain is the
strict space $H_0^{\alpha/2}(\Omega)$.
\end{remark}
The topology of $\mathbb{Q}_p^n$ and Hölder's inequality yield the following density results:
\begin{prop}
    If $|\Omega|<\infty$, $\mathcal{S}_{2}(\Omega)$ is dense in $L_0^2(\Omega)$, which is a closed subspace of $L^2(\Omega)$; and $\mathcal{S}_{2}(\Omega)$ is dense in $H_0^s(\Omega)\cap L_0^2(\Omega)$. If $|\Omega|=\infty$, $\mathcal{S}_{2}(\Omega)$ is dense in $L^2(\Omega)$; and $\overline{\mathcal{S}_{2}(\Omega)}^{H^s(\mathbb{Q}_p^n)}\subseteq H_0^s(\Omega)$.\\
    The space $\mathcal{S}(\mathbb{Q}_p^n)|_\Omega$ is dense in
    $L^2(\Omega)$ regardless of $|\Omega|$.
\end{prop}
\begin{proof}
Assume first that $|\Omega|<+\infty$. The mean-value functional is
continuous on $L^2(\Omega)$, so $L_0^2(\Omega)$ is closed. Choose
$\varphi\in\mathcal S(\Omega)$ with $\int_\Omega\varphi=1$. If
$f\in H_0^s(\Omega)\cap L_0^2(\Omega)$ and
$f_j\in\mathcal S(\Omega)$ converges to $f$ in $H^s$, then
\begin{align*}
f_j-\left(\int_\Omega f_j\right)\varphi\in\mathcal S_2(\Omega)
\end{align*}
and this sequence still converges to $f$. The same correction, starting
from an $L^2$ approximation, proves density in $L_0^2(\Omega)$.

If $|\Omega|=+\infty$, approximate $f\in L^2(\Omega)$ by
$g\in\mathcal S(\Omega)$. The open set
$\Omega\setminus\operatorname{supp}g$ still has infinite measure. By a
compact-open exhaustion, it contains finite compact open unions $A$,
disjoint from $\operatorname{supp}g$, of arbitrarily large measure. For such
an $A$,
\begin{align*}
g-\left(\int_\Omega g\right)|A|^{-1}\boldsymbol 1_A
\in\mathcal S_2(\Omega),
\end{align*}
and the $L^2$ norm of the correction tends to zero as $|A|\to\infty$.
The remaining inclusion follows from
$\mathcal S_2(\Omega)\subseteq\mathcal S(\Omega)$. Finally, restrictions of
Bruhat--Schwartz functions are dense in $L^2(\Omega)$ by compact exhaustion
and local averaging.
\end{proof}

\subsection{The Vladimirov-Taibleson operator}
\begin{definition}
    The Vladimirov-Taibleson operator $D^\alpha:\mathcal{S}_{2}(\mathbb{Q}_p^n)\to\mathcal{S}_{2}(\mathbb{Q}_p^n)$ is a pseudodifferential operator defined by
    \begin{align*}
        \widehat{D^\alpha f}(\boldsymbol{\xi})=\|\boldsymbol{\xi}\|_p^{\alpha}\hat{f}(\boldsymbol{\xi}),
    \end{align*}
    where $\alpha>0$, which can be extended to an operator from $H^{\alpha}(\mathbb{Q}_p^n)$ to $L^2(\mathbb{Q}_p^n)$. Its associated closed quadratic form is
    \begin{align*}
        \mathcal{E}(u,v)=\int_{\mathbb{Q}_p^n}\|\boldsymbol{\xi}\|_p^{\alpha}\hat{u}(\boldsymbol{\xi})\overline{\hat{v}}(\boldsymbol{\xi})d\boldsymbol{\xi}
    \end{align*}
    with $\mathcal{D}(\mathcal{E})=H^{\frac{\alpha}{2}}(\mathbb{Q}_p^n)$.
\end{definition}

\begin{remark}
    The Vladimirov-Taibleson operator can be viewed as the non-Archimedean counterpart of the fractional Laplacian $(-\Delta)^{\frac{\alpha}{2}}$ on $\mathbb{R}^n$, whose Fourier symbol is $\|\boldsymbol{\xi}\|^\alpha$.
\end{remark}

\begin{remark}
    By Fourier inversion, there is an equivalent definition of the Vladimirov-Taibleson operator by singular integral:
    \begin{align*}
        D^\alpha f(\boldsymbol{x})=\frac{1-p^{\alpha}}{1-p^{-\alpha-n}}\int_{\mathbb{Q}_p^n}\frac{f(\boldsymbol{y})-f(\boldsymbol{x})}{\|\boldsymbol{y}-\boldsymbol{x}\|_p^{\alpha+n}}d\boldsymbol{y}.
    \end{align*}
    Then its associated closed quadratic form can also be written as
    \begin{align*}
        \mathcal{E}(u,v)=\frac{1}{2}\frac{p^{\alpha}-1}{1-p^{-\alpha-n}}\int_{\mathbb{Q}_p^n\times\mathbb{Q}_p^n}\frac{(u(\boldsymbol{x})-u(\boldsymbol{y}))\overline{(v(\boldsymbol{x})-v(\boldsymbol{y}))}}{\|\boldsymbol{x}-\boldsymbol{y}\|_p^{\alpha+n}}d\boldsymbol{x}d\boldsymbol{y}.
    \end{align*}
\end{remark}

To investigate Weyl's law and Pólya's conjecture, the first step is to define the Dirichlet and Neumann Vladimirov-Taibleson operators on $\Omega\subseteq\mathbb{Q}_p^n$.

\begin{definition}
\label{def:Dirichlet_Vladimirov-Taibleson_operator}
    For $f\in\mathcal S(\Omega)$, let $\tilde f$ be its zero extension to
    $\mathbb Q_p^n$ and set
    \begin{align*}
        D_{\Omega,\mathcal{D}}^{\alpha}f(\boldsymbol{x})=(D^{\alpha}\tilde{f})|_{\Omega}(\boldsymbol{x}),
    \end{align*}
    for $\boldsymbol{x}\in\Omega$. Equivalently,
    \begin{align*}
        D_{\Omega,\mathcal{D}}^{\alpha}f(\boldsymbol{x})&=\frac{1-p^{\alpha}}{1-p^{-\alpha-n}}\int_{\mathbb{Q}_p^n}\frac{\tilde{f}(\boldsymbol{y})-\tilde{f}(\boldsymbol{x})}{\|\boldsymbol{y}-\boldsymbol{x}\|_p^{\alpha+n}}d\boldsymbol{y}\\
        &=\frac{1-p^{\alpha}}{1-p^{-\alpha-n}}\int_{\Omega}\frac{f(\boldsymbol{y})-f(\boldsymbol{x})}{\|\boldsymbol{y}-\boldsymbol{x}\|_p^{\alpha+n}}d\boldsymbol{y}-\frac{1-p^{\alpha}}{1-p^{-\alpha-n}}\left(\int_{\Omega^c}\frac{1}{\|\boldsymbol{y}-\boldsymbol{x}\|_p^{\alpha+n}}d\boldsymbol{y}\right)f(\boldsymbol{x}).
    \end{align*}
    The Dirichlet Vladimirov--Taibleson operator in $L^2(\Omega)$ is the
    unique nonnegative self-adjoint operator associated, by the first
    representation theorem, with the closed form
    \begin{align*}
        \mathcal{E}_{\Omega,\mathcal{D}}(u,v)=\mathcal{E}(\tilde{u},\tilde{v}),
        \qquad
        \mathcal{D}(\mathcal{E}_{\Omega,\mathcal{D}})=H_0^{\frac{\alpha}{2}}(\Omega).
    \end{align*}
    Thus $u\in\mathcal D(D_{\Omega,\mathcal D}^{\alpha})$ precisely when
    $u\in H_0^{\alpha/2}(\Omega)$ and there exists $f\in L^2(\Omega)$ such
    that
    \begin{align*}
        \mathcal E_{\Omega,\mathcal D}(u,v)
        =(f,v)_{L^2(\Omega)},
        \qquad v\in H_0^{\alpha/2}(\Omega),
    \end{align*}
    in which case $D_{\Omega,\mathcal D}^{\alpha}u=f$. This operator is the
    Friedrichs extension of the zero-extension operator initially defined on
    $\mathcal S(\Omega)$. In particular,
    \begin{align*}
        H_0^{\alpha}(\Omega)
        \subseteq\mathcal D(D_{\Omega,\mathcal D}^{\alpha}),
    \end{align*}
    and the inclusion need not be an equality for a general open set.
\end{definition}

\begin{prop}
For every $u\in H_0^\alpha(\Omega)$,
\begin{align*}
D_{\Omega,\mathcal D}^{\alpha}u
=(D^\alpha\tilde u)|_\Omega.
\end{align*}
\end{prop}
\begin{proof}
If $u\in H_0^\alpha(\Omega)$, then
$D^\alpha\tilde u\in L^2(\mathbb Q_p^n)$. For every
$v\in H_0^{\alpha/2}(\Omega)$, Plancherel's theorem gives
\begin{align*}
\mathcal E_{\Omega,\mathcal D}(u,v)
=(D^\alpha\tilde u,\tilde v)_{L^2(\mathbb Q_p^n)}
=((D^\alpha\tilde u)|_\Omega,v)_{L^2(\Omega)}.
\end{align*}
The weak characterization in Definition
\ref{def:Dirichlet_Vladimirov-Taibleson_operator} proves the assertion.
\end{proof}

\begin{definition}
\label{def:Neumann_Vladimirov-Taibleson_operator}
    The Neumann form is defined by
    \begin{align*}
        \mathcal{E}_{\Omega,\mathcal{N}}(u,v)=\frac{1}{2}\frac{p^{\alpha}-1}{1-p^{-\alpha-n}}\int_{\Omega\times\Omega}\frac{(u(\boldsymbol{x})-u(\boldsymbol{y}))\overline{(v(\boldsymbol{x})-v(\boldsymbol{y}))}}{\|\boldsymbol{x}-\boldsymbol{y}\|_p^{\alpha+n}}d\boldsymbol{x}d\boldsymbol{y}
    \end{align*}
    with $\mathcal{D}(\mathcal{E}_{\Omega,\mathcal{N}})=H^{\frac{\alpha}{2}}(\Omega)$.
    The Neumann Vladimirov--Taibleson operator
    $D_{\Omega,\mathcal N}^{\alpha}$ is the unique nonnegative
    self-adjoint operator associated with this form. For test functions, its
    integral expression is
    \begin{align*}
        D_{\Omega,\mathcal{N}}^{\alpha}f(\boldsymbol{x})
        =\frac{1-p^{\alpha}}{1-p^{-\alpha-n}}\int_{\Omega}
        \frac{f(\boldsymbol{y})-f(\boldsymbol{x})}
        {\|\boldsymbol{y}-\boldsymbol{x}\|_p^{\alpha+n}}d\boldsymbol{y}.
    \end{align*}
\end{definition}

\begin{prop}
Both forms are densely defined, closed, nonnegative, and symmetric. For every
$u\in H_0^{\alpha/2}(\Omega)$,
\begin{align}
\label{eq:Dirichlet_Neumann_form_identity}
\mathcal E_{\Omega,\mathcal D}[u]
=\mathcal E_{\Omega,\mathcal N}[u]
+\frac{p^\alpha-1}{1-p^{-\alpha-n}}
\int_\Omega |u(\boldsymbol{x})|^2
\left(\int_{\Omega^c}
\frac{d\boldsymbol{y}}
{\|\boldsymbol{x}-\boldsymbol{y}\|_p^{n+\alpha}}\right)d\boldsymbol{x}.
\end{align}
In particular,
$H_0^{\alpha/2}(\Omega)\subseteq H^{\alpha/2}(\Omega)$ and
$\mathcal E_{\Omega,\mathcal N}[u]\leq
\mathcal E_{\Omega,\mathcal D}[u]$ on the Dirichlet form domain.
\end{prop}
\begin{proof}
Closedness of the Dirichlet form follows directly from its definition as the
closure of $\mathcal S(\Omega)$ in the full-space form norm. Closedness of the
Neumann form follows from completeness of the intrinsic Gagliardo norm. For
$u\in\mathcal S(\Omega)$, split the full-space double integral into
$\Omega\times\Omega$, the two cross terms, and
$\Omega^c\times\Omega^c$. This gives
\eqref{eq:Dirichlet_Neumann_form_identity}. If
$u_j\in\mathcal S(\Omega)$ converges to
$u\in H_0^{\alpha/2}(\Omega)$ in the Dirichlet form norm, then
$\mathcal E_{\Omega,\mathcal N}[u_j-u_k]\leq
\mathcal E_{\Omega,\mathcal D}[u_j-u_k]$. Applying the identity to
$u_j-u_k$ also shows that the boundary terms are Cauchy. Passing to the
limit proves the formula on the full Dirichlet form domain.
\end{proof}

\begin{remark}
The additional positive term in
\eqref{eq:Dirichlet_Neumann_form_identity} is the interaction between
$\Omega$ and $\Omega^c$ created by strict zero extension. It records the
nonlocal Dirichlet boundary condition and is absent from the Neumann form,
which contains only interactions inside $\Omega$.
\end{remark}

\begin{prop}
Let $\Omega$ be bounded and choose a ball $B_L(\boldsymbol a)$ containing
$\Omega$. Then
\begin{align}
\label{eq:Dirichlet_Poincare}
\mathcal E_{\Omega,\mathcal D}[u]
\geq p^{\alpha L}\frac{1-p^{-n}}{1-p^{-\alpha-n}}
\|u\|_{L^2(\Omega)}^2,
\qquad u\in H_0^{\alpha/2}(\Omega).
\end{align}
Moreover, the embedding
$H_0^{\alpha/2}(\Omega)\hookrightarrow L^2(\Omega)$ is compact.
Consequently, $D_{\Omega,\mathcal D}^{\alpha}$ has compact resolvent and a
strictly positive first eigenvalue.
\end{prop}
\begin{proof}
For $\boldsymbol{x}\in\Omega\subseteq B_L(\boldsymbol a)$, retain in the
full-space energy only the interaction with
$B_L(\boldsymbol a)^c\subseteq\Omega^c$. The $p$-adic spherical-shell
decomposition gives
\begin{align*}
\int_{B_L(\boldsymbol a)^c}
\frac{d\boldsymbol y}{\|\boldsymbol{x}-\boldsymbol y\|_p^{n+\alpha}}
=(1-p^{-n})\frac{p^{\alpha(L-1)}}{1-p^{-\alpha}},
\end{align*}
and \eqref{eq:Dirichlet_Poincare} follows.

For compactness, regard $u$ as its zero extension and let the Fourier
multiplier $P_K$ be given by
\begin{align*}
\widehat{P_Ku}(\boldsymbol\xi)
=\boldsymbol 1_{\{\|\boldsymbol\xi\|_p\leq p^K\}}
\widehat u(\boldsymbol\xi).
\end{align*}
Since the values of the $p$-adic norm are discrete,
\begin{align*}
\|(I-P_K)u\|_2^2
\leq p^{-\alpha(K+1)}\mathcal E_{\Omega,\mathcal D}[u].
\end{align*}
For all sufficiently large $K$, the restrictions to $\Omega$ of the
functions $P_Ku$ lie in a fixed finite-dimensional space of functions that
are constant on balls of level $K$ inside $B_L(\boldsymbol a)$. These
finite-rank maps approximate the embedding uniformly on the form unit ball,
which proves compactness.
\end{proof}

\begin{prop}[Domain monotonicity]
If $\Omega_1\subseteq\Omega_2$ are bounded open sets, then
\begin{align*}
\lambda_j(D_{\Omega_2,\mathcal D}^{\alpha})
\leq\lambda_j(D_{\Omega_1,\mathcal D}^{\alpha}),
\qquad
N_{D_{\Omega_1,\mathcal D}^{\alpha}}(\lambda)
\leq N_{D_{\Omega_2,\mathcal D}^{\alpha}}(\lambda).
\end{align*}
\end{prop}
\begin{proof}
The inclusion $\mathcal S(\Omega_1)\subseteq\mathcal S(\Omega_2)$ gives
$H_0^{\alpha/2}(\Omega_1)\subseteq H_0^{\alpha/2}(\Omega_2)$ after taking
closures. The two forms agree on the smaller domain, so the min--max
principle proves the assertion.
\end{proof}

There is an alternative way of regarding the Vladimirov-Taibleson operator on $\Omega$. When $\Omega$ has a group structure, it carries a natural Fourier transform, and the operator can be defined directly through its Fourier symbol. For instance, take $\Omega=\mathbb{Z}_p^n$, an additive group with additive character $\chi_{\boldsymbol{\xi}}^p(\boldsymbol{x})$. Its Pontryagin dual is $(\mathbb{Q}_p/\mathbb{Z}_p)^n$. We define
\begin{align*}
    \widehat{D_{\mathbb{Z}_p^n}^{\alpha}f}(\boldsymbol{\xi}):=\|\boldsymbol{\xi}\|_p^{\alpha}\hat{f}(\boldsymbol{\xi}),
\end{align*}
where $\boldsymbol{x}\in\mathbb{Z}_p^n$ and $\boldsymbol{\xi}\in(\mathbb{Q}_p/\mathbb{Z}_p)^n$. The norm in $(\mathbb{Q}_p/\mathbb{Z}_p)^n$ is defined by
\begin{align*}
    \|\boldsymbol{\xi}\|_p:=\inf_{\boldsymbol{a}\in\boldsymbol{\xi}+\mathbb{Z}_p^n}\|\boldsymbol{a}\|_p.
\end{align*}
Accordingly, we may choose $\{(\{x_1\}_p,\dots,\{x_n\}_p)|\ \boldsymbol{x}=(x_1,\dots,x_n)\in\mathbb{Q}_p^n\}$ as the fundamental domain. Fourier inversion then gives
\begin{align*}
    D_{\mathbb{Z}_p^n}^{\alpha}f(\boldsymbol{x})=\frac{1-p^{\alpha}}{1-p^{-\alpha-n}}\int_{\mathbb{Z}_p^n}\frac{f(\boldsymbol{y})-f(\boldsymbol{x})}{\|\boldsymbol{y}-\boldsymbol{x}\|_p^{\alpha+n}}d\boldsymbol{y}+\frac{1-p^{-n}}{1-p^{-\alpha-n}}f(\boldsymbol{x})-\frac{1-p^{-n}}{1-p^{-\alpha-n}}\int_{\mathbb{Z}_p^n}f(\boldsymbol{y})d\boldsymbol{y}.
\end{align*}
A direct calculation yields
\begin{align*}
        D_{\mathbb{Z}_p^n,\mathcal{D}}^{\alpha}f(\boldsymbol{x})
        &=\frac{1-p^{\alpha}}{1-p^{-\alpha-n}}\int_{\mathbb{Z}_p^n}\frac{f(\boldsymbol{y})-f(\boldsymbol{x})}{\|\boldsymbol{y}-\boldsymbol{x}\|_p^{\alpha+n}}d\boldsymbol{y}-\frac{1-p^{\alpha}}{1-p^{-\alpha-n}}\left(\int_{\mathbb{Q}_p^n\backslash\mathbb{Z}_p^n}\frac{1}{\|\boldsymbol{y}\|_p^{\alpha+n}}d\boldsymbol{y}\right)f(\boldsymbol{x})\\
        &=\frac{1-p^{\alpha}}{1-p^{-\alpha-n}}\int_{\mathbb{Z}_p^n}\frac{f(\boldsymbol{y})-f(\boldsymbol{x})}{\|\boldsymbol{y}-\boldsymbol{x}\|_p^{\alpha+n}}d\boldsymbol{y}+\frac{1-p^{-n}}{1-p^{-\alpha-n}}f(\boldsymbol{x}),
    \end{align*}
and
\begin{align*}
        D_{\mathbb{Z}_p^n,\mathcal{N}}^{\alpha}f(\boldsymbol{x})=\frac{1-p^{\alpha}}{1-p^{-\alpha-n}}\int_{\mathbb{Z}_p^n}\frac{f(\boldsymbol{y})-f(\boldsymbol{x})}{\|\boldsymbol{y}-\boldsymbol{x}\|_p^{\alpha+n}}d\boldsymbol{y}.
    \end{align*}
Thus, on $\mathcal S_2(\mathbb Z_p^n)$,
$D_{\mathbb Z_p^n}^{\alpha}$ agrees with
$D_{\mathbb Z_p^n,\mathcal D}^{\alpha}$. This identity motivates the
mean-zero restriction used below.

\section{Eigenvalues on compact open domains}
Every compact open domain $\Omega\subseteq\mathbb{Q}_p^n$ is a finite disjoint union of maximal balls, i.e. $\Omega=\displaystyle\bigsqcup_{i=1}^{m}B_{L_i}(\boldsymbol{a}_i)$ for some disjoint maximal balls $\{B_{L_i}(\boldsymbol{a}_i)\}_{i=1}^{m}$. Here $B_{L_i}(\boldsymbol{a}_i)=\boldsymbol{a}_i+p^{L_i}\mathbb{Z}_p^n$ with $\|\boldsymbol{a}_i-\boldsymbol{a}_j\|_p>\max\{p^{-L_i},p^{-L_j}\}$ (i.e. $v_p(\boldsymbol{a}_i-\boldsymbol{a}_j)<\min\{L_i,L_j\}$) for all $i\neq j$. Consequently,
\begin{align*}
    L^2(\Omega)=\operatorname{span}\{\boldsymbol{1}_{B_{L_1}(\boldsymbol{a}_1)},\dots,\boldsymbol{1}_{B_{L_m}(\boldsymbol{a}_m)}\}\oplus\bigoplus_{i=1}^{m}L_0^2(B_{L_i}(\boldsymbol{a}_i)).
\end{align*}
It therefore suffices to study the eigenvalue problem on each summand.
\subsection{Dirichlet eigenvalues}
\begin{theorem}
\label{thm:eig_D_L02}
    The eigenvalues of the restriction of
    $D_{\mathbb{Z}_p^n,\mathcal{D}}^{\alpha}$ to
    $L_0^2(\mathbb{Z}_p^n)$ are
    \begin{align*}
        \lambda_{\gamma}=p^{\gamma\alpha},\ \gamma\in\mathbb{N}^*
    \end{align*}
    with multiplicity $m_\gamma=p^{n\gamma}-p^{n(\gamma-1)}$ \cite{MR1092528}. The corresponding eigenfunctions are
    \begin{align*}
        \Psi_{\gamma,\boldsymbol{a},\boldsymbol{j}}(\boldsymbol{x})=p^{\frac{n(\gamma-1)}{2}}\chi_p(p^{-\gamma}\boldsymbol{j}\cdot(\boldsymbol{x}-\boldsymbol{a}))\boldsymbol{1}_{\boldsymbol{a}+p^{\gamma-1}\mathbb{Z}_p^n}(\boldsymbol{x}),
    \end{align*}
    where $\boldsymbol{a}\in\mathbb{Z}_p^n/p^{\gamma-1}\mathbb{Z}_p^n$ and $\boldsymbol{j}\in\mathbb{F}_p^n\backslash\{\boldsymbol{0}\}$ with $(\Psi_{\gamma,\boldsymbol{a},\boldsymbol{j}},\Psi_{\gamma',\boldsymbol{a}',\boldsymbol{j}'})=\delta_{\gamma\gamma'}\delta_{\boldsymbol{a}\boldsymbol{a}'}\delta_{\boldsymbol{j}\boldsymbol{j}'}$.
\end{theorem}

\begin{remark}
    The eigenvalues and the eigenfunctions for $D^\alpha$ in $L^2(\mathbb{Q}_p^n)$ are the same as those in Theorem \ref{thm:eig_D_L02} \cite{MR1918846}. Here, however, $\gamma\in\mathbb{Z}$ and $\boldsymbol{a}\in\mathbb{Q}_p^n/p^{\gamma-1}\mathbb{Z}_p^n$. Thus $\sigma(D^\alpha)=\{0\}\cup\{p^{\gamma\alpha}\}_{\gamma\in\mathbb{Z}}$ is the essential spectrum, and every nonzero spectral value has infinite multiplicity.
\end{remark}

Since
\begin{align*}
    D_{\mathbb{Z}_p^n,\mathcal{D}}^{\alpha}\boldsymbol{1}_{\mathbb{Z}_p^n}(\boldsymbol{x})&=-\frac{1-p^{\alpha}}{1-p^{-\alpha-n}}\left(\int_{\mathbb{Q}_p^n\backslash\mathbb{Z}_p^n}\frac{1}{\|\boldsymbol{y}\|_p^{\alpha+n}}d\boldsymbol{y}\right)\boldsymbol{1}_{\mathbb{Z}_p^n}(\boldsymbol{x})\\
    &=\frac{1-p^{-n}}{1-p^{-\alpha-n}}\boldsymbol{1}_{\mathbb{Z}_p^n}(\boldsymbol{x}),
\end{align*}
so the full self-adjoint Dirichlet operator has the additional eigenvalue
$\frac{1-p^{-n}}{1-p^{-\alpha-n}}$ with eigenfunction
$\boldsymbol{1}_{\mathbb{Z}_p^n}$.

Let $N_{D_{\mathbb{Z}_p^n,\mathcal{D}}^{\alpha}}^{0}(\lambda)$ denote the
eigenvalue counting function of the restriction of the self-adjoint operator
$D_{\mathbb{Z}_p^n,\mathcal{D}}^{\alpha}$ to the reducing subspace
$L_0^2(\mathbb{Z}_p^n)$. Then
\begin{align*}
    \lambda_{\gamma}=p^{\gamma\alpha}\leq\lambda\Rightarrow\gamma\leq\frac{1}{\alpha}\log_p\lambda.
\end{align*}
Therefore,
\begin{align*}
    N_{D_{\mathbb{Z}_p^n,\mathcal{D}}^{\alpha}}^{0}(\lambda)=\sum_{j=1}^{[\frac{1}{\alpha}\log_p\lambda]}m_j=p^{n[\frac{1}{\alpha}\log_p\lambda]}-1
\end{align*}
for $\lambda\geq1$.
Then we have
\begin{align*}
    N_{D_{\mathbb{Z}_p^n,\mathcal{D}}^{\alpha}}^{0}(\lambda)=\begin{cases}
        0,&0\leq\lambda<1,\\
        p^{n[\frac{1}{\alpha}\log_p\lambda]}-1,&\lambda\geq1.
    \end{cases}
\end{align*}
Consequently,
\begin{align*}
    N_{D_{\mathbb{Z}_p^n,\mathcal{D}}^{\alpha}}(\lambda)=\begin{cases}
        0,&0\leq\lambda<\frac{1-p^{-n}}{1-p^{-\alpha-n}},\\
        1,&\frac{1-p^{-n}}{1-p^{-\alpha-n}}\leq\lambda<1,\\
        p^{n[\frac{1}{\alpha}\log_p\lambda]},&\lambda\geq1.
    \end{cases}
\end{align*}

For an arbitrary ball $B_L(\boldsymbol{a})=\boldsymbol{a}+p^{L}\mathbb{Z}_p^n$, let $\boldsymbol{x}=\boldsymbol{a}+p^{L}\boldsymbol{u}$, where $\boldsymbol{x}\in B_L(\boldsymbol{a})$ and $\boldsymbol{u}\in \mathbb{Z}_p^n$, and let $\varphi(\boldsymbol{u})$ be a function on $\mathbb{Z}_p^n$. Set
\begin{align*}
    \psi(\boldsymbol{x})=\varphi\left(\frac{\boldsymbol{x}-\boldsymbol{a}}{p^{L}}\right).
\end{align*}
Then
\begin{align*}
    D_{B_L(\boldsymbol{a}),\mathcal{D}}^{\alpha}\psi(\boldsymbol{x})&=\frac{1-p^{\alpha}}{1-p^{-\alpha-n}}\int_{\mathbb{Q}_p^n}\frac{\widetilde{\psi}(\boldsymbol{y})-\widetilde{\psi}(\boldsymbol{x})}{\|\boldsymbol{y}-\boldsymbol{x}\|_p^{\alpha+n}}d\boldsymbol{y}\\
    &=\frac{1-p^{\alpha}}{1-p^{-\alpha-n}}\int_{\mathbb{Q}_p^n}\frac{\widetilde{\varphi}(\boldsymbol{v})-\widetilde{\varphi}(\boldsymbol{u})}{p^{-L(n+\alpha)}\|\boldsymbol{v}-\boldsymbol{u}\|_p^{\alpha+n}}p^{-nL}d\boldsymbol{v}\\
    &=p^{\alpha L}\frac{1-p^{\alpha}}{1-p^{-\alpha-n}}\int_{\mathbb{Q}_p^n}\frac{\widetilde{\varphi}(\boldsymbol{v})-\widetilde{\varphi}(\boldsymbol{u})}{\|\boldsymbol{v}-\boldsymbol{u}\|_p^{\alpha+n}}d\boldsymbol{v}\\
    &=p^{\alpha L}D_{\mathbb{Z}_p^n,\mathcal{D}}^{\alpha}\varphi(\boldsymbol{u}).
\end{align*}
Similarly, we have
\begin{align}
\label{eq:N0_D_Ball_lambda}
    N_{D_{B_L(\boldsymbol{a}),\mathcal{D}}^{\alpha}}^{0}(\lambda)=\begin{cases}
        0,&0\leq\lambda<p^{\alpha L},\\
        p^{n[\frac{1}{\alpha}\log_p\lambda]-nL}-1=|B_L(\boldsymbol{a})|p^{n[\frac{1}{\alpha}\log_p\lambda]}-1,&\lambda\geq p^{\alpha L},
    \end{cases}
\end{align}
and
\begin{align}
\label{eq:N_D_Ball_lambda}
    N_{D_{B_L(\boldsymbol{a}),\mathcal{D}}^{\alpha}}(\lambda)=\begin{cases}
        0,&0\leq\lambda<p^{\alpha L}\frac{1-p^{-n}}{1-p^{-\alpha-n}},\\
        1,&p^{\alpha L}\frac{1-p^{-n}}{1-p^{-\alpha-n}}\leq\lambda<p^{\alpha L},\\
        |B_L(\boldsymbol{a})|p^{n[\frac{1}{\alpha}\log_p\lambda]},&\lambda\geq p^{\alpha L}.
    \end{cases}
\end{align}

For $\Omega=\displaystyle\bigsqcup_{i=1}^{m}B_{L_i}(\boldsymbol{a}_i)$, first consider the subspace $\displaystyle\bigoplus_{i=1}^{m}L_0^2(B_{L_i}(\boldsymbol{a}_i))$. The eigenvalues of the restriction to this subspace are $\lambda_{\gamma,i}=p^{\alpha(\gamma+L_i)}$ with the same multiplicity $m_{\gamma,i}=p^{n\gamma}-p^{n(\gamma-1)}$ as in Theorem \ref{thm:eig_D_L02} since the eigenfunctions are orthogonal to the constant functions on the balls.

On the subspace $\operatorname{span}\{\boldsymbol{1}_{B_{L_1}(\boldsymbol{a}_1)},\dots,\boldsymbol{1}_{B_{L_m}(\boldsymbol{a}_m)}\}$ part, we compute
\begin{equation}
\label{eq:D_Omiga_alpha_1}
\begin{aligned}
    &D_{\Omega,\mathcal{D}}^{\alpha}\boldsymbol{1}_{B_{L_i}(\boldsymbol{a}_i)}(\boldsymbol{x})\\
    &=\frac{p^{\alpha}-1}{1-p^{-\alpha-n}}\int_{\mathbb{Q}_p^n}\frac{\boldsymbol{1}_{B_{L_i}(\boldsymbol{a}_i)}(\boldsymbol{x})-\boldsymbol{1}_{B_{L_i}(\boldsymbol{a}_i)}(\boldsymbol{y})}{\|\boldsymbol{y}-\boldsymbol{x}\|_p^{\alpha+n}}d\boldsymbol{y}\\
    &=\frac{p^{\alpha}-1}{1-p^{-\alpha-n}}\sum_{j=1}^{m}\int_{B_{L_j}(\boldsymbol{a}_j)}\frac{\boldsymbol{1}_{B_{L_i}(\boldsymbol{a}_i)}(\boldsymbol{x})-\delta_{ij}}{\|\boldsymbol{y}-\boldsymbol{x}\|_p^{\alpha+n}}d\boldsymbol{y}+\frac{p^{\alpha}-1}{1-p^{-\alpha-n}}\left(\int_{\Omega^c}\frac{1}{\|\boldsymbol{y}-\boldsymbol{x}\|_p^{\alpha+n}}d\boldsymbol{y}\right)\boldsymbol{1}_{B_{L_i}(\boldsymbol{a}_i)}(\boldsymbol{x}).
\end{aligned}
\end{equation}
\begin{enumerate}
    \item[]\textbf{Case 1.} $\boldsymbol{x}\in B_{L_i}(\boldsymbol{a}_i)$: then \eqref{eq:D_Omiga_alpha_1} becomes
    \begin{align*}
        &\frac{p^{\alpha}-1}{1-p^{-\alpha-n}}\int_{B_{L_i}^c(\boldsymbol{a}_i)}\frac{1}{\|\boldsymbol{y}-\boldsymbol{x}\|_p^{\alpha+n}}d\boldsymbol{y}\\
        =&\frac{p^{\alpha}-1}{1-p^{-\alpha-n}}\sum_{\substack{j=1\\j\neq i}}^{m}\int_{B_{L_j}(\boldsymbol{a}_j)}\frac{1}{\|\boldsymbol{y}-\boldsymbol{x}\|_p^{\alpha+n}}d\boldsymbol{y}+\frac{p^{\alpha}-1}{1-p^{-\alpha-n}}\int_{\Omega^c}\frac{1}{\|\boldsymbol{y}-\boldsymbol{x}\|_p^{\alpha+n}}d\boldsymbol{y}\\
        =&C_{n,\alpha}\sum_{\substack{j=1\\j\neq i}}^{m}\frac{V_j}{d_{ij}^{\alpha+n}}+\Gamma_i,
    \end{align*}
    where $C_{n,\alpha}=\frac{p^{\alpha}-1}{1-p^{-\alpha-n}}$, $d_{ij}=\operatorname{dist}(B_{L_i}(\boldsymbol{a}_i),B_{L_j}(\boldsymbol{a}_j))=\|\boldsymbol{a}_i-\boldsymbol{a}_j\|_p$, $V_i=|B_{L_i}(\boldsymbol{a}_i)|=p^{-nL_i}$ and $\Gamma_i=C_{n,\alpha}\int_{\Omega^c}\frac{1}{\|\boldsymbol{y}-\boldsymbol{x}\|_p^{\alpha+n}}d\boldsymbol{y}$.
    Note that \eqref{eq:D_Omiga_alpha_1} can also be written as
    \begin{align*}
        &\frac{p^{\alpha}-1}{1-p^{-\alpha-n}}\int_{B_{L_i}^c(\boldsymbol{a}_i)}\frac{1}{\|\boldsymbol{y}-\boldsymbol{x}\|_p^{\alpha+n}}d\boldsymbol{y}\\
        =&\frac{p^{\alpha}-1}{1-p^{-\alpha-n}}\int_{\mathbb{Q}_p^n\backslash p^{L_i}\mathbb{Z}_p^n}\frac{1}{\|\boldsymbol{z}\|^{\alpha+n}}d\boldsymbol{z}\\
        =&\frac{p^{\alpha}-1}{1-p^{-\alpha-n}}p^{\alpha L_i}\frac{1-p^{-n}}{p^{\alpha}-1}\\
        =&p^{\alpha L_i}\frac{1-p^{-n}}{1-p^{-\alpha-n}}.
    \end{align*}
    \item[]\textbf{Case 2.} $\boldsymbol{x}\in B_{L_j}(\boldsymbol{a}_j)$, $j\neq i$: then \eqref{eq:D_Omiga_alpha_1} becomes
    \begin{align*}
        -\frac{p^{\alpha}-1}{1-p^{-\alpha-n}}\int_{B_{L_i}(\boldsymbol{a}_i)}\frac{1}{\|\boldsymbol{y}-\boldsymbol{x}\|_p^{\alpha+n}}d\boldsymbol{y}=-C_{n,\alpha}\frac{V_i}{d_{ij}^{\alpha+n}}.
    \end{align*}
\end{enumerate}
Passing through the normalized basis
\begin{align*}
    \boldsymbol{e}_i(\boldsymbol{x})=\frac{\boldsymbol{1}_{B_{L_i}(\boldsymbol{a}_i)}(\boldsymbol{x})}{\sqrt{V_i}}
\end{align*}
to ensure that $\|\boldsymbol{e}_i(\boldsymbol{x})\|_{L^2(\Omega)}=1$, the representation matrix of $D_{\Omega,\mathcal{D}}^{\alpha}$ is a symmetric matrix $\boldsymbol{H}$ defined by
\begin{align*}
    h_{ij}=\begin{cases}
        C_{n,\alpha}\displaystyle\sum_{\substack{k=1\\k\neq i}}^{m}\frac{V_k}{d_{ik}^{\alpha+n}}+\Gamma_i=p^{\alpha L_i}\frac{1-p^{-n}}{1-p^{-\alpha-n}},&i=j,\\
        -C_{n,\alpha}\frac{\sqrt{V_iV_j}}{d_{ij}^{\alpha+n}},&i\neq j,
    \end{cases}
\end{align*}
satisfying
\begin{align*}
D_{\Omega,\mathcal{D}}^{\alpha}\boldsymbol{e}_i(\boldsymbol{x})=\sum_{j=1}^{m}h_{ji}\boldsymbol{e}_j(\boldsymbol{x}).
\end{align*}
There exists an orthogonal matrix $\boldsymbol{Q}$ that diagonalizes
$\boldsymbol{H}$. The constant-on-each-ball subspace and the spaces
$L_0^2(B_{L_i}(\boldsymbol a_i))$ reduce the Dirichlet form. Hence its
spectrum, as a set, is
\begin{align*}
    \sigma(\boldsymbol{H})\cup
    \{\lambda_{\gamma,i}:\gamma\in\mathbb N^*,\ 1\leq i\leq m\}.
\end{align*}
If the same numerical value occurs in different summands, the corresponding
multiplicities are added.
In summary, we have
\begin{theorem}
\label{thm:eig_Dirichlet_compact}
    The eigenvalues of the self-adjoint Dirichlet operator
    $D_{\Omega,\mathcal{D}}^{\alpha}$ are
    \begin{align*}
        \sigma(\boldsymbol{H})\cup\left\{\lambda_{\gamma,i}=p^{\alpha(\gamma+L_i)}\right\}_{\substack{\gamma\in\mathbb{N}^*\\1\leq i\leq m}}.
    \end{align*}
    Each value $\lambda_{\gamma,i}$ has multiplicity
    $m_{\gamma,i}=p^{n\gamma}-p^{n(\gamma-1)}$ in its mean-zero summand.
    If the same numerical value occurs in different summands, all
    multiplicities are added. The corresponding normalized eigenfunctions are
    \begin{align*}
        \Psi_{i}(\boldsymbol{x})=\sum_{j=1}^{m}q_{ji}\boldsymbol{e}_j(\boldsymbol{x}),
    \end{align*}
    and
    \begin{align*}
        \Psi_{\gamma,\boldsymbol{a},\boldsymbol{j}}^{(i)}(\boldsymbol{x})
        =p^{\frac n2(\gamma+L_i-1)}
        \chi_p\!\left(p^{-(\gamma+L_i)}\boldsymbol j\cdot
        (\boldsymbol x-\boldsymbol a_i-p^{L_i}\boldsymbol a)\right)
        \boldsymbol 1_{\boldsymbol a_i+p^{L_i}\boldsymbol a
        +p^{\gamma+L_i-1}\mathbb Z_p^n}(\boldsymbol x).
    \end{align*}
    Here $\boldsymbol{a}\in\mathbb{Z}_p^n/p^{\gamma-1}\mathbb{Z}_p^n$ and $\boldsymbol{j}\in\mathbb{F}_p^n\backslash\{\boldsymbol{0}\}$.
\end{theorem}

Similarly, let $N_{D_{\Omega,\mathcal{D}}^{\alpha}}^{0}(\lambda)$ denote the
eigenvalue counting function of the restriction of the self-adjoint operator
$D_{\Omega,\mathcal{D}}^{\alpha}$ to the reducing subspace
$\displaystyle\bigoplus_{i=1}^{m}L_0^2(B_{L_i}(\boldsymbol{a}_i))$.
For $\lambda\geq \displaystyle\max_{1\leq i\leq m}\left\{p^{\alpha L_i}\right\}$ we have
\begin{equation}
\label{eq:N0_D_Omiga_lambda}
\begin{aligned}
    N_{D_{\Omega,\mathcal{D}}^{\alpha}}^{0}(\lambda)&=\sum_{i=1}^{m}N_{D_{B_{L_i}(\boldsymbol{a}_i),\mathcal{D}}^{\alpha}}^{0}(\lambda)\\
    &=\sum_{i=1}^{m}\left(|B_{L_i}(\boldsymbol{a}_i)|p^{n[\frac{1}{\alpha}\log_p\lambda]}-1\right)\\
    &=\left(\sum_{i=1}^{m}|B_{L_i}(\boldsymbol{a}_i)|\right)p^{n[\frac{1}{\alpha}\log_p\lambda]}-m\\
    &=|\Omega|p^{n[\frac{1}{\alpha}\log_p\lambda]}-m.
\end{aligned}
\end{equation}
Let $L=\max\limits_{1\leq i\leq m}L_i$. The Fourier support of every
$u\in\operatorname{span}\{\boldsymbol e_1,\ldots,\boldsymbol e_m\}$ is
contained in $\{\|\boldsymbol\xi\|_p\leq p^L\}$. Therefore
$0<\boldsymbol H\leq p^{\alpha L}I$. Hence, if
$\lambda\geq p^{\alpha L}$, then
\begin{align}
\label{eq:N_Dirichlet_compact}
    N_{D_{\Omega,\mathcal{D}}^{\alpha}}(\lambda)=|\Omega|p^{n[\frac{1}{\alpha}\log_p\lambda]}.
\end{align}

For completeness, the entrywise Gershgorin estimate from the finite-ball
description gives the supplementary geometric bounds
\begin{align*}
    \inf\sigma(\boldsymbol H)
    &\geq\min_{1\leq i\leq m}
    \left\{
    C_{n,\alpha}\sum_{\substack{j=1\\j\neq i}}^m
    \frac{V_j-\sqrt{V_iV_j}}{d_{ij}^{n+\alpha}}+\Gamma_i
    \right\},\\
    \sup\sigma(\boldsymbol H)
    &\leq\max_{1\leq i\leq m}
    \left\{
    C_{n,\alpha}\sum_{\substack{j=1\\j\neq i}}^m
    \frac{V_j+\sqrt{V_iV_j}}{d_{ij}^{n+\alpha}}+\Gamma_i
    \right\}.
\end{align*}
These bounds retain the explicit dependence on the radii and pairwise ball
distances. The uniform operator estimate
$0<\boldsymbol H\leq p^{\alpha L}I$ used above is sharper for the counting
threshold needed later.

\begin{remark}
The symmetric representation by $\boldsymbol H$ also makes the lower
Gershgorin estimate transparent. Indeed, pairing the terms indexed by
$i$ and $j$ gives
\begin{align*}
    \sum_{i=1}^{m}\sum_{\substack{j=1\\j\neq i}}^{m}
    \frac{V_j-\sqrt{V_iV_j}}{d_{ij}^{n+\alpha}}
    =\sum_{1\leq i<j\leq m}
    \frac{\left(\sqrt{V_i}-\sqrt{V_j}\right)^2}
    {d_{ij}^{n+\alpha}}
    \geq0.
\end{align*}
This is the additional geometric information that is obscured if one uses
the corresponding non-symmetric matrix in the unnormalized basis.
\end{remark}

\subsection{Neumann eigenvalues}
Let $M_\Gamma$ denote multiplication by the function that equals
$\Gamma_i$ on $B_{L_i}(\boldsymbol a_i)$. Since $\Omega$ is compact and
open, the Dirichlet and Neumann forms have the same form domain, and the
operator identity
\begin{align*}
    D_{\Omega,\mathcal{N}}^{\alpha}
    =D_{\Omega,\mathcal{D}}^{\alpha}-M_\Gamma
\end{align*}
gives the eigenvalues
\begin{align*}
    \mu_{\gamma,i}=\lambda_{\gamma,i}-\Gamma_i.
\end{align*}
Then the representation matrix of $D_{\Omega,\mathcal{N}}^{\alpha}$ with respect to the basis $\boldsymbol{e}_i(\boldsymbol{x})=\frac{\boldsymbol{1}_{B_{L_i}(\boldsymbol{a}_i)}(\boldsymbol{x})}{\sqrt{V_i}}$ is a symmetric matrix $\boldsymbol{R}$ defined by
\begin{align*}
    r_{ij}=\begin{cases}
        C_{n,\alpha}\displaystyle\sum_{\substack{k=1\\k\neq i}}^{m}\frac{V_k}{d_{ik}^{\alpha+n}},&i=j,\\
        -C_{n,\alpha}\frac{\sqrt{V_iV_j}}{d_{ij}^{\alpha+n}},&i\neq j.
    \end{cases}
\end{align*}
There exists an orthogonal matrix $\boldsymbol{W}$ that diagonalizes $\boldsymbol{R}$. Similarly, we have
\begin{theorem}
\label{thm:eig_Neumann_compact}
    The eigenvalues of the self-adjoint Neumann operator
    $D_{\Omega,\mathcal{N}}^{\alpha}$ are
    \begin{align*}
        \sigma(\boldsymbol{R})\cup\{\mu_{\gamma,i}=p^{\alpha(\gamma+L_i)}-\Gamma_i\}_{\substack{\gamma\in\mathbb{N}^*\\1\leq i\leq m}}.
    \end{align*}
    Each value $\mu_{\gamma,i}$ has multiplicity
    $m_{\gamma,i}=p^{n\gamma}-p^{n(\gamma-1)}$ in its mean-zero summand;
    coincident values from different summands have their multiplicities
    added. The corresponding normalized eigenfunctions are
    \begin{align*}
        \Phi_{i}(\boldsymbol{x})=\sum_{j=1}^{m}w_{ji}\boldsymbol{e}_j(\boldsymbol{x}),
    \end{align*}
    and
    \begin{align*}
        \Phi_{\gamma,\boldsymbol{a},\boldsymbol{j}}^{(i)}(\boldsymbol{x})
        =\Psi_{\gamma,\boldsymbol{a},\boldsymbol{j}}^{(i)}(\boldsymbol{x}).
    \end{align*}
    Here $\boldsymbol{a}\in\mathbb{Z}_p^n/p^{\gamma-1}\mathbb{Z}_p^n$ and $\boldsymbol{j}\in\mathbb{F}_p^n\backslash\{\boldsymbol{0}\}$.
\end{theorem}

When $m=1$, $\Gamma=p^{\alpha L}\frac{1-p^{-n}}{1-p^{-\alpha-n}}$, so that
\begin{align*}
    N_{D_{B_L(\boldsymbol{a}),\mathcal{N}}^{\alpha}}(\lambda)&=N_{D_{B_L(\boldsymbol{a}),\mathcal{D}}^{\alpha}}\left(\lambda+p^{\alpha L}\frac{1-p^{-n}}{1-p^{-\alpha-n}}\right)\\
    &=\begin{cases}
        1,&0\leq\lambda<p^{\alpha L}-p^{\alpha L}\frac{1-p^{-n}}{1-p^{-\alpha-n}},\\
        |B_L(\boldsymbol{a})|p^{n\left[\frac{1}{\alpha}\log_p\left(\lambda+p^{\alpha L}\frac{1-p^{-n}}{1-p^{-\alpha-n}}\right)\right]},&\lambda\geq p^{\alpha L}-p^{\alpha L}\frac{1-p^{-n}}{1-p^{-\alpha-n}}.
    \end{cases}
\end{align*}

The form comparison on the constant-on-each-ball subspace gives
$0\leq\boldsymbol R\leq\boldsymbol H\leq p^{\alpha L}I$. Thus, for
$\lambda\geq p^{\alpha L}$,
\begin{align}
\label{eq:N_Neumann_compact}
    N_{D_{\Omega,\mathcal{N}}^{\alpha}}(\lambda)=\sum_{i=1}^{m}|B_{L_i}(\boldsymbol{a}_i)|p^{n[\frac{1}{\alpha}\log_p(\lambda+\Gamma_i)]}.
\end{align}
Equation \eqref{eq:N_Neumann_compact} yields the bounds
\begin{align}
\label{eq:rough_estimate_Neumann_counting}
    |\Omega|p^{n[\frac{1}{\alpha}\log_p(\lambda+\Gamma_{\min})]}\leq N_{D_{\Omega,\mathcal{N}}^{\alpha}}(\lambda)\leq|\Omega|p^{n[\frac{1}{\alpha}\log_p(\lambda+\Gamma_{\max})]},
\end{align}
where $\Gamma_{\min}=\displaystyle\min_{1\leq i\leq m}\Gamma_{i}$ and $\Gamma_{\max}=\displaystyle\max_{1\leq i\leq m}\Gamma_{i}$.

\begin{remark}
The entrywise Gershgorin estimate for $\boldsymbol R$ gives the additional
geometric bounds
\begin{align*}
    \inf\sigma(\boldsymbol R)
    &\geq\min_{1\leq i\leq m}
    C_{n,\alpha}\sum_{\substack{j=1\\j\neq i}}^{m}
    \frac{V_j-\sqrt{V_iV_j}}{d_{ij}^{n+\alpha}},\\
    \sup\sigma(\boldsymbol R)
    &\leq\max_{1\leq i\leq m}
    C_{n,\alpha}\sum_{\substack{j=1\\j\neq i}}^{m}
    \frac{V_j+\sqrt{V_iV_j}}{d_{ij}^{n+\alpha}}.
\end{align*}
Consequently, the exact counting formula
\eqref{eq:N_Neumann_compact} is already valid whenever
\begin{align*}
    \lambda\geq\max_{1\leq i\leq m}
    \left\{
    C_{n,\alpha}\sum_{\substack{j=1\\j\neq i}}^{m}
    \frac{V_j+\sqrt{V_iV_j}}{d_{ij}^{n+\alpha}},
    \ p^{\alpha L_i}-\Gamma_i
    \right\}.
\end{align*}
This is a geometry-dependent alternative to the uniform sufficient
threshold $\lambda\geq p^{\alpha L}$; it may be sharper for a particular
configuration of maximal balls.
\end{remark}

\subsection{Extension to bounded open domains}

The finite maximal-ball decomposition was treated above. We now consider the
remaining case, in which the bounded open set has countably many maximal
balls.

\begin{theorem}[Orthogonal decomposition on a bounded open set]
Let $\Omega\subseteq\mathbb Q_p^n$ be bounded and open, with maximal-ball
decomposition
\begin{align*}
    \Omega=\bigsqcup_{i=1}^{\infty}B_{L_i}(\boldsymbol a_i),
\end{align*}
and set
\begin{align*}
V=\overline{\operatorname{span}
\{\boldsymbol 1_{B_{L_i}(\boldsymbol a_i)}:i\geq1\}}^{L^2(\Omega)}.
\end{align*}
Then
\begin{align}
\label{eq:bounded_open_L2_decomposition}
L^2(\Omega)=V\oplus
\bigoplus_{i=1}^{\infty}L_0^2(B_{L_i}(\boldsymbol a_i)).
\end{align}
Let $\boldsymbol H$ be the nonnegative self-adjoint operator in $V$
associated with the restriction of $\mathcal E_{\Omega,\mathcal D}$ to
\begin{align*}
    V\cap H_0^{\alpha/2}(\Omega).
\end{align*}
Then
\begin{align}
\label{eq:bounded_open_Dirichlet_direct_sum}
D_{\Omega,\mathcal D}^{\alpha}
=\boldsymbol H\oplus\bigoplus_{i=1}^{\infty}
\left(D_{B_{L_i}(\boldsymbol a_i),\mathcal D}^{\alpha}
\big|_{L_0^2(B_{L_i}(\boldsymbol a_i))}\right),
\end{align}
and
\begin{align}
\label{eq:bounded_open_Dirichlet_spectrum}
\sigma(D_{\Omega,\mathcal D}^{\alpha})
=\sigma(\boldsymbol H)\cup
\{p^{\alpha(\gamma+L_i)}:i\geq1,\ \gamma\in\mathbb N^*\}.
\end{align}
Both $D_{\Omega,\mathcal D}^{\alpha}$ and $\boldsymbol H$ have compact
resolvent, and
\begin{align}
\label{eq:bounded_open_Dirichlet_counting}
N_{D_{\Omega,\mathcal D}^{\alpha}}(\lambda)
=N_{\boldsymbol H}(\lambda)
+\sum_{i=1}^{\infty}
N_{D_{B_{L_i}(\boldsymbol a_i),\mathcal D}^{\alpha}}^0(\lambda),
\end{align}
where the sum is finite for every fixed $\lambda$.

For $\boldsymbol x\in B_{L_i}(\boldsymbol a_i)$, let $\Gamma_i$ have the
same meaning as above, and let $\boldsymbol R$ be the nonnegative
self-adjoint operator in $V$ associated with the restriction of
$\mathcal E_{\Omega,\mathcal N}$ to
$V\cap H^{\alpha/2}(\Omega)$. Then
\begin{align}
\label{eq:bounded_open_Neumann_direct_sum}
D_{\Omega,\mathcal N}^{\alpha}
=\boldsymbol R\oplus\bigoplus_{i=1}^{\infty}
\left(\left(D_{B_{L_i}(\boldsymbol a_i),\mathcal D}^{\alpha}
-\Gamma_iI\right)\big|_{L_0^2(B_{L_i}(\boldsymbol a_i))}\right),
\end{align}
and
\begin{align}
\label{eq:bounded_open_Neumann_spectrum}
\sigma(D_{\Omega,\mathcal N}^{\alpha})
=\sigma(\boldsymbol R)\cup
\{p^{\alpha(\gamma+L_i)}-\Gamma_i:
i\geq1,\ \gamma\in\mathbb N^*\}.
\end{align}
The second summand in \eqref{eq:bounded_open_Neumann_direct_sum} has compact
resolvent. Consequently,
\begin{align*}
\sigma_{\mathrm{ess}}(D_{\Omega,\mathcal N}^{\alpha})
=\sigma_{\mathrm{ess}}(\boldsymbol R).
\end{align*}
\end{theorem}
\begin{proof}
The $L^2$ decomposition follows by separating, on each maximal ball, the
mean and the mean-zero part. For $u\in\mathcal S(\Omega)$ only finitely many
of these parts occur. If
$w_i\in\mathcal S_2(B_{L_i}(\boldsymbol a_i))$ and
$\boldsymbol x\notin B_{L_i}(\boldsymbol a_i)$, then the distance from
$\boldsymbol x$ to $\boldsymbol y\in B_{L_i}(\boldsymbol a_i)$ is constant.
Therefore
\begin{align*}
D^\alpha w_i(\boldsymbol x)
=-C_{n,\alpha}\int_{B_{L_i}(\boldsymbol a_i)}
\frac{w_i(\boldsymbol y)}
{\|\boldsymbol x-\boldsymbol y\|_p^{n+\alpha}}d\boldsymbol y=0.
\end{align*}
Together with $\int_{\mathbb Q_p^n}D^\alpha w_i=0$, this also shows that
$w_i$ is form-orthogonal to every function that is constant on each maximal
ball.
Thus distinct mean-zero parts are mutually form-orthogonal and are
form-orthogonal to the constant-on-each-maximal-ball part. The corresponding
orthogonal projections are contractions in the Dirichlet form norm.
Passing to the closure of $\mathcal S(\Omega)$ gives a form-orthogonal
decomposition of $H_0^{\alpha/2}(\Omega)$. It also shows that finite linear
combinations of the maximal-ball indicators form a core of
$V\cap H_0^{\alpha/2}(\Omega)$. The representation theorem for orthogonal
sums of closed forms proves \eqref{eq:bounded_open_Dirichlet_direct_sum}.

The compact-resolvent assertion was proved in Section 2. For every fixed
$\lambda$, only finitely many pairs $(i,\gamma)$ satisfy
$p^{\alpha(\gamma+L_i)}\leq\lambda$: a fixed bounded ball contains only
finitely many pairwise disjoint balls whose levels are bounded above.
This proves \eqref{eq:bounded_open_Dirichlet_spectrum} without an additional
closure and gives \eqref{eq:bounded_open_Dirichlet_counting}.

For the Neumann form, let $w_i$ have mean zero on
$B_{L_i}(\boldsymbol a_i)$ and vanish on the other maximal balls, and let
$v$ be constant on every maximal ball. On
$B_{L_i}(\boldsymbol a_i)\times B_{L_i}(\boldsymbol a_i)$ the difference of
$v$ vanishes. On the product of two distinct maximal balls, both the kernel
and the difference of $v$ are constant, while
\begin{align*}
    \int_{B_{L_i}(\boldsymbol a_i)}w_i(\boldsymbol x)
    d\boldsymbol x=0.
\end{align*}
It follows directly that
\begin{align*}
    \mathcal E_{\Omega,\mathcal N}(w_i,v)=0,
    \qquad
    \mathcal E_{\Omega,\mathcal N}(w_i,w_j)=0
    \quad(i\neq j).
\end{align*}
Replacing a function on each maximal ball by its average therefore preserves
the mean component and removes only mutually orthogonal, nonnegative
mean-zero contributions. Hence the averaging projection is a contraction in
the intrinsic form norm. Passing to the form closure proves the Neumann
form-orthogonal decomposition. On the $i$-th mean-zero part, the boundary
term is multiplication by the constant $\Gamma_i$, which proves
\eqref{eq:bounded_open_Neumann_direct_sum}. Moreover,
\begin{align*}
0\leq\Gamma_i
\leq p^{\alpha L_i}\frac{1-p^{-n}}{1-p^{-\alpha-n}}.
\end{align*}
It follows that the shifted mean-zero eigenvalues tend to infinity locally
uniformly in $i$, so the second summand has compact resolvent and its
spectral values have no finite accumulation point. This proves
\eqref{eq:bounded_open_Neumann_spectrum} and the assertion concerning the
essential spectrum. The set $\sigma(\boldsymbol R)$ may contain continuous
spectrum; no diagonalization of $\boldsymbol R$ by an infinite orthogonal
matrix is asserted.
\end{proof}

\begin{lemma}
For every ball $B\subseteq\mathbb Q_p^n$, multiplication by
$\boldsymbol 1_B$ is bounded on $H^s(\mathbb Q_p^n)$ and maps
$H_0^s(\Omega)$ into itself. In particular, if
$u\in V\cap H_0^s(\Omega)$, then the part of $u$ supported in any union of
maximal balls contained in $B$ also belongs to $V\cap H_0^s(\Omega)$.
\end{lemma}
\begin{proof}
The indicator of a ball is a Bruhat--Schwartz function. Its Fourier
transform is compactly supported. Since
\begin{align*}
\widehat{\boldsymbol 1_Bu}
=\widehat{\boldsymbol 1_B}*\widehat u
\end{align*}
and, on the support of $\widehat{\boldsymbol 1_B}$,
$\max\{1,\|\boldsymbol\xi\|_p\}^s$ is bounded by a fixed multiple of
$\max\{1,\|\boldsymbol\xi-\boldsymbol\eta\|_p\}^s$, Young's inequality
gives the required $H^s$ bound. If
$u_j\in\mathcal S(\Omega)$ converges to $u$ in $H^s$, then
$\boldsymbol 1_Bu_j\in\mathcal S(\Omega)$ and converges to
$\boldsymbol 1_Bu$. The last assertion follows from the maximal-ball
decomposition.
\end{proof}

\section{Weyl's law for the Vladimirov-Taibleson operator}
\subsection{Leading asymptotic term}
Given a bounded open set $\Omega\subseteq\mathbb{Q}_p^n$, we study Weyl's law for the Vladimirov-Taibleson operator. Our goal is to follow the bracketing arguments in Weyl's original paper \cite{MR1511670}, namely, approximating $\Omega$ from the external and internal by cubes. Note that the topology on the product space $\mathbb{Q}_p^n$ is induced by the max norm, so $p$-adic balls play the role of cubes in these approximations. We compare the operator on $\Omega$ with the operators on the external and internal approximation domains whose eigenvalue counting functions are explicit, and then apply a squeezing argument to prove Weyl's law. To solve this problem, we introduce the Min-Max Principle to handle the monotonicity of the eigenvalue counting functions:

\begin{lemma}[Courant-Fischer-Weyl min-max principle]
\label{lem:Min-Max_Principle}
    Let $A$ be a self-adjoint operator that is bounded below and has purely discrete spectrum, and let $\mathcal{E}_A(u,u)$ be its associated closed quadratic form. Its $j$-th eigenvalue can be characterized by the Rayleigh–Ritz formula:
    \begin{align*}
    \lambda_j(A) = \min_{\substack{L \subseteq \mathcal{D}(\mathcal{E}_A) \\ \dim L = j}} \max_{\substack{u \in L \\ \|u\| = 1}} \mathcal{E}_A(u, u).
    \end{align*}
\end{lemma}

\begin{corollary}[Spectral monotonicity]
\label{cor:Spectral_monotonicity}
    Suppose $\mathcal{D}(\mathcal{E}_B)\subseteq\mathcal{D}(\mathcal{E}_A)$ and $\mathcal{E}_A(u,u)\leq\mathcal{E}_B(u,u)$ for every $u\in\mathcal{D}(\mathcal{E}_B)$. Then the corresponding eigenvalue counting functions satisfy the reverse inequality:
    \begin{align*}
        N_A(\lambda)\geq N_B(\lambda).
    \end{align*}
\end{corollary}

We now prove Weyl's law for the Vladimirov-Taibleson operator, beginning with the Dirichlet operator $D_{\Omega,\mathcal{D}}^{\alpha}$.

\begin{theorem}
\label{thm:Weyl_law_Dirichlet_operator}
    Let $\Omega\subseteq\mathbb{Q}_p^n$ be a bounded, open, Jordan measurable set (i.e. $\mathcal{H}^n(\partial\Omega)=0$). Then $N_{D_{\Omega,\mathcal{D}}^{\alpha}}(\lambda)\sim|\Omega|p^{n[\frac{1}{\alpha}\log_p\lambda]}$ as $\lambda\to+\infty$.
\end{theorem}
\begin{proof}
Since $\Omega$ is open, it is a countable disjoint union of maximal balls, i.e. $\Omega=\displaystyle\bigsqcup_{i=1}^{\infty}B_{L_i}(\boldsymbol{a}_i)$ for some disjoint maximal balls $\{B_{L_i}(\boldsymbol{a}_i)\}_{i=1}^{\infty}$. Let
\begin{align*}
    E_k=\bigsqcup_{i=1}^{k}B_{L_i}(\boldsymbol{a}_i).
\end{align*}
Then we have $E_1\subseteq E_2\subseteq\dots\subseteq E_k\subseteq\dots\subseteq\Omega$. Let
\begin{align*}
    F_k=\{\boldsymbol{x}\in\mathbb{Q}_p^n:\
    \operatorname{dist}(\boldsymbol{x},\Omega)\leq p^{-k}\}.
\end{align*}
Then $F_1\supseteq F_2\supseteq\cdots\supseteq\Omega$, and each $F_j$
is compact and open because $\Omega$ is bounded. The inclusions of test
function spaces, followed by closure in $H^{\alpha/2}(\mathbb Q_p^n)$,
give
\begin{align*}
    H_0^{\alpha/2}(E_j)
    \subseteq H_0^{\alpha/2}(\Omega)
    \subseteq H_0^{\alpha/2}(F_j).
\end{align*}
Therefore, by Lemma \ref{lem:Min-Max_Principle}, we have
\begin{align*}
    N_{D_{E_j,\mathcal{D}}^{\alpha}}(\lambda)\leq N_{D_{\Omega,\mathcal{D}}^{\alpha}}(\lambda)\leq N_{D_{F_j,\mathcal{D}}^{\alpha}}(\lambda).
\end{align*}
For each fixed $j$, both $E_j$ and $F_j$ are compact. Hence by \eqref{eq:N_Dirichlet_compact}, there exists $\Lambda_j$ such that for every $\lambda\geq\Lambda_j$, we have $N_{D_{E_j,\mathcal{D}}^{\alpha}}(\lambda)=|E_j|p^{n[\frac{1}{\alpha}\log_p\lambda]}$ and $N_{D_{F_j,\mathcal{D}}^{\alpha}}(\lambda)=|F_j|p^{n[\frac{1}{\alpha}\log_p\lambda]}$. Consequently,
\begin{align*}
    \frac{|E_j|}{|\Omega|}\leq\varliminf_{\lambda\to+\infty}\frac{N_{D_{\Omega,\mathcal{D}}^{\alpha}}(\lambda)}{|\Omega|p^{n[\frac{1}{\alpha}\log_p\lambda]}}\leq\varlimsup_{\lambda\to+\infty}\frac{N_{D_{\Omega,\mathcal{D}}^{\alpha}}(\lambda)}{|\Omega|p^{n[\frac{1}{\alpha}\log_p\lambda]}}\leq\frac{|F_j|}{|\Omega|}.
\end{align*}
Letting $j\to+\infty$, we have $\frac{|E_j|}{|\Omega|}\to1$ and $\frac{|F_j|}{|\Omega|}\to\frac{|\overline{\Omega}|}{|\Omega|}=1$ since $\Omega$ is Jordan measurable. In summary, we have
\begin{align*}
    \lim_{\lambda\to+\infty}\frac{N_{D_{\Omega,\mathcal{D}}^{\alpha}}(\lambda)}{|\Omega|p^{n[\frac{1}{\alpha}\log_p\lambda]}}=1.
\end{align*}
\end{proof}
\begin{remark}
    If $\Omega$ is compact, then Weyl's law for the Dirichlet operator becomes an equality whenever $\lambda$ exceeds a constant depending only on $\Omega$.
\end{remark}

\begin{remark}
    The Jordan measurability is essential since fat Cantor sets occur naturally over $p$-adic fields. For instance, if we take a dense subset $\{\boldsymbol{x}_i\}_{i=1}^{\infty}\subseteq\mathbb{Z}_p^n$ (e.g. $\mathbb{Z}^n$) and let $\Omega=\displaystyle\bigcup_{i=1}^{\infty}(\boldsymbol{x}_i+p^{i+k}\mathbb{Z}_p^n)$ for a fixed $k\in\mathbb{N}^*$. Then $\Omega$ is open, $\overline{\Omega}=\mathbb{Z}_p^n$ and
    \begin{align*}
        |\Omega|\leq\sum_{i=1}^{\infty}p^{-n(i+k)}=\frac{p^{-nk}}{p^n-1}<1=|\mathbb{Z}_p^n|,
    \end{align*}
    so that $\mathbb{Z}_p^n\backslash\Omega$ is a fat Cantor set and any approximation of $\Omega$ from the exterior converges to $\mathbb{Z}_p^n$, rather than to $\Omega$.
\end{remark}

\begin{remark}
    Using the heat trace method, one can only obtain $N_{D_{\Omega,\mathcal{D}}^{\alpha}}(\lambda)=O(\lambda^{\frac{n}{\alpha}})$ \cite{MR3708861,MR3793137}. However, difficulties arise in the proof of the reverse direction, since the spectral zeta function has infinitely many poles $s_k=\frac{n}{\alpha}+i\frac{2\pi k}{\alpha\ln p}$ for $k\in\mathbb{Z}$ at $\Re(s)=\frac{n}{\alpha}$, which prevent us from using the classical Ikehara Tauberian theorem to obtain the asymptotic behavior of $N_{D_{\Omega,\mathcal{D}}^{\alpha}}(\lambda)$. In fact,
    \begin{align*}
        |\Omega|p^{n[\frac{1}{\alpha}\log_p\lambda]}=|\Omega|\lambda^{\frac{n}{\alpha}}p^{-n\{\frac{1}{\alpha}\log_p\lambda\}}.
    \end{align*}
    Unlike its Euclidean counterpart, the leading term of the Weyl asymptotics contains an oscillatory factor $p^{-n\{\frac{1}{\alpha}\log_p\lambda\}}$. This explains why the heat trace method does not directly yield the asymptotic behavior of the eigenvalue counting function. Consequently, the conjectured asymptotic formulas in \cite{MR3708861,MR3793137} do not hold.
\end{remark}

In order to handle the Neumann operator $D_{\Omega,\mathcal{N}}^{\alpha}$, we first introduce Weyl's Inequality for Bounded Perturbations:
\begin{lemma}[Weyl's Inequality for Bounded Perturbations]
\label{lem:Weyl_bounded_perturbation}
    Let $A$ and $B$ be self-adjoint operators on a Hilbert space with discrete spectra, and suppose they differ by a bounded operator $R$, i.e. $B=A+R$ with $\|R\|\leq\delta$. If the eigenvalues of both operators are listed in non-decreasing order and repeated according to multiplicity, then
    \begin{align*}
        |\lambda_j(A)-\lambda_j(B)|\leq\delta, \ \forall j.
    \end{align*}
    Equivalently, for the eigenvalue counting function
    \begin{align*}
        N(\lambda)=\#\{j|\ \lambda_j\leq\lambda\},
    \end{align*}
    the following sharp two-sided estimate holds:
    \begin{align*}
        N_A(\lambda-\delta)\leq N_B(\lambda)\leq N_A(\lambda+\delta).
    \end{align*}
\end{lemma}

\begin{theorem}
Let $\Omega\subseteq\mathbb{Q}_p^n$ be bounded, open, and Jordan
measurable. Suppose
\begin{align*}
    \Gamma_{\max}:=\sup_{i\geq1}\Gamma_i<+\infty,
\end{align*}
and define
\begin{align*}
    E=\bigcup_{j=1}^{\infty}
    [p^{\alpha j}-\Gamma_{\max},p^{\alpha j}).
\end{align*}
Then
\begin{align*}
\lim_{\substack{\lambda\to+\infty\\ \lambda\notin E}}
\frac{N_{D_{\Omega,\mathcal N}^{\alpha}}(\lambda)}
{|\Omega|p^{n[\frac1\alpha\log_p\lambda]}}=1,
\qquad
\lim_{R\to\infty}\frac{|E\cap[0,R]|}{R}=0.
\end{align*}
Thus the Neumann operator satisfies Weyl's law outside an exceptional set of
zero Lebesgue asymptotic density.
\end{theorem}
\begin{proof}
Let $u\in H^{\alpha/2}(\Omega)$ and extend it by zero to $\Omega^c$.
Splitting the full Gagliardo integral into $\Omega\times\Omega$ and the two
cross terms gives the following identity of extended nonnegative integrals:
\begin{align*}
\mathcal E(\widetilde u,\widetilde u)
=\mathcal E_{\Omega,\mathcal N}[u]
+\int_\Omega\Gamma(\boldsymbol x)|u(\boldsymbol x)|^2d\boldsymbol x
\leq\mathcal E_{\Omega,\mathcal N}[u]
+\Gamma_{\max}\|u\|_2^2.
\end{align*}
The right-hand side is finite, so the zero extension belongs to
$H^{\alpha/2}(\mathbb Q_p^n)$. Hence
zero extension identifies $H^{\alpha/2}(\Omega)$ with the class of functions
in $H^{\alpha/2}(\mathbb Q_p^n)$ that vanish almost everywhere on $\Omega^c$.
For this proof only, let $A$ be the self-adjoint operator
associated with the zero-extension form
$u\mapsto\mathcal E(\widetilde u,\widetilde u)$ on this space. The Fourier
cutoff argument in Section 2 shows that $A$ has compact resolvent. Moreover,
the same inner and outer compact-open approximations used in Theorem
\ref{thm:Weyl_law_Dirichlet_operator}, now using the inclusions defined by
almost-everywhere support, give
\begin{align}
\label{eq:auxiliary_Weyl_law}
N_A(\lambda)\sim
|\Omega|p^{n[\frac1\alpha\log_p\lambda]}.
\end{align}
The Neumann form and the form of $A$ have the same domain and differ by the
bounded multiplication operator $M_\Gamma$. Therefore the Neumann operator
also has compact resolvent, and Lemma
\ref{lem:Weyl_bounded_perturbation} gives
\begin{align*}
N_A(\lambda)\leq
N_{D_{\Omega,\mathcal N}^{\alpha}}(\lambda)
\leq N_A(\lambda+\Gamma_{\max}).
\end{align*}
The lower bound may alternatively be strengthened by the strict Dirichlet
count because
$H_0^{\alpha/2}(\Omega)\subseteq H^{\alpha/2}(\Omega)$ and
$\mathcal E_{\Omega,\mathcal N}\leq
\mathcal E_{\Omega,\mathcal D}$ on the smaller domain.

If $\lambda\notin E$, then
\begin{align*}
p^{n[\frac1\alpha\log_p(\lambda+\Gamma_{\max})]}
=p^{n[\frac1\alpha\log_p\lambda]}.
\end{align*}
Combining the preceding bounds with \eqref{eq:auxiliary_Weyl_law} proves the
asserted limit. Finally, the number of intervals in $E$ meeting $[0,R]$ is
$O(\log R)$, and each has length at most $\Gamma_{\max}$. Thus
$|E\cap[0,R]|/R\to0$.
\end{proof}

\begin{remark}
    The classical power-law Weyl asymptotic fails even when $\Omega$ is
    compact. Indeed, as $\lambda$ crosses a threshold $p^{\alpha j}$,
    $[\alpha^{-1}\log_p\lambda]$ changes by one. Equivalently, the
    oscillatory factor $p^{-n\{\alpha^{-1}\log_p\lambda\}}$ jumps by a factor of $p^n$.
\end{remark}

\subsection{Why Haran's Theorem does not apply}
In this subsection, we explain why Haran’s theorem cannot be used to derive Weyl’s law for the Vladimirov-Taibleson operator considered here \cite{MR1252936}. This distinction explains why the result proved above is not an immediate consequence of Haran’s phase-space theorem.

Let $\mathbb{Q}_p^n$ be the $n$-dimensional $p$-adic physical space. Its
corresponding phase space is $\mathbb{Q}_p^n\times\mathbb{Q}_p^n$, endowed
with coordinates $\boldsymbol{z}=(\boldsymbol{x},\boldsymbol{\xi})$.

\begin{definition}[Metric covering]
    A metric covering is a family
    $g=(g_{\boldsymbol{z}})_{\boldsymbol{z}\in
    \mathbb Q_p^n\times\mathbb Q_p^n}$ of metrics on phase-space increments.
    It is associated with a decomposition
    \(\mathbb Q_p^n\times\mathbb Q_p^n
    =\bigsqcup\limits_j(\boldsymbol z_j+B_j)\), where each \(B_j\) is a
    lattice. Let $g_{\boldsymbol{z}}^\sigma$ denote its symplectic dual and
    define
    $h_g(\boldsymbol{z})=\sup_{\boldsymbol w\neq\boldsymbol0}
    \frac{g_{\boldsymbol z}(\boldsymbol w)}
    {g_{\boldsymbol z}^\sigma(\boldsymbol w)}$. A metric covering that is
    certain and temperate is called global if there exist $C$, $\delta>0$
    such that
    $h_g(\boldsymbol{z})\leq
    C\max\{1,\|\boldsymbol{z}\|_p\}^{-\delta}$.
\end{definition}

\begin{definition}[Symbol Class $S(m,g)$]
\label{def:symbol_class}
Let $m:\mathbb Q_p^n\times\mathbb Q_p^n\to\mathbb{R}_{>0}$ be a
\emph{$g$-locally constant weight function}, meaning that $m$ is constant on
each cell $\boldsymbol{z}_j+B_j$ of the covering. The \emph{symbol class}
$S(m,g)$ is defined as the Fréchet space of all functions
$f:\mathbb Q_p^n\times\mathbb Q_p^n\to\mathbb{C}$ satisfying the following
two conditions:
\begin{enumerate}
\item \textbf{Amplitude estimate:}
There exists a constant $C > 0$ such that for all
$\boldsymbol{z}\in\mathbb Q_p^n\times\mathbb Q_p^n$,
\[
|f(\boldsymbol{z})| \le C m(\boldsymbol{z}).
\]

\item \textbf{Smoothness estimate:}
For every $s>0$, there exists a constant $C_s>0$ such that for all
$\boldsymbol{z},\boldsymbol{w}\in
\mathbb Q_p^n\times\mathbb Q_p^n$ with
$g_{\boldsymbol{z}}(\boldsymbol{w})\leq1$,
\[
|f(\boldsymbol{z}) - f(\boldsymbol{z}+\boldsymbol{w})|
\leq C_s m(\boldsymbol{z})g_{\boldsymbol{z}}(\boldsymbol{w})^s.
\]
\end{enumerate}

Equivalently, the topology on $S(m,g)$ is defined by the family of seminorms
\[
\| f \|_{B_{\infty,\infty}^{s}(m,g)}
:= \sup_{\boldsymbol{z}\in\mathbb Q_p^n\times\mathbb Q_p^n}
\frac{|f(\boldsymbol{z})|}{m(\boldsymbol{z})}
+ \sup_{\substack{\boldsymbol{z},\boldsymbol{w}\in
\mathbb Q_p^n\times\mathbb Q_p^n\\
g_{\boldsymbol{z}}(\boldsymbol{w})\leq1}}
\frac{|f(\boldsymbol{z})-f(\boldsymbol{z}+\boldsymbol{w})|}
{m(\boldsymbol{z})g_{\boldsymbol{z}}(\boldsymbol{w})^s},
\qquad s>0,
\]
so that
\[
S(m,g)=\Bigl\{f:\mathbb Q_p^n\times\mathbb Q_p^n\to\mathbb C
\;\Big|\;
\|f\|_{B_{\infty,\infty}^{s}(m,g)}<\infty
\ \text{for all }s>0\Bigr\}.
\]
\end{definition}

\begin{theorem}[Weyl Asymptotics for $p$-adic pseudodifferential operators \cite{MR1252936}]
\label{thm:Weyl_Asymptotics_Pseudodifferential_Operator}

Let $g$ be a global, certain, and temperate metric covering of
$\mathbb Q_p^n\times\mathbb Q_p^n$, and let $m$ be a temperate weight satisfying
\begin{align*}
    m(\boldsymbol{z})\geq C\max\{1,\|\boldsymbol{z}\|_p\}^{\alpha}
\end{align*}
for some $C$, $\alpha>0$.

Consider a pseudodifferential operator $\hat{\rho}(f)$ with a positive elliptic symbol $f(\boldsymbol{z})\in S(m, g)$, meaning that there exists a constant $C > 0$ such that $f(\boldsymbol{z}) \ge C m(\boldsymbol{z})$ outside a compact subset of $\mathbb Q_p^n\times\mathbb Q_p^n$. Then the pseudodifferential operator $\hat{\rho}(f)$ admits a unique self-adjoint extension $\mathcal{A}$. Under these assumptions, the spectrum is purely discrete and consists of eigenvalues $\{\lambda_j\}$ satisfying $\lambda_j \to +\infty$. In the strict-counting convention used in Haran's theorem, set $N_{\mathcal{A}}^-(\lambda) = \#\{\lambda_j < \lambda\}$. The following rough asymptotic upper bound holds for any $\epsilon > 0$ as $\lambda \to +\infty$:
\[
N_{\mathcal{A}}^-(\lambda) = O(\lambda^{2n/\alpha + \epsilon}).
\]
Furthermore, the eigenvalue counting function satisfies the phase-space volume asymptotic formula:
\[
N_{\mathcal{A}}^-(\lambda) = \int_{f(\boldsymbol{z}) < \lambda} d\boldsymbol{z} + O(\lambda^\epsilon), \quad \text{as } \lambda \to +\infty.
\]
\end{theorem}

Theorem \ref{thm:Weyl_Asymptotics_Pseudodifferential_Operator} provides Weyl asymptotics for a broad class of $p$-adic pseudodifferential operators on $\mathbb{Q}_p^n$. The usual infinite-well heuristic would attach to the Dirichlet problem the expression $D^{\alpha}+V_{\Omega}(\boldsymbol{x})$, where
\begin{align*}
    V_{\Omega}(\boldsymbol{x})=\begin{cases}
        0,&\boldsymbol{x}\in\Omega,\\
        +\infty,&\boldsymbol{x}\in\Omega^c
    \end{cases}
\end{align*}
is the infinite potential well. For a general irregular domain this is only
a formal mnemonic: the Friedrichs realization used here has form domain
$H_0^{\alpha/2}(\Omega)$, fixed by closure of $\mathcal S(\Omega)$, whereas
an almost-everywhere support condition alone may give a larger space. The
corresponding formal expression is
\begin{align*}
\varphi(\boldsymbol{x},\boldsymbol{\xi})=\|\boldsymbol{\xi}\|_p^{\alpha}+V_{\Omega}(\boldsymbol{x}).
\end{align*}
However, this is an extended-valued expression rather than a symbol in
\(S(m,g)\), and hence it does not satisfy the hypotheses of Theorem
\ref{thm:Weyl_Asymptotics_Pseudodifferential_Operator}. A purely formal
phase-space calculation gives only
\begin{align*}
\operatorname{vol}\{(\boldsymbol x,\boldsymbol\xi):
\boldsymbol x\in\Omega,\ \|\boldsymbol\xi\|_p^\alpha\leq\lambda\}
=|\Omega|p^{n[\frac1\alpha\log_p\lambda]}.
\end{align*}
Because the symbol hypotheses fail, this calculation by itself gives no
asymptotic formula, and in particular no $O(\lambda^\epsilon)$ remainder, for $N_{D_{\Omega,\mathcal D}^{\alpha}}(\lambda)$.

\section{The Weyl-Berry conjecture}
\subsection{Remainder estimate}
\begin{definition}[$s$-dimensional inner Minkowski content]
    Let $\Omega\subseteq\mathbb{Q}_p^n$ be an open set with its topological boundary $\partial\Omega$. The inner $\varepsilon$-parallel neighborhood of $\partial\Omega$ is defined as
    \begin{align*}
        \partial\Omega_{\varepsilon,\text{in}}=\left\{\boldsymbol{x}\in\Omega|\ \inf_{\boldsymbol{y}\in\partial\Omega}\|\boldsymbol{x}-\boldsymbol{y}\|_p\leq\varepsilon\right\}.
    \end{align*}
    The upper and lower s-dimensional inner Minkowski content of $\partial\Omega$ are defined respectively by
    \begin{align*}
        \overline{\mathcal{M}}_{\text{in}}^{s}(\partial\Omega)=\varlimsup_{\varepsilon\to0}\frac{\mathcal{H}^n(\partial\Omega_{\varepsilon,\text{in}})}{p^{[\log_p\varepsilon](n-s)}};\\
        \underline{\mathcal{M}}_{\text{in}}^{s}(\partial\Omega)=\varliminf_{\varepsilon\to0}\frac{\mathcal{H}^n(\partial\Omega_{\varepsilon,\text{in}})}{p^{[\log_p\varepsilon](n-s)}}.
    \end{align*}
    If $\overline{\mathcal{M}}_{\text{in}}^{s}(\partial\Omega)=\underline{\mathcal{M}}_{\text{in}}^{s}(\partial\Omega)$, the common value is called the $s$-dimensional inner Minkowski content of $\partial\Omega$, denoted by $\mathcal{M}_{\text{in}}^{s}(\partial\Omega)$.
\end{definition}
\begin{remark}
    The $s$-dimensional Minkowski content $\mathcal{M}^{s}(\partial\Omega)$ is defined by
    \begin{align*}
        \mathcal{M}^{s}(\partial\Omega)=\lim_{\varepsilon\to0}\frac{\mathcal{H}^n(\partial\Omega_{\varepsilon})}{p^{[\log_p\varepsilon](n-s)}},
    \end{align*}
    where
    \begin{align*}
        \partial\Omega_{\varepsilon}=\left\{\boldsymbol{x}\in\mathbb{Q}_p^n|\ \inf_{\boldsymbol{y}\in\partial\Omega}\|\boldsymbol{x}-\boldsymbol{y}\|_p\leq\varepsilon\right\}.
    \end{align*}
    The exterior content $\mathcal{M}_{\text{ex}}^{s}(\partial\Omega)$ is defined analogously, with $\Omega$ replaced by $\Omega^c$ in the parallel neighborhood.
    When the following limit exists, the Minkowski dimension of
    $\partial\Omega$ is defined by
    \begin{align*}
        \dim_M(\partial\Omega)
        =n-\lim_{K\to\infty}
        \frac{\log_p|\partial\Omega_{p^{-K}}|}{-K}.
    \end{align*}
    Without assuming that the limit exists, replacing it by the lower
    limit gives the upper Minkowski dimension, whereas replacing it by
    the upper limit gives the lower Minkowski dimension. When these two
    quantities agree, their common value is denoted by
    $\dim_M(\partial\Omega)$.
\end{remark}

Write $\Omega=\displaystyle\bigsqcup_{i=1}^{\infty}B_{L_i}(\boldsymbol{a}_i)$ as the disjoint union of its maximal balls $\{B_{L_i}(\boldsymbol{a}_i)\}_{i=1}^{\infty}$, and define
\begin{align*}
    &V=\overline{\operatorname{span}\{\boldsymbol{1}_{B_{L_i}(\boldsymbol{a}_i)}\}_{i=1}^{\infty}}^{L^2(\Omega)};\\
    &V_{=i}=\operatorname{span}\{\boldsymbol{1}_{B_{L_j}(\boldsymbol{a}_j)}|\ L_j=i\};\\
    &n_i=\dim V_{=i};\\
    &V_{\leq K}=\bigoplus_{i\leq K} V_{=i};\\
    &V_{>K}=\overline{\bigoplus_{i\geq K+1} V_{=i}}^{L^2(\Omega)}=V\ominus V_{\leq K};\\
    &B_{\leq K}=\{B_{L_i}(\boldsymbol{a}_i)|\ \boldsymbol{1}_{B_{L_i}(\boldsymbol{a}_i)}\in V_{\leq K}\};\\
    &B_{>K}=\{B_{L_i}(\boldsymbol{a}_i)|\ \boldsymbol{1}_{B_{L_i}(\boldsymbol{a}_i)}\in V_{>K}\};\\
    &\Omega_{\leq K}=\bigcup_{L_i\leq K}B_{L_i}(\boldsymbol{a}_i);\\
    &\Omega_{>K}=\bigcup_{L_i\geq K+1}B_{L_i}(\boldsymbol{a}_i).
\end{align*}
With this notation, $V=V_{\leq K}\oplus V_{>K}$ and $\Omega=\Omega_{\leq K}\sqcup\Omega_{>K}$.

\begin{lemma}
\label{lem:dist_maximalball_partialOmega}
    For each maximal ball $B_{L_i}(\boldsymbol{a}_i)$, we have $\operatorname{dist}(B_{L_i}(\boldsymbol{a}_i),\Omega^c)=p^{-L_i+1}$.
\end{lemma}
\begin{proof}
    Because $B_{L_i}(\boldsymbol{a}_i)\subseteq\Omega$, every $\boldsymbol{x}\in B_{L_i}(\boldsymbol{a}_i)$ and $\boldsymbol{y}\in\Omega^c$ satisfy
    \begin{align*}
        \|\boldsymbol{x}-\boldsymbol{y}\|_p>p^{-L_i}.
    \end{align*}
    Since the $p$-adic norm takes values in powers of $p$, this implies
    \begin{align*}
        \|\boldsymbol{x}-\boldsymbol{y}\|_p\geq p^{-L_i+1}.
    \end{align*}
    Consequently,
    \begin{align*}
        \operatorname{dist}(B_{L_i}(\boldsymbol{a}_i),\Omega^c)=\inf_{\substack{\boldsymbol{x}\in B_{L_i}(\boldsymbol{a}_i)\\\boldsymbol{y}\in\Omega^c}}\|\boldsymbol{x}-\boldsymbol{y}\|_p\geq p^{-L_i+1}.
    \end{align*}
    By maximality of $B_{L_i}(\boldsymbol{a}_i)\in\Omega$, the larger ball satisfies $B_{L_i-1}(\boldsymbol{a}_i)\cap\Omega^c\neq\emptyset$. Take $\boldsymbol{z}\in B_{L_i-1}(\boldsymbol{a}_i)\cap\Omega^c$. Then
    \begin{align*}
        \|\boldsymbol{x}-\boldsymbol{z}\|_p\leq\max\{\|\boldsymbol{x}-\boldsymbol{a}_i\|_p,\|\boldsymbol{a}_i-\boldsymbol{z}\|_p\}\leq\max\{p^{-L_i},p^{-L_i+1}\}=p^{-L_i+1}.
    \end{align*}
    Suppose, for contradiction, that there exists
    $\boldsymbol{x}\in B_{L_i}(\boldsymbol{a}_i)$ such that
    $\|\boldsymbol{x}-\boldsymbol{z}\|_p\leq p^{-L_i}$. Then
    \begin{align*}
        \|\boldsymbol{a}_i-\boldsymbol{z}\|_p\leq\max\{\|\boldsymbol{a}_i-\boldsymbol{x}\|_p,\|\boldsymbol{x}-\boldsymbol{z}\|_p\}\leq\max\{p^{-L_i},p^{-L_i}\}=p^{-L_i},
    \end{align*}
    which means $\boldsymbol{z}\in B_{L_i}(\boldsymbol{a}_i)\subseteq\Omega$. This contradicts $\boldsymbol{z}\in\Omega^c$. Therefore, for every $\boldsymbol{x}\in B_{L_i}(\boldsymbol{a}_i)$, one has
    \begin{align*}
         \|\boldsymbol{x}-\boldsymbol{z}\|_p=p^{-L_i+1},
    \end{align*}
    Consequently,
    \begin{align*}
        \operatorname{dist}(B_{L_i}(\boldsymbol{a}_i),\Omega^c)\leq\operatorname{dist}(B_{L_i}(\boldsymbol{a}_i),\boldsymbol{z})=p^{-L_i+1}.
    \end{align*}
    Combining the two inequalities yields $\operatorname{dist}(B_{L_i}(\boldsymbol{a}_i),\Omega^c)=p^{-L_i+1}$.
\end{proof}

Define orthogonal projection $P_K$ by
\begin{align}
\label{eq:P_Ku_Fourier}
    \widehat{P_Ku}(\boldsymbol{\xi})=\widehat{u}(\boldsymbol{\xi})\boldsymbol{1}_{p^{-K}\mathbb{Z}_p^n}(\boldsymbol{\xi}).
\end{align}
Equivalently, Fourier inversion gives
\begin{align}
\label{eq:P_Ku}
    P_Ku(\boldsymbol{x})=p^{nK}\int_{B_K(\boldsymbol{x})}u(\boldsymbol{y})d\boldsymbol{y},
\end{align}
because
\begin{align*}
    \check{\boldsymbol{1}}_{p^{-K}\mathbb{Z}_p^n}(\boldsymbol{x})=p^{nK}\boldsymbol{1}_{p^{K}\mathbb{Z}_p^n}(\boldsymbol{x}).
\end{align*}
Denote $\mathcal{B}_K$ and $\mathcal{U}_K$ by
\begin{align*}
    &\mathcal{B}_K=\{B_K(\boldsymbol{a}_i)| B_{L_i}(\boldsymbol{a}_i)\subseteq B_K(\boldsymbol{a}_i),\ B_{L_i}(\boldsymbol{a}_i)\in B_{>K}\};\\
    &\mathcal{U}_K=\bigcup\{B_K(\boldsymbol{a}_i)|\ B_K(\boldsymbol{a}_i)\in\mathcal{B}_K\}.
\end{align*}

\begin{lemma}
\label{lem:P_K_prpoerty>K}
    One has
    \begin{align*}
    P_KV_{>K}=\operatorname{span}
    \{\boldsymbol 1_{B_K}:B_K\in\mathcal{B}_K\}.
    \end{align*}
\end{lemma}
\begin{proof}
    It suffices to consider a basis function $\boldsymbol{1}_{B_{L_i}(\boldsymbol{a}_i)}(\boldsymbol{x})\in V_{>K}$. Since $L_i\geq K+1$, \eqref{eq:P_Ku} gives
    \begin{align}
    \label{eq:P_K_1}
        P_K\boldsymbol{1}_{B_{L_i}(\boldsymbol{a}_i)}(\boldsymbol{x})=p^{nK}\int_{B_K(\boldsymbol{x})}\boldsymbol{1}_{B_{L_i}(\boldsymbol{a}_i)}(\boldsymbol{y})d\boldsymbol{y}.
    \end{align}
    \begin{enumerate}
        \item[]\textbf{Case 1.} If $B_K(\boldsymbol{x})\cap B_{L_i}(\boldsymbol{a}_i)=\emptyset$, then the integral in \eqref{eq:P_K_1} vanishes.
        \item[]\textbf{Case 2.} If $B_K(\boldsymbol{x})\cap B_{L_i}(\boldsymbol{a}_i)\neq\emptyset$, then $B_{L_i}(\boldsymbol{a}_i)\subseteq B_K(\boldsymbol{x})$. Therefore, the integral in \eqref{eq:P_K_1} equals $p^{nK}|B_{L_i}(\boldsymbol{a}_i)|=p^{n(K-L_i)}$.
    \end{enumerate}
    Thus
    \begin{align*}
        P_K\boldsymbol{1}_{B_{L_i}(\boldsymbol{a}_i)}(\boldsymbol{x})=p^{n(K-L_i)}\boldsymbol{1}_{B_K(\boldsymbol{a}_i)}(\boldsymbol{x}).
    \end{align*}
    For fixed $K$, the boundedness of $\Omega$ implies that only finitely many
    balls of level $K$ occur in $\mathcal{B}_K$. Hence the displayed span is
    closed, and taking the span proves the assertion. The basis indicators
    have pairwise disjoint supports, and the union of those supports is
    precisely $\mathcal{U}_K$.
\end{proof}

\begin{lemma}
\label{lem:P_K_prpoerty<=K}
    $P_K|_{V_{\leq K}}=\operatorname{id}_{V_{\leq K}}$ and $\mathcal{U}_K\cap\Omega_{\leq K}=\emptyset$.
\end{lemma}
\begin{proof}
It suffices to verify the assertion for each basis function $\boldsymbol{1}_{B_{L_i}(\boldsymbol{a}_i)}(\boldsymbol{x})$. Since
\begin{align*}
    \hat{\boldsymbol{1}}_{B_{L_i}(\boldsymbol{a}_i)}(\boldsymbol{\xi})=\chi_p(\boldsymbol{\xi}\cdot\boldsymbol{a}_i)p^{-nL_i}\boldsymbol{1}_{B_{-L_i}(\boldsymbol{0})}(\boldsymbol{\xi}).
\end{align*}
Here every $u\in V_{\leq K}$ satisfies $\operatorname{supp}\hat{u}\subseteq\{\boldsymbol{\xi}\in\mathbb{Q}_p^n|\ \|\boldsymbol{\xi}\|_p\leq p^K\}$. Then by \eqref{eq:P_Ku_Fourier}, for $u\in V_{\leq K}$, we have
\begin{align*}
    \widehat{P_Ku}(\boldsymbol{\xi})=\widehat{u}(\boldsymbol{\xi})\boldsymbol{1}_{p^{-K}\mathbb{Z}_p^n}(\boldsymbol{\xi})=\widehat{u}(\boldsymbol{\xi}),
\end{align*}
thus
\begin{align*}
    P_Ku(\boldsymbol{x})=u(\boldsymbol{x}),
\end{align*}
hence $P_K|_{V_{\leq K}}=\operatorname{id}_{V_{\leq K}}$.

Now let $u\in V_{>K}$. Then
\begin{align*}
   u|_{\Omega_{\leq K}}=0.
\end{align*}
Thus, for every $\boldsymbol{x}\in\Omega_{\leq K}$, there exists a maximal ball $B_{L_j}(\boldsymbol{a}_j)$ containing $\boldsymbol{x}$ with $L_j\leq K$. Then \eqref{eq:P_Ku} gives
\begin{align*}
    B_K(\boldsymbol{x})\subseteq B_{L_j}(\boldsymbol{a}_j)\subseteq\Omega_{\leq K}.
\end{align*}
Because $u|_{\Omega_{\leq K}}=0$, formula \eqref{eq:P_Ku} gives
$P_Ku|_{\Omega_{\leq K}}=0$. The preceding lemma and the definition of
$\mathcal{U}_K$ therefore imply
$\mathcal{U}_K\cap\Omega_{\leq K}=\emptyset$.
\end{proof}

\begin{lemma}
    \begin{align*}
        &\mathcal{B}_K=\{B_K(\boldsymbol{a})|\ B_K(\boldsymbol{a})\cap\Omega\neq\emptyset,\ B_K(\boldsymbol{a})\cap\Omega^c\neq\emptyset\};\\
        &\mathcal{U}_K=\{\boldsymbol{x}\in\mathbb{Q}_p^n|\ \operatorname{dist}(\boldsymbol{x},\Omega)\leq p^{-K},\ \operatorname{dist}(\boldsymbol{x},\Omega^c)\leq p^{-K}\}.
    \end{align*}
    Moreover, $\partial\Omega_{p^{-K}}\subseteq\mathcal{U}_K$ and $\displaystyle\bigcap_{K}\mathcal{U}_K=\partial\Omega$.
\end{lemma}
\begin{proof}
    Let $B_K(\boldsymbol{a})\in\mathcal{B}_K$. By definition, there exists a maximal ball $B_{L}(\boldsymbol{a})\in B_{>K}$ such that $B_{L}(\boldsymbol{a})\subseteq B_K(\boldsymbol{a})$. Consequently, $B_K(\boldsymbol{a})\cap\Omega\neq\emptyset$. If $B_K(\boldsymbol{a})\subseteq\Omega$, then the maximal ball $B_{L}(\boldsymbol{a})$ contained in it would not be maximal in $\Omega$, a contradiction. Hence $B_K(\boldsymbol{a})\cap\Omega^c\neq\emptyset$.

    Conversely, let $B_K(\boldsymbol{a})\in\{B_K(\boldsymbol{a})|\ B_K(\boldsymbol{a})\cap\Omega\neq\emptyset,\ B_K(\boldsymbol{a})\cap\Omega^c\neq\emptyset\}$, and choose $\boldsymbol{x}\in B_K(\boldsymbol{a})\cap\Omega$. Let $B_L(\boldsymbol{b})$ be the maximal ball containing $\boldsymbol{x}$. If $L\leq K$, then $B_K(\boldsymbol{a})=B_K(\boldsymbol{x})\subseteq B_L(\boldsymbol{x})=B_L(\boldsymbol{b})\subseteq\Omega$. This contradicts $B_K(\boldsymbol{a})\cap\Omega^c\neq\emptyset$. Therefore $L\geq K+1$, so $B_L(\boldsymbol{b})\subseteq B_K(\boldsymbol{a})$ and $B_K(\boldsymbol{a})\in\mathcal{B}_K$.

    This proves the stated characterization of \(\mathcal{B}_K\). Taking the union of these balls gives the characterization of \(\mathcal{U}_K\). Finally, the definition of \(\partial\Omega\) yields $\partial\Omega_{p^{-K}}\subseteq\mathcal{U}_K$ and $\displaystyle\bigcap_{K}\mathcal{U}_K=\partial\Omega$.
\end{proof}

\begin{corollary}
    $\Omega_{>K}=\Omega\cap\mathcal{U}_K=\{\boldsymbol{x}\in\Omega|\ \operatorname{dist}(\boldsymbol{x},\Omega^c)\leq p^{-K}\}$. Moreover, $\partial\Omega_{p^{-K},\text{in}}\subseteq\Omega_{>K}$.
\end{corollary}

\begin{definition}
We say the condition \eqref{eq:condition_AK} holds if
    \begin{align}
    \label{eq:condition_AK}
    \tag{$A_K$}
        B_K(\boldsymbol{a})\in\mathcal{B}_K\Rightarrow B_K(\boldsymbol{a})\cap\partial\Omega\neq\emptyset.
    \end{align}
\end{definition}
\begin{remark}
    In the Archimedean setting this condition is automatic because every connected ball meeting both \(\Omega\) and \(\Omega^c\) meets \(\partial\Omega\); in the non-Archimedean setting it can fail since each ball in $\mathbb{Q}_p^n$ is totally disconnected. For example, let $\Omega=4\mathbb{Z}_2\sqcup((1+2\mathbb{Z}_2)\setminus\{1\})$. Then $\partial\Omega=\{1\}$. Consider the ball $2\mathbb{Z}_2=4\mathbb{Z}_2\sqcup(2+4\mathbb{Z}_2)$: we have $4\mathbb{Z}_2\subseteq\Omega$ and $2+4\mathbb{Z}_2\subseteq\Omega^c$, but $2\mathbb{Z}_2\cap\{1\}=\emptyset$.
\end{remark}

\begin{lemma}
\label{lem:equiv_condition_AK}
    The following are equivalent:
    \begin{enumerate}
        \item The condition \eqref{eq:condition_AK} holds.
        \item $B_K(\boldsymbol{a})\in\mathcal{B}_K\Leftrightarrow B_K(\boldsymbol{a})\cap\partial\Omega\neq\emptyset$.
        \item $\mathcal{U}_K=\partial\Omega_{p^{-K}}$.
        \item $\Omega_{>K}=\partial\Omega_{p^{-K},\text{in}}$.
    \end{enumerate}
\end{lemma}
\begin{proof}
    \begin{enumerate}
        \item[] $1\Rightarrow2$: It remains to prove the reverse implication. Suppose that $B_K(\boldsymbol{a})\cap\partial\Omega\neq\emptyset$, and choose $\boldsymbol{z}\in B_K(\boldsymbol{a})\cap\partial\Omega$. Then $B_K(\boldsymbol{a})$ is an open neighborhood of $\boldsymbol{z}\in\partial\Omega$. Consequently, we have $B_K(\boldsymbol{a})\cap\Omega\neq\emptyset$ and $
            B_K(\boldsymbol{a})\cap\Omega^c\neq\emptyset$, which means $B_K(\boldsymbol{a})\in\mathcal{B}_K$.
        \item[] $2\Rightarrow3$: If $\boldsymbol{x}\in\partial\Omega_{p^{-K}}$, then $B_K(\boldsymbol{x})\cap\partial\Omega\neq\emptyset$. Conversely, suppose that $B_K(\boldsymbol{x})\cap\partial\Omega\neq\emptyset$. Choose $\boldsymbol{y}\in B_K(\boldsymbol{x})\cap\partial\Omega$. Then $\|\boldsymbol{x}-\boldsymbol{y}\|_p\leq p^{-K}$. Hence, $\operatorname{dist}(\boldsymbol{x},\partial\Omega)\leq\|\boldsymbol{x}-\boldsymbol{y}\|_p\leq p^{-K}$, which means $\boldsymbol{x}\in\partial\Omega_{p^{-K}}$. Therefore,
        \begin{align*}
            \partial\Omega_{p^{-K}}=\bigcup_{B_K(\boldsymbol{a})\cap\partial\Omega\neq\emptyset}B_K(\boldsymbol{a})=\bigcup_{B_K(\boldsymbol{a})\in\mathcal{B}_K}B_K(\boldsymbol{a})=\mathcal{U}_K.
        \end{align*}
        \item[] $3\Rightarrow4$: Since $\Omega_{>K}=\Omega\cap\mathcal{U}_K$ and $\partial\Omega_{p^{-K},\text{in}}=\Omega\cap\partial\Omega_{p^{-K}}$, it follows that $\Omega_{>K}=\partial\Omega_{p^{-K},\text{in}}$.
        \item[] $4\Rightarrow1$: For $B_K(\boldsymbol{a})\in\mathcal{B}_K$, take $\boldsymbol{x}\in B_K(\boldsymbol{a})\cap\Omega\subseteq\Omega\cap\mathcal{U}_K=\Omega_{>K}=\partial\Omega_{p^{-K},\text{in}}$. Since $\partial\Omega$ is compact, there exists $\boldsymbol{y}\in\partial\Omega$ such that $\|\boldsymbol{x}-\boldsymbol{y}\|_p\leq p^{-K}$. Hence $\boldsymbol{y}\in B_K(\boldsymbol{x})=B_K(\boldsymbol{a})$, which means $B_K(\boldsymbol{a})\cap\partial\Omega\neq\emptyset$.
    \end{enumerate}
\end{proof}

\begin{theorem}
\label{thm:remainder_estimate}
    Let $\Omega\subseteq\mathbb{Q}_p^n$ be a bounded, open, Jordan measurable set. For $\lambda\in[p^{\alpha K},p^{\alpha(K+1)})$, we have the remainder estimate
    \begin{align*}
        -|\mathcal{U}_K\cap\Omega^c|p^{nK}\leq|\Omega|p^{n[\frac{1}{\alpha}\log_p\lambda]}-N_{D_{\Omega,\mathcal{D}}^{\alpha}}(\lambda)\leq|\mathcal{U}_K\cap\Omega|p^{nK}.
    \end{align*}
\end{theorem}
\begin{proof}
    We use the decomposition $N_{D_{\Omega,\mathcal{D}}^{\alpha}}(\lambda)=N_{D_{\Omega,\mathcal{D}}^{\alpha}}^{0}(\lambda)+N_{\boldsymbol{H}}(\lambda)$, where
\begin{align*}
    N_{\boldsymbol{H}}(\lambda)=\#\{\lambda_i\in\sigma(\boldsymbol{H}):\ \lambda_i\leq\lambda\},
\end{align*}
where eigenvalues are counted with multiplicity.
For $\lambda\in[p^{\alpha K},p^{\alpha(K+1)})$ with $K\in\mathbb{Z}$, one has $[\frac{1}{\alpha}\log_p\lambda]=K$, and by \eqref{eq:N0_D_Ball_lambda} and \eqref{eq:N0_D_Omiga_lambda}, we have
\begin{equation}
\label{eq:remainder_all}
\begin{aligned}
    |\Omega|p^{n[\frac{1}{\alpha}\log_p\lambda]}-N_{D_{\Omega,\mathcal{D}}^{\alpha}}(\lambda)&=\sum_{i\in\mathbb{Z}}n_ip^{-ni}p^{nK}-\sum_{i\leq K}n_i(p^{-ni}p^{nK}-1)-N_{\boldsymbol{H}}(\lambda)\\
    &=\sum_{i\geq K+1}n_ip^{-ni}p^{nK}+\sum_{i\leq K}n_i-N_{\boldsymbol{H}}(\lambda)\\
    &=|\Omega_{>K}|p^{nK}+\sum_{i\leq K}n_i-N_{\boldsymbol{H}}(\lambda).
\end{aligned}
\end{equation}
Because \(\Omega\) is bounded, there exists \(i_0\in\mathbb Z\) such that every maximal ball occurring in the decomposition has level \(i\ge i_0\). Consequently, $\sum\limits_{i\leq K}n_i$ is a finite sum. We next estimate $N_{\boldsymbol{H}}(\lambda)$. For every $u\in V_{\leq K}$, $\operatorname{supp}\hat{u}\subseteq\{\boldsymbol{\xi}\in\mathbb{Q}_p^n|\ \|\boldsymbol{\xi}\|_p\leq p^K\}$. Hence
\begin{align*}
    \mathcal E_{\Omega,\mathcal D}[u]&=(D^{\alpha}\tilde{u},\tilde{u})\\
    &=\int_{\mathbb{Q}_p^n}\|\boldsymbol{\xi}\|_p^{\alpha}|\hat{u}(\boldsymbol{\xi})|^2d\boldsymbol{\xi}\\
    &=\int_{\|\boldsymbol{\xi}\|_p\leq p^K}\|\boldsymbol{\xi}\|_p^{\alpha}|\hat{u}(\boldsymbol{\xi})|^2d\boldsymbol{\xi}\\
    &\leq p^{\alpha K}\int_{\|\boldsymbol{\xi}\|_p\leq p^K}|\hat{u}(\boldsymbol{\xi})|^2d\boldsymbol{\xi}\\
    &=p^{\alpha K}(u,u)\\
    &\leq \lambda(u,u).
\end{align*}
For $j\leq\dim V_{\leq K}$, the min--max principle (Lemma \ref{lem:Min-Max_Principle}) gives
\begin{align*}
    \lambda_j(\boldsymbol{H})=\min_{\substack{S \subseteq V\cap H_0^{\alpha/2}(\Omega) \\ \dim S = j}}
    \max_{\substack{0\neq u \in S}}
    \frac{\mathcal E_{\Omega,\mathcal D}[u]}{(u,u)}\leq\min_{\substack{S \subseteq V_{\leq K} \\ \dim S = j}}
    \max_{\substack{0\neq u \in S}}\frac{\mathcal E_{\Omega,\mathcal D}[u]}{(u,u)}\leq\lambda.
\end{align*}
It follows that
\begin{align*}
    N_{\boldsymbol{H}}(\lambda)\geq \dim V_{\leq K}=\sum_{i\leq K}n_i.
\end{align*}
On the other hand, let $u\in V\cap H_0^{\alpha/2}(\Omega)$ satisfy
$P_Ku=0$. Then $\widehat{u}(\boldsymbol{\xi})=0$ on
$\{\boldsymbol{\xi}\in\mathbb{Q}_p^n:\ \|\boldsymbol{\xi}\|_p\leq p^K\}$, and
\begin{align*}
    \mathcal E_{\Omega,\mathcal D}[u]=\int_{\|\boldsymbol{\xi}\|_p>p^K}\|\boldsymbol{\xi}\|_p^{\alpha}|\hat{u}(\boldsymbol{\xi})|^2d\boldsymbol{\xi}\geq p^{\alpha(K+1)}(u,u)>\lambda(u,u).
\end{align*}
Now let $S\subseteq V\cap H_0^{\alpha/2}(\Omega)$ be a subspace such that
$\mathcal E_{\Omega,\mathcal D}[u]\leq\lambda(u,u)$ for every $u\in S$.
Suppose, for contradiction, that there exists a nonzero
$u\in S\cap(P_K V)^{\perp}$. Since $u\in(P_K V)^{\perp}$, we have
\begin{align*}
    0=(u,P_Ku)=(P_Ku,P_Ku)\Rightarrow P_Ku=0
    \Rightarrow\mathcal E_{\Omega,\mathcal D}[u]>\lambda(u,u).
\end{align*}
This contradicts the defining property of $S$, which means $S\cap(P_K V)^{\perp}=\{\boldsymbol{0}\}$. Therefore,
\begin{align*}
    \dim S\leq\dim P_K V=\dim P_KV_{\leq K}+\dim P_KV_{>K}.
\end{align*}
Lemmas \ref{lem:P_K_prpoerty>K} and \ref{lem:P_K_prpoerty<=K} imply
\begin{align*}
    \dim P_KV_{\leq K}=\dim V_{\leq K}=\sum_{i\leq K}n_i,
\end{align*}
and
\begin{align*}
    \dim P_KV_{>K}=\#\mathcal{B}_K=p^{nK}|\mathcal{U}_K|.
\end{align*}
Consequently,
\begin{align}
\label{eq:N_H_lambda_up_low_bound}
    \sum_{i\leq K}n_i\leq N_{\boldsymbol{H}}(\lambda)\leq \sum_{i\leq K}n_i+p^{nK}|\mathcal{U}_K|.
\end{align}
Combining \eqref{eq:remainder_all} and \eqref{eq:N_H_lambda_up_low_bound} yields
\begin{align*}
    |\Omega|p^{n[\frac{1}{\alpha}\log_p\lambda]}-N_{D_{\Omega,\mathcal{D}}^{\alpha}}(\lambda)\in&\left[|\Omega_{>K}|p^{nK}-|\mathcal{U}_K|p^{nK},|\Omega_{>K}|p^{nK}\right]\\
    &=\left[|\mathcal{U}_K\cap\Omega|p^{nK}-|\mathcal{U}_K|p^{nK},|\mathcal{U}_K\cap\Omega|p^{nK}\right]\\
    &=\left[-|\mathcal{U}_K\cap\Omega^c|p^{nK},|\mathcal{U}_K\cap\Omega|p^{nK}\right].
\end{align*}
\end{proof}

\begin{theorem}
\label{thm:Weyl-Berry}
    Let $\Omega\subseteq\mathbb{Q}_p^n$ be a bounded, open, Jordan measurable set with $\dim_{M}(\partial\Omega)=d$. Define
    \begin{align*}
        C(\lambda)=\frac{|\Omega|p^{n[\frac{1}{\alpha}\log_p\lambda]}-N_{D_{\Omega,\mathcal{D}}^{\alpha}}(\lambda)}{p^{d[\frac{1}{\alpha}\log_p\lambda]}}.
    \end{align*}
    Assume the condition \eqref{eq:condition_AK} holds for every $K\geq L$.
    Then
    \begin{align*}
        -\overline{\mathcal{M}}_{\text{ex}}^{d}(\partial\Omega)\leq\varliminf_{\lambda\to+\infty}C(\lambda)
        \leq\varlimsup_{\lambda\to+\infty}C(\lambda)\leq\overline{\mathcal{M}}_{\text{in}}^{d}(\partial\Omega).
    \end{align*}
\end{theorem}
\begin{proof}
Let $K\geq L$ and $\lambda\in[p^{\alpha K},p^{\alpha(K+1)})$ for $K\in\mathbb{Z}$. Lemma \ref{lem:equiv_condition_AK} and Theorem \ref{thm:remainder_estimate} give
\begin{align*}
    |\Omega|p^{n[\frac{1}{\alpha}\log_p\lambda]}-N_{D_{\Omega,\mathcal{D}}^{\alpha}}(\lambda)\in\left[-|\partial\Omega_{p^{-K},\text{ex}}|p^{nK},|\partial\Omega_{p^{-K},\text{in}}|p^{nK}\right].
\end{align*}
Dividing the upper bound by \(p^{Kd}\) and taking the upper limit gives
\begin{align*}
    \varlimsup_{\lambda\to+\infty}C(\lambda)=\varlimsup_{\lambda\to+\infty}
    \frac{|\Omega|p^{n[\frac{1}{\alpha}\log_p\lambda]}
    -N_{D_{\Omega,\mathcal{D}}^{\alpha}}(\lambda)}{p^{Kd}}\leq\varlimsup_{K\to+\infty}
    \frac{|\partial\Omega_{p^{-K},\text{in}}|}{p^{-K(n-d)}}
    =\overline{\mathcal{M}}_{\text{in}}^{d}(\partial\Omega).
\end{align*}
Similarly, the lower bound gives
\begin{align*}
    \varliminf_{\lambda\to+\infty}C(\lambda)=\varliminf_{\lambda\to+\infty}
    \frac{|\Omega|p^{n[\frac{1}{\alpha}\log_p\lambda]}
    -N_{D_{\Omega,\mathcal{D}}^{\alpha}}(\lambda)}{p^{Kd}}\geq-\varlimsup_{K\to+\infty}
    \frac{|\partial\Omega_{p^{-K},\text{ex}}|}{p^{-K(n-d)}}
    =-\overline{\mathcal{M}}_{\text{ex}}^{d}(\partial\Omega).
\end{align*}
\end{proof}

\begin{remark}
    If both $\mathcal{M}_{\text{ex}}^{d}(\partial\Omega)$ and
    $\mathcal{M}_{\text{in}}^{d}(\partial\Omega)$ exist, then
    \begin{align*}
        -\mathcal{M}_{\text{ex}}^{d}(\partial\Omega)\leq\varliminf_{\lambda\to\infty}C(\lambda)
        \leq\varlimsup_{\lambda\to\infty}C(\lambda)\leq\mathcal{M}_{\text{in}}^{d}(\partial\Omega).
    \end{align*}
\end{remark}

\subsection{Sharp bounds for the remainder}

We will show that the bounds for the normalized remainder $C(\lambda)$ in Theorem \ref{thm:Weyl-Berry} are sharp, namely, that there are domains $\Omega$ such that
\begin{align*}
    &\varlimsup_{\lambda\to\infty}C(\lambda)=\mathcal{M}_{\text{in}}^{d}(\partial\Omega);\\
    &\varliminf_{\lambda\to\infty}C(\lambda)=-\mathcal{M}_{\text{ex}}^{d}(\partial\Omega).
\end{align*}

Note that $\mathbb{Z}_p^n/p\mathbb{Z}_p^n\cong\mathbb{F}_p^n$. Choose a
nonempty proper subset $I\subsetneq\mathbb F_p^n$, set
\begin{align*}
    d=\log_p\#I\in[0,n),
    \qquad
    J=\mathbb F_p^n\backslash I,
\end{align*}
and write $J=J_{\text{in}}\sqcup J_{\text{ex}}$ with
$J_{\text{in}}\neq\emptyset$. Thus $\#I=p^d$ and
$\#J=p^n-p^d$. Define
\begin{align*}
    W_k^{\text{in}}=\{\boldsymbol{a}=(\boldsymbol{a}_0,\dots,\boldsymbol{a}_{k-1})|\ \boldsymbol{a}_0,\dots,\boldsymbol{a}_{k-2}\in I,\boldsymbol{a}_{k-1}\in J_{\text{in}}\}.
\end{align*}
Each $\boldsymbol{a}\in W_k^{\text{in}}$ determines a unique ball
\begin{align*}
    B_{k}(\boldsymbol{a})=\sum_{i=0}^{k-1}\boldsymbol{a}_{i}p^i+p^k\mathbb{Z}_p^n.
\end{align*}
Define
\begin{align*}
    \Omega=\bigcup_{k=1}^{\infty}\bigcup_{\boldsymbol{a}\in W_k^{\text{in}}}B_{k}(\boldsymbol{a}).
\end{align*}
\begin{lemma}
    The set $\Omega$ is bounded, open, and Jordan measurable, and it satisfies the condition \eqref{eq:condition_AK} for every $k\geq 1$.
\end{lemma}
\begin{proof}
    Since $\Omega$ is a union of open balls $B_{k}(\boldsymbol{a})$, it is open. Moreover $\Omega\subseteq\mathbb{Z}_p^n$, so it is bounded. Since
    \begin{align*}
        \mathcal{B}_k=\{B_k(\boldsymbol{a})|\ \boldsymbol{a}=(\boldsymbol{a}_0,\dots,\boldsymbol{a}_{k-1})\in I^k\},
    \end{align*}
    and
    \begin{align*}
        \partial\Omega=\{\boldsymbol{x}=\sum_{i=0}^{\infty}\boldsymbol{a}_{i}p^i|\ \boldsymbol{a}_{i}\in I\}.
    \end{align*}
    The preceding descriptions show directly that condition \eqref{eq:condition_AK} holds for every $k\geq 1$. Therefore,
    \begin{align*}
        |(\partial\Omega)_{p^{-k}}|=\#I^k|B_{k}(\boldsymbol{a})|=p^{k(d-n)}.
    \end{align*}
    Hence
    \begin{align*}
        |\partial\Omega|=\lim_{k\to\infty}|(\partial\Omega)_{p^{-k}}|=\lim_{k\to\infty}p^{k(d-n)}=0,
    \end{align*}
    which means $\Omega$ is Jordan measurable.
\end{proof}

\begin{lemma}
    $\mathcal{M}^{d}(\partial\Omega)=1$, $\mathcal{M}_{\text{in}}^{d}(\partial\Omega)=\frac{\# J_{\text{in}}}{\# J}$, and $\mathcal{M}_{\text{ex}}^{d}(\partial\Omega)=\frac{\# J_{\text{ex}}}{\# J}$.
\end{lemma}
\begin{proof}
    Because $|(\partial\Omega)_{p^{-K}}|=p^{K(d-n)}$,
    \begin{align*}
        \mathcal{M}^{d}(\partial\Omega)=\lim_{K\to\infty}\frac{|(\partial\Omega)_{p^{-K}}|}{p^{K(d-n)}}=1.
    \end{align*}

    Fix $\boldsymbol{a}\in I^K$, a point of $\Omega\cap B_K(\boldsymbol{a})$ has an initial block of digits in $I$, followed at some level by a digit in $J_{\text{in}}$. Therefore,
    \begin{equation}
    \label{eq:Omega_cap_B_Ka}
        \begin{aligned}
            |\Omega\cap B_K(\boldsymbol{a})|&=\sum_{l=1}^{\infty}\# J_{\text{in}}p^{d(l-1)}p^{-n(K+l)}\\
            &=\# J_{\text{in}}p^{-d-nK}\sum_{l=1}^{\infty}p^{(d-n)l}\\
            &=\# J_{\text{in}}\frac{p^{-nK}}{p^n-p^d}\\
            &=p^{-nK}\frac{\# J_{\text{in}}}{\# J}.
        \end{aligned}
    \end{equation}
    Hence
    \begin{align*}
        |\partial\Omega_{p^{-K},\text{in}}|=\# I^K|\Omega\cap B_K(\boldsymbol{a})|=p^{dK}p^{-nK}\frac{\# J_{\text{in}}}{\# J}=p^{K(d-n)}\frac{\# J_{\text{in}}}{\# J}.
    \end{align*}
    Dividing by \(p^{K(d-n)}\) and passing to the limit gives
    \begin{align*}
        \mathcal{M}_{\text{in}}^{d}(\partial\Omega)=\lim_{K\to\infty}\frac{|\partial\Omega_{p^{-K},\text{in}}|}{p^{K(d-n)}}=\frac{\# J_{\text{in}}}{\# J}.
    \end{align*}
    The same calculation, with \(J_{\mathrm{in}}\) replaced by \(J_{\mathrm{ex}}\), yields
    \begin{align*}
        \mathcal{M}_{\text{ex}}^{d}(\partial\Omega)=\frac{\# J_{\text{ex}}}{\# J}.
    \end{align*}
\end{proof}

For $\boldsymbol{a}\in I^K$, define $C_K(\boldsymbol{a})$ by
    \begin{align*}
        C_K(\boldsymbol{a})=\{B_{K+1}(\boldsymbol{b})|\ \boldsymbol{b}=(\boldsymbol{a},\boldsymbol{b}_K),\boldsymbol{b}_K\in J_{\text{in}}\},
    \end{align*}
and set
    \begin{align*}
        u_{K,\boldsymbol{a}}(\boldsymbol{x})=\sum_{B_{K+1}(\boldsymbol{b})\in C_K(\boldsymbol{a})}\frac{\boldsymbol{1}_{B_{K+1}(\boldsymbol{b})}(\boldsymbol{x})}{\sqrt{V_{K+1}}}=\sum_{B_{K+1}(\boldsymbol{b})\in C_K(\boldsymbol{a})}\boldsymbol{e}_{B_{K+1}(\boldsymbol{b})}(\boldsymbol{x}).
    \end{align*}
    Then $\|u_{K,\boldsymbol{a}}\|_{L^2(\Omega)}^2=\# J_{\text{in}}$. Finally, let
\begin{align*}
    U_K=\operatorname{span}\{u_{K,\boldsymbol{a}}(\boldsymbol{x})|\ \boldsymbol{a}\in I^K\}.
\end{align*}
Hence $\dim U_K=\# I^K=p^{Kd}$.

\begin{lemma}
\label{lem:P_Ku_and_u_in_L2}
    For every $u\in U_K$, we have $\|P_Ku\|_{L^2(\mathbb{Q}_p^n)}^2=\frac{\# J_{\text{in}}}{p^n}\|u\|_{L^2(\Omega)}^2$.
\end{lemma}
\begin{proof}
    It suffices to verify the identity for each base $u_{K,\boldsymbol{a}}$. Formula \eqref{eq:P_Ku} gives
    \begin{align*}
        P_Ku_{K,\boldsymbol{a}}(\boldsymbol{x})=p^{nK}\int_{B_K(\boldsymbol{a})}u_{K,\boldsymbol{a}}(\boldsymbol{y})d\boldsymbol{y}\boldsymbol{1}_{B_K(\boldsymbol{a})}(\boldsymbol{x}).
\end{align*}
    Moreover,
    \begin{align*}
        \int_{B_{K}(\boldsymbol{a})}\boldsymbol{e}_{B_{K+1}(\boldsymbol{b})}(\boldsymbol{y})d\boldsymbol{y}=\int_{B_{K+1}(\boldsymbol{b})}\boldsymbol{e}_{B_{K+1}(\boldsymbol{b})}(\boldsymbol{y})d\boldsymbol{y}=\sqrt{V_{K+1}}=p^{\frac{-n(K+1)}{2}}.
    \end{align*}
    It follows that
    \begin{align*}
        P_Ku_{K,\boldsymbol{a}}(\boldsymbol{x})=\# J_{\text{in}}p^{nK}p^{\frac{-n(K+1)}{2}}\boldsymbol{1}_{B_K(\boldsymbol{a})}(\boldsymbol{x}).
    \end{align*}
    Consequently,
    \begin{align*}
        \|P_Ku_{K,\boldsymbol{a}}\|_{L^2(\mathbb{Q}_p^n)}^2&=\left(\# J_{\text{in}}p^{nK}p^{\frac{-n(K+1)}{2}}\right)^2V_K\\
        &=(\# J_{\text{in}})^2p^{2nK}p^{-n(K+1)}p^{-nK}\\
        &=(\# J_{\text{in}})^2p^{-n}\\
        &=\frac{\# J_{\text{in}}}{p^n}\|u_{K,\boldsymbol{a}}\|_{L^2(\Omega)}^2.
    \end{align*}
\end{proof}

\begin{lemma}
\label{lem:P_Kw_and_w_in_L2}
    For every $w\in V_{\leq K}\oplus U_K$, we have $P_{K+1}w=w$ and $\|P_Kw\|_{L^2(\mathbb{Q}_p^n)}^2\geq\frac{\# J_{\text{in}}}{p^n}\|w\|_{L^2(\Omega)}^2$.
\end{lemma}
\begin{proof}
    Since $V_{\leq K}\oplus U_K\subseteq V_{\leq K+1}$, Lemma \ref{lem:P_K_prpoerty<=K} gives $P_{K+1}w=w$. Write $w=v+u$, where $v\in V_{\leq K}$ and $u\in U_K\subseteq V_{=K+1}$. Then $\operatorname{supp}u\cap\operatorname{supp}v=\emptyset$, hence
    \begin{align*}
        \|w\|_{L^2(\Omega)}^2=\|v\|_{L^2(\Omega)}^2+\|u\|_{L^2(\Omega)}^2.
    \end{align*}
    Moreover, Lemma \ref{lem:P_K_prpoerty<=K} and Lemma \ref{lem:P_Ku_and_u_in_L2} imply that
    \begin{align*}
        \|P_Kw\|_{L^2(\mathbb{Q}_p^n)}^2&=\|P_Kv+P_Ku\|_{L^2(\mathbb{Q}_p^n)}^2\\
        &=\|P_Kv\|_{L^2(\mathbb{Q}_p^n)}^2+\|P_Ku\|_{L^2(\mathbb{Q}_p^n)}^2\\
        &=\|v\|_{L^2(\Omega)}^2+\frac{\# J_{\text{in}}}{p^n}\|u\|_{L^2(\Omega)}^2\\
        &\geq\frac{\# J_{\text{in}}}{p^n}\|w\|_{L^2(\Omega)}^2.
    \end{align*}
\end{proof}

\begin{theorem}
    For the domain $\Omega$ constructed above, we have
    \begin{align*}
        \varliminf_{\lambda\to\infty}C(\lambda)=-\mathcal{M}_{\text{ex}}^{d}(\partial\Omega).
    \end{align*}
    Furthermore, if $\frac{\# J_{\text{in}}}{\# J}<1-p^{-\alpha}$, then we have
    \begin{align*}
        \varlimsup_{\lambda\to\infty}C(\lambda)=\mathcal{M}_{\text{in}}^{d}(\partial\Omega).
    \end{align*}
\end{theorem}

\begin{proof}
First,
\begin{align*}
    n_i=\#W_i^{\text{in}}=\# I^{i-1}\# J_{\text{in}}=\# J_{\text{in}}p^{d(i-1)}.
\end{align*}

For $\lambda\in[p^{\alpha K},p^{\alpha(K+1)})$ with $K\in\mathbb{Z}$, the first term in \eqref{eq:remainder_all} is
\begin{align*}
    \sum_{i\geq K+1}n_ip^{-ni}p^{nK}&=\# J_{\text{in}}\sum_{i\geq K+1}p^{d(i-1)}p^{-ni}p^{nK}\\
    &=\# J_{\text{in}}p^{nK-d}\sum_{i\geq K+1}p^{(d-n)i}\\
    &=\# J_{\text{in}}\frac{p^{dK-n}}{1-p^{d-n}}\\
    &=p^{dK}\frac{\# J_{\text{in}}}{\# J}\\
    &=p^{dK}\mathcal{M}_{\text{in}}^{d}(\partial\Omega).
\end{align*}
Consequently
\begin{align*}
    C(\lambda)=\mathcal{M}_{\text{in}}^{d}(\partial\Omega)-\frac{N_{\boldsymbol{H}}(\lambda)-\dim V_{\leq K}}{p^{dK}}.
\end{align*}
Theorem \ref{thm:Weyl-Berry} now gives
\begin{equation}
\label{eq:C_lambda_bounds}
\begin{aligned}
    -\mathcal{M}_{\text{ex}}^{d}(\partial\Omega)\leq\varliminf_{\lambda\to\infty}C(\lambda)
    \leq\varlimsup_{\lambda\to\infty}C(\lambda)\leq\mathcal{M}_{\text{in}}^{d}(\partial\Omega).
\end{aligned}
\end{equation}
In view of the preceding identity for $C(\lambda)$, these inequalities are
equivalent to
\begin{align*}
    0\leq\varliminf_{\lambda\to\infty}
    \frac{N_{\boldsymbol H}(\lambda)-\dim V_{\leq K}}{p^{dK}}\leq\varlimsup_{\lambda\to\infty}
    \frac{N_{\boldsymbol H}(\lambda)-\dim V_{\leq K}}{p^{dK}}\leq1.
\end{align*}

\begin{enumerate}
    \item Attainment of the lower bound $-\mathcal{M}_{\text{ex}}^{d}(\partial\Omega)$:\\
    Let $w\in V_{\leq K}\oplus U_K\subseteq V_{\leq K+1}$. Since $\operatorname{supp}\hat{w}\subseteq\{\boldsymbol{\xi}\in\mathbb{Q}_p^n|\ \|\boldsymbol{\xi}\|_p\leq p^{K+1}\}$, Lemma \ref{lem:P_Kw_and_w_in_L2} yields
\begin{align*}
    \mathcal E_{\Omega,\mathcal D}[w]&=\int_{\|\boldsymbol{\xi}\|_p\leq p^{K+1}}\|\boldsymbol{\xi}\|_p^{\alpha}|\hat{w}(\boldsymbol{\xi})|^2d\boldsymbol{\xi}\\
    &=\int_{\|\boldsymbol{\xi}\|_p\leq p^{K}}\|\boldsymbol{\xi}\|_p^{\alpha}|\hat{w}(\boldsymbol{\xi})|^2d\boldsymbol{\xi}+\int_{p^K<\|\boldsymbol{\xi}\|_p\leq p^{K+1}}\|\boldsymbol{\xi}\|_p^{\alpha}|\hat{w}(\boldsymbol{\xi})|^2d\boldsymbol{\xi}\\
    &\leq p^{\alpha K}\|P_Kw\|_{L^2(\mathbb{Q}_p^n)}^2+p^{\alpha (K+1)}\|(I-P_K)w\|_{L^2(\mathbb{Q}_p^n)}^2\\
    &=p^{\alpha K}\|P_Kw\|_{L^2(\mathbb{Q}_p^n)}^2+p^{\alpha (K+1)}\left(\|w\|_{L^2(\Omega)}^2-\|P_Kw\|_{L^2(\mathbb{Q}_p^n)}^2\right)\\
    &=p^{\alpha (K+1)}\|w\|_{L^2(\Omega)}^2-\left(p^{\alpha (K+1)}-p^{\alpha K}\right)\|P_Kw\|_{L^2(\mathbb{Q}_p^n)}^2\\
    &\leq p^{\alpha (K+1)}\|w\|_{L^2(\Omega)}^2-\left(p^{\alpha (K+1)}-p^{\alpha K}\right)\frac{\# J_{\text{in}}}{p^n}\|w\|_{L^2(\Omega)}^2\\
    &=p^{\alpha (K+1)}\left(1-\frac{\# J_{\text{in}}}{p^n}(1-p^{-\alpha})\right)\|w\|_{L^2(\Omega)}^2.
\end{align*}
Observe that
\begin{align*}
    p^{-\alpha}<1-\frac{\# J_{\text{in}}}{p^n}(1-p^{-\alpha})<1.
\end{align*}
Define $\lambda_K^{+}=p^{\alpha (K+1)}\left(1-\frac{\# J_{\text{in}}}{p^n}(1-p^{-\alpha})\right)$. Then $\lambda_K^{+}\in(p^{\alpha K},p^{\alpha (K+1)})$ and the preceding estimate gives $\mathcal E_{\Omega,\mathcal D}[w]\leq\lambda_K^{+}(w,w)$. Lemma \ref{lem:Min-Max_Principle} therefore yields
\begin{align*}
    N_{\boldsymbol{H}}(\lambda_K^{+})\geq\dim V_{\leq K}+\dim U_K=\dim V_{\leq K}+p^{Kd}.
\end{align*}
Consequently,
\begin{align*}
    \varlimsup_{\lambda\to\infty}\frac{N_{\boldsymbol{H}}(\lambda)-\dim V_{\leq K}}{p^{dK}}\geq 1.
\end{align*}
Together with \eqref{eq:C_lambda_bounds}, this implies
\begin{align*}
    \varlimsup_{\lambda\to\infty}\frac{N_{\boldsymbol{H}}(\lambda)-\dim V_{\leq K}}{p^{dK}}=1.
\end{align*}
    \item Attainment of the upper bound $\mathcal{M}_{\text{in}}^{d}(\partial\Omega)$:\\
    For any $u\in V_{>K}\cap H_0^{\alpha/2}(\Omega)$, the Cauchy-Schwarz inequality and \eqref{eq:Omega_cap_B_Ka} give
    \begin{align*}
        \|P_Ku\|_{L^2(\mathbb{Q}_p^n)}^2&=\sum_{\boldsymbol{a}\in I^K}\frac{1}{|B_K(\boldsymbol{a})|}\left|\int_{B_K(\boldsymbol{a})\cap\Omega}u(\boldsymbol{y})d\boldsymbol{y}\right|^2\\
        &\leq\sum_{\boldsymbol{a}\in I^K}\frac{|B_K(\boldsymbol{a})\cap\Omega|}{|B_K(\boldsymbol{a})|}\|u\|_{L^2(B_K(\boldsymbol{a})\cap\Omega)}^2\\
        &=\frac{\# J_{\text{in}}}{\# J}\sum_{\boldsymbol{a}\in I^K}\|u\|_{L^2(B_K(\boldsymbol{a})\cap\Omega)}^2\\
        &<(1-p^{-\alpha})\|u\|_{L^2(\Omega)}^2.
    \end{align*}
    Consequently,
    \begin{align*}
        \|(I-P_K)u\|_{L^2(\mathbb{Q}_p^n)}^2> p^{-\alpha}\|u\|_{L^2(\Omega)}^2.
    \end{align*}
    Therefore,
    \begin{align*}
        \mathcal E_{\Omega,\mathcal D}[u]&=\int_{\|\boldsymbol{\xi}\|_p\leq p^{K}}\|\boldsymbol{\xi}\|_p^{\alpha}|(P_Ku)^{\hat{}}(\boldsymbol{\xi})|^2d\boldsymbol{\xi}+\int_{\|\boldsymbol{\xi}\|_p\geq p^{K+1}}\|\boldsymbol{\xi}\|_p^{\alpha}|((I-P_K)u)^{\hat{}}(\boldsymbol{\xi})|^2d\boldsymbol{\xi}\\
        &\geq\int_{\|\boldsymbol{\xi}\|_p\geq p^{K+1}}\|\boldsymbol{\xi}\|_p^{\alpha}|((I-P_K)u)^{\hat{}}(\boldsymbol{\xi})|^2d\boldsymbol{\xi}\\
        &\geq p^{\alpha(K+1)}\|(I-P_K)u\|_{L^2(\mathbb{Q}_p^n)}^2\\
        &>p^{\alpha K}\|u\|_{L^2(\Omega)}^2.
    \end{align*}
    Set $\lambda_K^{-}=p^{\alpha K}$. Then $\mathcal E_{\Omega,\mathcal D}[u]>\lambda_K^{-}(u,u)$, and Lemma \ref{lem:Min-Max_Principle} yields
    \begin{align*}
        N_{\boldsymbol{H}}(\lambda_K^{-})\leq\dim V_{\leq K}.
    \end{align*}
    Therefore,
    \begin{align*}
        \varliminf_{\lambda\to\infty}\frac{N_{\boldsymbol{H}}(\lambda)-\dim V_{\leq K}}{p^{dK}}\leq 0.
    \end{align*}
    Combining this estimate with \eqref{eq:C_lambda_bounds}, we obtain
    \begin{align*}
        \varliminf_{\lambda\to\infty}\frac{N_{\boldsymbol{H}}(\lambda)-\dim V_{\leq K}}{p^{dK}}=0.
    \end{align*}
\end{enumerate}
\end{proof}

\section{Pólya's conjecture for the Vladimirov-Taibleson operator}
In the $p$-adic setting, Pólya’s conjecture takes the form
\begin{align*}
    N_{D_{\Omega,\mathcal{D}}^{\alpha}}(\lambda)\leq|\Omega|p^{n[\frac{1}{\alpha}\log_p\lambda]}.
\end{align*}

\subsection{Pólya's conjecture for compact open domains}
\begin{lemma}
    Pólya’s conjecture fails for $\Omega=B_L(\boldsymbol{a})$.
\end{lemma}
\begin{proof}
It follows from \eqref{eq:N_D_Ball_lambda} that, for $\lambda\in[p^{\alpha L}\frac{1-p^{-n}}{1-p^{-\alpha-n}},p^{\alpha L})$, we have
\begin{align*}
    N_{D_{B_L(\boldsymbol{a}),\mathcal{D}}^{\alpha}}(\lambda)=1>p^{-n}=|B_L(\boldsymbol{a})|p^{n[\frac{1}{\alpha}\log_p\lambda]}
\end{align*}
since
\begin{align*}
    \frac{1-p^{-n}}{1-p^{-\alpha-n}}-p^{-\alpha}
    =\frac{(p^\alpha-1)((p^n-1)p^\alpha-1)}
    {p^\alpha(p^{n+\alpha}-1)}>0.
\end{align*}
\end{proof}

\begin{theorem}
\label{thm:Polya_conjecture_not_hold_compact}
    When $\Omega$ is a compact open set, Pólya's conjecture does not hold.
\end{theorem}
\begin{proof}
    Let $\Omega=\displaystyle\bigsqcup_{i=1}^{m}B_{L_i}(\boldsymbol{a}_i)$ be its decomposition into disjoint maximal balls. Let $L=\displaystyle\max_{1\leq i\leq m}\{L_i\}$. Partition these $m$ maximal balls into $m_{=}$ balls of radius $p^{-L}$ and $m_{<}$ balls of radius greater than $p^{-L}$. Thus $m=m_{=}+m_{<}$ and $m_{=}\geq 1$. Choose $\varepsilon>0$ smaller than $p^{\alpha L}-p^{\alpha(L-1)}$ and smaller than the distance from $p^{\alpha L}$ to every eigenvalue of $\boldsymbol H$ that is strictly below $p^{\alpha L}$, and set $\lambda^-=p^{\alpha L}-\varepsilon$. We compute $N_{D_{\Omega,\mathcal{D}}^{\alpha}}(\lambda^-)$ by separating its two contributions.
    \begin{enumerate}
        \item $N_{D_{\Omega,\mathcal{D}}^{\alpha}}^0(\lambda^-)$: Theorem \ref{thm:eig_Dirichlet_compact} gives
        \begin{align*}
            p^{\alpha(\gamma+L_i)}\leq \lambda^-\Rightarrow\gamma\leq L-1-L_i.
        \end{align*}
        Since $\gamma\in\mathbb{N}^*$, this is possible only when $L_i<L$. Therefore,
        \begin{align*}
        N_{D_{\Omega,\mathcal{D}}^{\alpha}}^0(\lambda^-)&=\sum_{L_i<L}\sum_{\gamma=1}^{L-1-L_i} (p^{n\gamma} - p^{n(\gamma-1)})\\
            &=\sum_{L_i<L}(p^{n(L-1-L_i)} - 1)\\
            &=\sum_{L_i<L}p^{n(L-1-L_i)}-m_{<}.
        \end{align*}
        \item $N_{\boldsymbol{H}}(\lambda^-)$: For every $u\in \operatorname{span}\{\boldsymbol{1}_{B_{L_1}(\boldsymbol{a}_1)},\dots,\boldsymbol{1}_{B_{L_m}(\boldsymbol{a}_m)}\}$, $\operatorname{supp}\hat{u}\subseteq\{\boldsymbol{\xi}\in\mathbb{Q}_p^n|\ \|\boldsymbol{\xi}\|_p\leq p^L\}$. Hence
\begin{align*}
    (\boldsymbol{H}u,u)&=(D^{\alpha}\tilde{u},\tilde{u})\\
    &=\int_{\mathbb{Q}_p^n}\|\boldsymbol{\xi}\|_p^{\alpha}|\hat{u}(\boldsymbol{\xi})|^2d\boldsymbol{\xi}\\
    &=\int_{\|\boldsymbol{\xi}\|_p\leq p^L}\|\boldsymbol{\xi}\|_p^{\alpha}|\hat{u}(\boldsymbol{\xi})|^2d\boldsymbol{\xi}\\
    &\leq p^{\alpha L}\int_{\|\boldsymbol{\xi}\|_p\leq p^L}|\hat{u}(\boldsymbol{\xi})|^2d\boldsymbol{\xi}\\
    &=p^{\alpha L}(u,u).
\end{align*}
Equality holds if and only if $\hat{u}(\boldsymbol{\xi})=0$ for all $\|\boldsymbol{\xi}\|_p\leq p^{L-1}$, equivalently, the integral of $u$ over every ball of level $L-1$ is zero. Let $q$ be the number of level-$L-1$ balls that contain at least one of the $m_{=}$ maximal balls of level $L$, and let the $j$-th such ball contain $c_j$ of them. Maximality gives $c_j\leq p^n-1$, and
\begin{align*}
    \sum_{j=1}^q c_j = m_{=}.
\end{align*}
In a cluster containing $c_j$ maximal balls of smallest radius, requiring the total integral to vanish yields $c_j-1$ linearly independent eigenfunctions with eigenvalue $p^{\alpha L}$. These eigenfunctions contribute a total multiplicity of $\displaystyle\sum_{j=1}^q (c_j - 1) = m_{=} - q$. Therefore,
\begin{align*}
    N_{\boldsymbol H}(\lambda^-) = m - (m_{=} - q) = m_{<} + q.
\end{align*}
\end{enumerate}
Combining the two contributions, we obtain
\begin{align*}
    N_{D_{\Omega,\mathcal{D}}^{\alpha}}(\lambda^-)&=N_{D_{\Omega,\mathcal{D}}^{\alpha}}^0(\lambda^-)+N_{\boldsymbol H}(\lambda^-)\\
    &=\sum_{L_i<L}p^{n(L-1-L_i)}-m_{<}+m_{<} + q\\
    &=\sum_{L_i<L}p^{n(L-1-L_i)}+q.
\end{align*}
Consequently,
\begin{align*}
    N_{D_{\Omega,\mathcal{D}}^{\alpha}}(\lambda^-)-|\Omega|p^{n[\frac{1}{\alpha}\log_p\lambda^-]}&=\sum_{L_i<L}p^{n(L-1-L_i)}+q-\left(\sum_{L_i=L}p^{-nL}+\sum_{L_i<L}p^{-nL_i}\right)p^{n(L-1)}\\
    &=q-m_{=}p^{-n}\\
    &=\sum_{j=1}^{q}\left(1-\frac{c_j}{p^n}\right)\\
    &>0.
\end{align*}
\end{proof}

\subsection{A geometric characterization of Pólya's conjecture}

By \eqref{eq:remainder_all}, for $\lambda\in[p^{\alpha K},p^{\alpha (K+1)})$, we have
\begin{align*}
    |\Omega|p^{n[\frac{1}{\alpha}\log_p\lambda]}-N_{D_{\Omega,\mathcal{D}}^{\alpha}}(\lambda)=|\Omega_{>K}|p^{nK}+\sum_{i\leq K}n_i-N_{\boldsymbol{H}}(\lambda).
\end{align*}
Because $N_{\boldsymbol{H}}(\lambda)$ is nondecreasing in $\lambda$, Pólya's conjecture holds if and only if
\begin{align}
\label{eq:condition_Polya_conjecture_holds}
    N_{\boldsymbol{H}}^{-}(\lambda_{K+1})\leq|\Omega_{>K}|p^{nK}+\sum_{i\leq K}n_i,\ \forall K\in\mathbb{Z}.
\end{align}
Here $\lambda_{K+1}=p^{\alpha (K+1)}$, and
$N_{\boldsymbol{H}}^{-}(\lambda)
=\#\{i:\ \lambda_i(\boldsymbol H)<\lambda\}$, where eigenvalues are
repeated according to multiplicity.

For notational convenience, we identify each
$u\in V\cap H_0^{\alpha/2}(\Omega)$ with its zero extension $\tilde u$.

\begin{lemma}
    Let $\Delta_j=P_j-P_{j-1}$. Then
    \begin{align*}
        \mathcal{E}_{\Omega,\mathcal{D}}(u,u)=\displaystyle\sum_{j\in\mathbb{Z}}p^{\alpha j}\|\Delta_ju\|_{L^2(\mathbb{Q}_p^n)}^2,
        \qquad u\in V\cap H_0^{\alpha/2}(\Omega).
    \end{align*}
\end{lemma}
\begin{proof}
    A direct computation gives
    \begin{align*}
        \mathcal{E}_{\Omega,\mathcal{D}}(u,u)&=\int_{\mathbb{Q}_p^n}\|\boldsymbol{\xi}\|_p^{\alpha}|\hat{u}(\boldsymbol{\xi})|^2d\boldsymbol{\xi}\\
        &=\sum_{j\in\mathbb{Z}}\int_{\|\boldsymbol{\xi}\|_p=p^j}\|\boldsymbol{\xi}\|_p^{\alpha}|\hat{u}(\boldsymbol{\xi})|^2d\boldsymbol{\xi}\\
        &=\sum_{j\in\mathbb{Z}}p^{\alpha j}\int_{\|\boldsymbol{\xi}\|_p=p^j}|\widehat{P_ju}(\boldsymbol{\xi})-\widehat{P_{j-1}u}(\boldsymbol{\xi})|^2d\boldsymbol{\xi}\\
        &=\sum_{j\in\mathbb{Z}}p^{\alpha j}\int_{\mathbb{Q}_p^n}|\widehat{\Delta_ju}(\boldsymbol{\xi})|^2d\boldsymbol{\xi}\\
        &=\sum_{j\in\mathbb{Z}}p^{\alpha j}\|\Delta_ju\|_{L^2(\mathbb{Q}_p^n)}^2.
    \end{align*}
\end{proof}

Define the quadratic form $\mathcal{H}$ by
\begin{align*}
    \mathcal{H}[u]=\mathcal{E}_{\Omega,\mathcal{D}}(u,u),
    \qquad
    \mathcal{D}(\mathcal{H})=V\cap H_0^{\frac{\alpha}{2}}(\Omega).
\end{align*}
Let $\boldsymbol{H}$ be the self-adjoint operator associated with
$\mathcal{H}$. By the compactness result in Section 2,
$\boldsymbol{H}$ has compact resolvent. For $K\in\mathbb{Z}$, set
\begin{align*}
    \mathcal{Q}_K[u]
    =\mathcal{H}[u]-\lambda_{K+1}\|u\|_{L^2(\Omega)}^2,
    \qquad
    \mathcal{D}(\mathcal{Q}_K)=\mathcal{D}(\mathcal{H}).
\end{align*}
The min--max principle gives
\begin{align}
\label{eq:ind_Q_K=NH_lambda_K+1}
    \operatorname{ind}_{-}(\mathcal{Q}_K)
    =N_{\boldsymbol{H}}^{-}(\lambda_{K+1}),
\end{align}
where
\begin{align*}
    \operatorname{ind}_{-}(\mathcal{Q}_K)
    =\sup\{\dim W:\ \mathcal{Q}_K[u]<0
    \text{ for every }0\neq u\in W\subseteq\mathcal{D}(\mathcal{Q}_K)\}.
\end{align*}

Fix $K\in\mathbb{Z}$ and $B_K\in\mathcal{B}_K$, and set
\begin{align*}
    w_j=\lambda_{j+1}-\lambda_{K+1},\qquad j\geq K.
\end{align*}
All the minimizations below may be carried out in the corresponding
real subspace. Complexification does not change the eigenvalues,
their multiplicities, or the negative inertia.
For every function in $\mathcal{D}(\mathcal{H})$ that vanishes almost
everywhere outside $B_K$, define
\begin{align*}
    \mathcal{E}_{B_K}[u]
    =\sum_{j\geq K}w_j
    \|\Delta_{j+1}u\|_{L^2(B_K)}^2.
\end{align*}
Let
\begin{align*}
    \mathcal{A}_{B_K}=\{u\in\mathcal{D}(\mathcal{H}):\
    u=0\text{ a.e. on }B_K^c,\
    \bar{u}_{B_K}=1\},
\end{align*}
where
\begin{align*}
    \bar{u}_{B_K}=\frac{1}{|B_K|}\int_{B_K}u(\boldsymbol{x})d\boldsymbol{x},
\end{align*}
and define
\begin{align*}
    \kappa_{B_K}
    =\frac{1}{|B_K|}\inf_{u\in\mathcal{A}_{B_K}}
    \mathcal{E}_{B_K}[u].
\end{align*}

\begin{lemma}
For every $B_K\in\mathcal{B}_K$, the set $\mathcal{A}_{B_K}$ is
nonempty and $0\leq\kappa_{B_K}<+\infty$.
\end{lemma}
\begin{proof}
By the definition of $\mathcal{B}_K$, there is a maximal ball
$M_i=B_{L_i}(\boldsymbol{a}_i)\subseteq B_K$. Set
\begin{align*}
    u=\frac{|B_K|}{|M_i|}\boldsymbol{1}_{M_i}.
\end{align*}
Since $M_i$ is a compact open subset of $\Omega$, we have
$\boldsymbol{1}_{M_i}\in\mathcal{S}(\Omega)\cap V$. Thus
$u\in\mathcal{D}(\mathcal{H})$, it vanishes outside $B_K$, and
$\bar u_{B_K}=1$. Moreover, $u$ is a Bruhat--Schwartz function, so
$\mathcal{E}_{B_K}[u]<+\infty$.
\end{proof}

We now compute $\kappa_{B_K}$ by finite inner approximations. Fix
$r\in\mathbb{N}^*$. Starting from $B_K$, whenever a ball $B_j$
satisfies
\begin{align*}
    B_j\cap\Omega\neq\emptyset,
    \qquad B_j\cap\Omega^c\neq\emptyset,
    \qquad j<K+r,
\end{align*}
consider all the $p^n$ balls $B_{j+1}\subseteq B_j$. Stop considering
sub-balls as soon as either $B_j\subseteq\Omega$, or
$B_j\cap\Omega=\emptyset$, or $j=K+r$. Define
\begin{equation}
\label{eq:strict_Lambda_recursion}
\Lambda_{K,r}(B_j)=
\begin{cases}
0,& B_j\subseteq\Omega,\ K<j\leq K+r;\\
+\infty,& B_j\cap\Omega=\emptyset,\ K\leq j\leq K+r\text{ or }B_{j}\in\mathcal{B}_{j},\ j=K+r;\\
\frac{p^n}{\sum\limits_{B_{j+1}\subseteq B_j}\frac{1}{w_j+\Lambda_{K,r}(B_{j+1})}}-w_j,& B_j\in\mathcal{B}_j,\ K\leq j< K+r,
\end{cases}
\end{equation}
We use the conventions $\frac{1}{0}=+\infty$ and $\frac{1}{+\infty}=0$.

For $A\in[0,+\infty]$, the expression $Aa^2$ is interpreted as zero
when $A=+\infty$ and $a=0$, and as $+\infty$ when
$A=+\infty$ and $a\neq0$.

For every ball $B_j$ considered above, let
\begin{align*}
    \mathcal{A}_{B_j,r}^{a}
    =\Bigl\{u\in
    \operatorname{span}\{\boldsymbol{1}_{M_i}:\
    M_i=B_{L_i}(\boldsymbol{a}_i)\subseteq B_j,\ L_i\leq K+r\}:\
    \bar u_{B_j}=a\Bigr\},
\end{align*}
and set
\begin{align*}
    \mathcal{F}_{B_j,r}[u]
    =\sum_{l\geq j}w_l
    \|\Delta_{l+1}u\|_{L^2(B_j)}^2.
\end{align*}
If the displayed span is empty, it is understood to be $\{0\}$.

\begin{lemma}
Let $A_{B_{j+1}}\in[0,+\infty]$ be indexed by the $p^n$ balls
$B_{j+1}\subseteq B_j$. With the preceding conventions,
\begin{align}
\label{eq:extended_parallel_minimum}
    &\inf\left\{
    p^{-n}\sum_{B_{j+1}\subseteq B_j}
    A_{B_{j+1}}b_{B_{j+1}}^2
    \ \middle|\
    \sum_{B_{j+1}\subseteq B_j}b_{B_{j+1}}=p^na
    \right\}=
    p^n\left(\sum_{B_{j+1}\subseteq B_j}
    A_{B_{j+1}}^{-1}\right)^{-1}a^2.
\end{align}
\end{lemma}
\begin{proof}
If one of the coefficients is zero, assign the value $p^na$ to a
coordinate with zero coefficient and set all other coordinates equal
to zero. Both sides of \eqref{eq:extended_parallel_minimum} are then
zero. Suppose next that there is no zero coefficient and that at least
one coefficient is finite. Every coordinate with coefficient
$+\infty$ must vanish whenever the left-hand side is finite. Applied
to the remaining positive finite coefficients, the Cauchy--Schwarz
inequality gives
\begin{align*}
    \sum_{B_{j+1}\subseteq B_j} A_{B_{j+1}}b_{B_{j+1}}^2
    \geq
    \frac{\left(\sum\limits_{B_{j+1}\subseteq B_j} b_{B_{j+1}}\right)^2}
    {\sum\limits_{B_{j+1}\subseteq B_j} A_{B_{j+1}}^{-1}}.
\end{align*}
Equality is attained by taking the non-zero coordinates proportional to
$A_{B_{j+1}}^{-1}$. Finally, if all coefficients are $+\infty$, a
finite value is possible only when every coordinate is zero. Thus the
infimum is $+\infty$ for $a\neq0$ and is zero for $a=0$, exactly as on
the right-hand side.
\end{proof}

\begin{lemma}
\label{lem:Lambda_KrBj=inf_F}
For every ball $B_j$ considered in the construction and every
$a\in\mathbb{R}$,
\begin{align}
\label{eq:Lambda_variational_strict}
    \Lambda_{K,r}(B_j)a^2
    =\frac{1}{|B_j|}
    \inf_{u\in\mathcal{A}_{B_j,r}^{a}}
    \mathcal{F}_{B_j,r}[u].
\end{align}
\end{lemma}
\begin{proof}
We use backward induction on $j$. If $B_j\subseteq\Omega$ occurs in
the construction, the ball of level $j-1$ that contains $B_j$ is not
contained in $\Omega$. Hence $B_j$ is a maximal ball of $\Omega$.
Consequently,
\begin{align*}
    \mathcal{A}_{B_j,r}^{a}=\{a\boldsymbol{1}_{B_j}\},
    \qquad
    \mathcal{F}_{B_j,r}[a\boldsymbol{1}_{B_j}]=0,
\end{align*}
which agrees with $\Lambda_{K,r}(B_j)=0$.

If $B_j\cap\Omega=\emptyset$, no maximal ball of $\Omega$ is contained
in $B_j$. Suppose instead that $j=K+r$ and
$B_j\not\subseteq\Omega$. If a maximal ball
$M_i=B_{L_i}(\boldsymbol a_i)\subseteq B_j$ satisfied
$L_i\leq K+r$, then the ball inclusions would force $L_i=K+r$ and
$M_i=B_j$, contradicting $B_j\not\subseteq\Omega$. In either case, the
span in the definition of $\mathcal{A}_{B_j,r}^{a}$ is $\{0\}$. Thus a
nonzero prescribed average is impossible, whereas the zero function
realizes the zero average. This agrees with
$\Lambda_{K,r}(B_j)=+\infty$.

Now let $K\leq j<K+r$, assume that $B_j$ meets both $\Omega$ and
$\Omega^c$, and suppose that the assertion holds for every
$B_{j+1}\subseteq B_j$. Write
\begin{align*}
    b_{B_{j+1}}=\bar u_{B_{j+1}}.
\end{align*}
Since the $p^n$ balls $B_{j+1}\subseteq B_j$ are pairwise disjoint and
their union is $B_j$,
\begin{align*}
    \sum_{B_{j+1}\subseteq B_j}b_{B_{j+1}}=p^na
\end{align*}
and
\begin{align*}
    \|\Delta_{j+1}u\|_{L^2(B_j)}^2
    =|B_j|\left(
    p^{-n}\sum_{B_{j+1}\subseteq B_j}b_{B_{j+1}}^2-a^2
    \right).
\end{align*}
Moreover, restriction to the balls $B_{j+1}\subseteq B_j$ gives
\begin{align*}
    \mathcal{F}_{B_j,r}[u]
    =w_j\|\Delta_{j+1}u\|_{L^2(B_j)}^2
    +\sum_{B_{j+1}\subseteq B_j}
    \mathcal{F}_{B_{j+1},r}[u].
\end{align*}
First fixing all the averages $b_{B_{j+1}}$ and then applying the
induction hypothesis yields
\begin{align*}
    \frac{1}{|B_j|}
    \inf_{u\in\mathcal{A}_{B_j,r}^{a}}
    \mathcal{F}_{B_j,r}[u]=\inf_{\sum\limits_{B_{j+1}\subseteq B_j} b_{B_{j+1}}=p^na}
    \left\{
    p^{-n}\sum_{B_{j+1}\subseteq B_j}
    \bigl(w_j+\Lambda_{K,r}(B_{j+1})\bigr)b_{B_{j+1}}^2
    -w_ja^2\right\}.
\end{align*}
Applying \eqref{eq:extended_parallel_minimum} with
$A_{B_{j+1}}=w_j+\Lambda_{K,r}(B_{j+1})$ gives exactly
\eqref{eq:strict_Lambda_recursion} and
\eqref{eq:Lambda_variational_strict}. This also covers every possible
occurrence of zero and $+\infty$.
\end{proof}

\begin{theorem}
For every $B_K\in\mathcal{B}_K$, the sequence
$\Lambda_{K,r}(B_K)$ is nonincreasing and
\begin{align}
\label{eq:kappa_limit_strict}
    \kappa_{B_K}=\lim_{r\to\infty}\Lambda_{K,r}(B_K).
\end{align}
\end{theorem}
\begin{proof}
As $r$ increases, the definition of $\mathcal{A}_{B_K,r}^{1}$ permits
indicators of additional maximal balls. Hence
\begin{align*}
    \mathcal{A}_{B_K,r}^{1}
    \subseteq\mathcal{A}_{B_K,r+1}^{1},
\end{align*}
and Lemma \ref{lem:Lambda_KrBj=inf_F} shows that
\begin{align*}
    \Lambda_{K,r+1}(B_K)\leq\Lambda_{K,r}(B_K).
\end{align*}
Every function in $\mathcal{A}_{B_K,r}^{1}$ belongs to
$\mathcal{A}_{B_K}$, so
\begin{align*}
    \kappa_{B_K}\leq\Lambda_{K,r}(B_K).
\end{align*}

It remains to prove the reverse inequality in the limit. Let
$u\in\mathcal{A}_{B_K}$. By the maximal-ball averaging result in
Section 3, finite linear combinations of the functions
$\boldsymbol{1}_{M_i}$ are dense in
$V\cap H_0^{\alpha/2}(\Omega)$ in the form norm. Since multiplication
by $\boldsymbol{1}_{B_K}$ is bounded on
$H^{\alpha/2}(\mathbb{Q}_p^n)$, the approximating functions may be
chosen to vanish outside $B_K$. Thus there are finite linear
combinations $u_m$ of the functions $\boldsymbol{1}_{M_i}$ with
$M_i\subseteq B_K$ such that
\begin{align*}
    u_m\longrightarrow u
    \quad\text{in }L^2(\Omega),
    \qquad
    \mathcal{H}[u_m-u]\longrightarrow0.
\end{align*}
In particular, $\overline{u_m}_{B_K}\to1$. For all sufficiently large
$m$, replace $u_m$ by $u_m/\overline{u_m}_{B_K}$. The resulting functions have
average one and still converge to $u$ in the form norm. Since
\begin{align*}
    \mathcal{E}_{B_K}[z]\leq\mathcal{H}[z],
\end{align*}
their local energies converge to $\mathcal{E}_{B_K}[u]$.
Each such finite linear combination belongs to
$\mathcal{A}_{B_K,r}^{1}$ for all sufficiently large $r$. Therefore,
\begin{align*}
    \lim_{r\to\infty}\Lambda_{K,r}(B_K)
    \leq\frac{1}{|B_K|}\mathcal{E}_{B_K}[u].
\end{align*}
Taking the infimum over $u\in\mathcal{A}_{B_K}$ proves the reverse
inequality and hence \eqref{eq:kappa_limit_strict}.
\end{proof}

Every $u\in\mathcal{D}(\mathcal{H})$ can be written uniquely as
\begin{align*}
    u=v+\sum_{B_K\in\mathcal{B}_K}u_{B_K},
\end{align*}
where $v\in V_{\leq K}$ and, for each $B_K\in\mathcal{B}_K$,
$u_{B_K}\in\mathcal{D}(\mathcal{H})$ vanishes almost everywhere
outside $B_K$. Indeed, the maximal balls of level at most $K$ and the
maximal balls contained in the mutually disjoint balls
$B_K\in\mathcal{B}_K$ partition $\Omega$; multiplication by the
indicator of each $B_K$ preserves the strict form domain. Write
\begin{align*}
    f:=P_Ku
    =v+\sum_{B_K\in\mathcal{B}_K}
    a_{B_K}\boldsymbol{1}_{B_K}.
\end{align*}
For $f\in P_KV$, define
\begin{align*}
    \mathcal{W}_K[f]
    =\sum_{j\leq K}(\lambda_j-\lambda_{K+1})
    \|\Delta_jf\|_{L^2(\mathbb{Q}_p^n)}^2
\end{align*}
and
\begin{align*}
    \mathcal{S}_K[f]
    =\mathcal{W}_K[f]
    +\sum_{B_K\in\mathcal{B}_K}
    \kappa_{B_K}|B_K||a_{B_K}|^2.
\end{align*}

\begin{lemma}
For every $f\in P_KV$,
\begin{align}
\label{eq:strict_compression}
    \mathcal{S}_K[f]
    =\inf\{\mathcal{Q}_K[u]:\
    u\in\mathcal{D}(\mathcal{H}),\ P_Ku=f\}.
\end{align}
\end{lemma}
\begin{proof}
The constraint set is nonempty. Indeed, every basis function in
$V_{\leq K}$ is the indicator of a maximal ball of $\Omega$ and hence
belongs to $\mathcal{S}(\Omega)$. Moreover, if
$B_K\in\mathcal{B}_K$ and $M_i\subseteq B_K$ is a maximal ball, then
\begin{align*}
    P_K\left(\frac{|B_K|}{|M_i|}\boldsymbol{1}_{M_i}\right)
    =\boldsymbol{1}_{B_K}.
\end{align*}

Let $u\in\mathcal{D}(\mathcal{H})$ satisfy $P_Ku=f$ and use the
decomposition above. For $j\geq K$, the function
$\Delta_{j+1}u_{B_K}$ vanishes outside $B_K$. Since the balls in
$\mathcal{B}_K$ are pairwise disjoint, the corresponding functions are
orthogonal. Also, $\Delta_{j+1}v=0$ for $j\geq K$. Hence
\begin{align*}
    \mathcal{Q}_K[u]
    &=\sum_{j\leq K}(\lambda_j-\lambda_{K+1})
    \|\Delta_jf\|_{L^2(\mathbb{Q}_p^n)}^2+\sum_{B_K\in\mathcal{B}_K}
    \mathcal{E}_{B_K}[u_{B_K}].
\end{align*}
Since $\bar u_{B_K}=a_{B_K}$, the definition of $\kappa_{B_K}$,
applied after the homogeneous rescaling by $a_{B_K}$, gives
\begin{align*}
    \mathcal{Q}_K[u]\geq\mathcal{S}_K[f].
\end{align*}

Conversely, fix $\varepsilon>0$. For every
$B_K\in\mathcal{B}_K$, choose $g_{B_K}\in\mathcal{A}_{B_K}$ such
that
\begin{align*}
    \mathcal{E}_{B_K}[g_{B_K}]
    \leq |B_K|(\kappa_{B_K}+\varepsilon).
\end{align*}
The set $\mathcal{B}_K$ is finite, and therefore
\begin{align*}
    u_{\varepsilon}
    =v+\sum_{B_K\in\mathcal{B}_K}a_{B_K}g_{B_K}
    \in\mathcal{D}(\mathcal{H}),
    \qquad P_Ku_{\varepsilon}=f.
\end{align*}
The same orthogonal decomposition yields
\begin{align*}
    \mathcal{Q}_K[u_{\varepsilon}]
    \leq\mathcal{S}_K[f]
    +\varepsilon\sum_{B_K\in\mathcal{B}_K}
    |B_K||a_{B_K}|^2.
\end{align*}
Letting $\varepsilon\downarrow0$ proves
\eqref{eq:strict_compression}.
\end{proof}

Let $n_-$ denote the number of strictly negative eigenvalues, counted
with multiplicity, of a finite-dimensional Hermitian quadratic form.

\begin{lemma}
\label{lem:ind_Q_K=n-S_K}
\begin{align*}
    \operatorname{ind}_{-}(\mathcal{Q}_K)=n_-(\mathcal{S}_K).
\end{align*}
\end{lemma}
\begin{proof}
Let $W\subseteq\mathcal{D}(\mathcal{Q}_K)$ be a subspace on which
$\mathcal{Q}_K$ is strictly negative. The restriction of $P_K$ to
$W$ is injective. Indeed, if $0\neq u\in W$ and $P_Ku=0$, the
orthogonal decomposition used above gives
\begin{align*}
    \mathcal{Q}_K[u]
    =\sum_{B_K\in\mathcal{B}_K}
    \mathcal{E}_{B_K}[u_{B_K}]\geq0,
\end{align*}
a contradiction. By \eqref{eq:strict_compression}, $P_KW$ is a
strictly negative subspace for $\mathcal{S}_K$. Thus
\begin{align*}
    \operatorname{ind}_{-}(\mathcal{Q}_K)
    \leq n_-(\mathcal{S}_K).
\end{align*}

Conversely, let $F\subseteq P_KV$ be a maximal strictly negative
subspace for $\mathcal{S}_K$. Since $F$ is finite-dimensional, there
is a constant $\gamma>0$ such that
\begin{align*}
    \mathcal{S}_K[f]\leq-\gamma\|f\|_{L^2(\mathbb{Q}_p^n)}^2,
    \qquad f\in F.
\end{align*}
Choose the functions $g_{B_K}$ as in the preceding proof, and, for
\begin{align*}
    f=v+\sum_{B_K\in\mathcal{B}_K}
    a_{B_K}\boldsymbol{1}_{B_K},
\end{align*}
define
\begin{align*}
    L_{\varepsilon}f
    =v+\sum_{B_K\in\mathcal{B}_K}a_{B_K}g_{B_K}.
\end{align*}
Then $P_KL_{\varepsilon}f=f$, so $L_{\varepsilon}$ is injective, and
\begin{align*}
    \mathcal{Q}_K[L_{\varepsilon}f]
    &\leq\mathcal{S}_K[f]
    +\varepsilon\sum_{B_K\in\mathcal{B}_K}|B_K||a_{B_K}|^2\leq-(\gamma-\varepsilon)\|f\|_{L^2(\mathbb{Q}_p^n)}^2.
\end{align*}
Here we used
\begin{align*}
    \sum_{B_K\in\mathcal{B}_K}|B_K||a_{B_K}|^2
    \leq\|f\|_{L^2(\mathbb{Q}_p^n)}^2.
\end{align*}
For $0<\varepsilon<\gamma$, $L_{\varepsilon}F$ is a strictly
negative subspace for $\mathcal{Q}_K$. Hence
\begin{align*}
    \operatorname{ind}_{-}(\mathcal{Q}_K)
    \geq n_-(\mathcal{S}_K).
\end{align*}
\end{proof}

With respect to the basis $\boldsymbol{e}_{B_j}$ used above, write
the matrix of $\mathcal{W}_K$ as
\begin{align*}
    \boldsymbol{W}_K=
    \begin{pmatrix}
        \boldsymbol{W}_{00}&\boldsymbol{W}_{01}\\
        \boldsymbol{W}_{10}&\boldsymbol{W}_{11}
    \end{pmatrix},
\end{align*}
where $\boldsymbol{W}_{00}$ is negative definite and has size
$\sum\limits_{i\leq K}n_i$. If this sum is zero, set
$\boldsymbol{W}_{00}=\emptyset$. Define, as before,
\begin{align*}
    \boldsymbol{\kappa}_K
    =\operatorname{diag}(\kappa_{B_K}\mid B_K\in\mathcal{B}_K)
\end{align*}
and
\begin{align*}
    \boldsymbol{R}_K=
    \begin{cases}
        \boldsymbol{W}_{11}+\boldsymbol{\kappa}_K,
        &\sum_{i\leq K}n_i=0,\\[1mm]
        \boldsymbol{W}_{11}+\boldsymbol{\kappa}_K
        -\boldsymbol{W}_{10}\boldsymbol{W}_{00}^{-1}
        \boldsymbol{W}_{01},
        &\sum_{i\leq K}n_i>0.
    \end{cases}
\end{align*}
The matrix of $\mathcal{S}_K$ is
\begin{align*}
    \boldsymbol{S}_K=
    \begin{pmatrix}
        \boldsymbol{W}_{00}&\boldsymbol{W}_{01}\\
        \boldsymbol{W}_{10}&\boldsymbol{W}_{11}+\boldsymbol{\kappa}_K
    \end{pmatrix}.
\end{align*}
If $\boldsymbol{W}_{00}\neq\emptyset$, then
\begin{align*}
&\begin{pmatrix}
I&0\\-\boldsymbol{W}_{10}\boldsymbol{W}_{00}^{-1}&I
\end{pmatrix}
\boldsymbol{S}_K
\begin{pmatrix}
I&-\boldsymbol{W}_{00}^{-1}\boldsymbol{W}_{01}\\0&I
\end{pmatrix}=
\begin{pmatrix}
\boldsymbol{W}_{00}&0\\0&\boldsymbol{R}_K
\end{pmatrix}.
\end{align*}
The matrix on the left is obtained from $\boldsymbol{S}_K$ by a
congruence transformation. Sylvester's law of inertia therefore gives
\begin{align*}
    n_-(\boldsymbol{S}_K)
    =\sum_{i\leq K}n_i+n_-(\boldsymbol{R}_K).
\end{align*}
The same identity is immediate when
$\boldsymbol{W}_{00}=\emptyset$.

Combining \eqref{eq:condition_Polya_conjecture_holds},
\eqref{eq:ind_Q_K=NH_lambda_K+1}, and Lemma
\ref{lem:ind_Q_K=n-S_K} proves the following theorem.

\begin{theorem}[Pólya's conjecture for the Dirichlet operator]
\label{thm:Polya_strict_Schur_criterion}
Let $\Omega\subseteq\mathbb{Q}_p^n$ be bounded, open, and Jordan
measurable. Then Pólya's conjecture holds if and only if
\begin{align*}
    n_-(\boldsymbol{R}_K)
    \leq\left[|\Omega_{>K}|p^{nK}\right],
    \qquad K\in\mathbb{Z}.
\end{align*}
\end{theorem}
\begin{proof}
By the three identities cited above, condition
\eqref{eq:condition_Polya_conjecture_holds} is equivalent to
\begin{align*}
    \sum_{i\leq K}n_i+n_-(\boldsymbol{R}_K)
    \leq |\Omega_{>K}|p^{nK}+\sum_{i\leq K}n_i.
\end{align*}
After cancelling the common term, this becomes
\begin{align*}
    n_-(\boldsymbol{R}_K)\leq|\Omega_{>K}|p^{nK}.
\end{align*}
Since the left-hand side is an integer, the latter inequality is
equivalent to the asserted inequality with the floor function.
\end{proof}

For fixed $p$, $n$, and $\alpha$, this criterion is determined by the maximal
balls, the mixed balls, and their inclusion relations.

\begin{corollary}
\label{cor:bigeometry_condition}
Suppose that, for every $K\in\mathbb{Z}$ and every
$B_K\in\mathcal{B}_K$, either
\begin{align*}
    \kappa_{B_K}<\delta_K
    \qquad\text{or}\qquad
    \kappa_{B_K}\geq\lambda_{K+1},
\end{align*}
where $\delta_K=\lambda_{K+1}-\lambda_K$. Define
\begin{align*}
    \mathcal{G}_K
    =\{B_K\in\mathcal{B}_K:\ \kappa_{B_K}<\delta_K\}.
\end{align*}
Then Pólya's conjecture holds if and only if
\begin{align*}
    \#\mathcal{G}_K
    \leq\left[|\Omega_{>K}|p^{nK}\right],
    \qquad K\in\mathbb{Z}.
\end{align*}
\end{corollary}
\begin{proof}
On
\begin{align*}
    V_{\leq K}\oplus
    \operatorname{span}\{\boldsymbol{1}_{B_K}|
    B_K\in\mathcal{G}_K\},
\end{align*}
we have
\begin{align*}
    \mathcal{W}_K[f]\leq-\delta_K\|f\|_2^2.
\end{align*}
\par\noindent
The inequalities $\kappa_{B_K}<\delta_K$ then imply
$\mathcal{S}_K[f]<0$ for every nonzero $f$ in this subspace.
Consequently,
\begin{align*}
    n_-(\boldsymbol{S}_K)
    \geq\sum_{i\leq K}n_i+\#\mathcal{G}_K.
\end{align*}
\par\noindent
On the other hand, if
\begin{align*}
    f\in\operatorname{span}\{\boldsymbol{1}_{B_K}|
    B_K\in\mathcal{B}_K\setminus\mathcal{G}_K\},
\end{align*}
then
\begin{align*}
    \mathcal{S}_K[f]
    \geq\mathcal{W}_K[f]+\lambda_{K+1}\|f\|_2^2=\sum_{j\leq K}\lambda_j
    \|\Delta_jf\|_{L^2(\mathbb{Q}_p^n)}^2\geq0.
\end{align*}
This is a subspace of codimension
$\sum_{i\leq K}n_i+\#\mathcal{G}_K$ in $P_KV$.
Every strictly negative subspace for $\mathcal{S}_K$ has trivial
intersection with it, and hence has dimension at most this
codimension. Therefore,
\begin{align*}
    n_-(\boldsymbol{S}_K)
    =\sum_{i\leq K}n_i+\#\mathcal{G}_K.
\end{align*}
The assertion now follows from the preceding theorem.
\end{proof}

\begin{lemma}[Fixed-scale geometric reduction]
Fix a single scale $K\in\mathbb{Z}$. The assertion below concerns only
this fixed $K$. Suppose that, for every
$B_K\in\mathcal{B}_K$, there is a maximal ball
$B_{K+1}\subseteq B_K\cap\Omega$. Then
\begin{align*}
    \kappa_{B_K}=0,
    \qquad B_K\in\mathcal{B}_K.
\end{align*}
Moreover, the $K$-th inequality in Theorem
\ref{thm:Polya_strict_Schur_criterion} holds if and only if
\begin{align*}
    \#\mathcal{B}_K=|\Omega_{>K}|p^{nK}.
\end{align*}
\end{lemma}
\begin{proof}
For each $B_K\in\mathcal{B}_K$, let
$B_{K+1}\subseteq B_K\cap\Omega$ be a maximal ball of $\Omega$ and set
\begin{align*}
    g=p^n\boldsymbol{1}_{B_{K+1}}.
\end{align*}
Since $B_{K+1}$ is a compact open subset of $\Omega$,
$g\in\mathcal{S}(\Omega)\cap V$ and
\begin{align*}
    P_Kg=\boldsymbol{1}_{B_K},
    \qquad
    \Delta_{j+1}g=0\quad(j>K).
\end{align*}
The remaining term in $\mathcal{E}_{B_K}$ has coefficient
\begin{align*}
    w_K=\lambda_{K+1}-\lambda_{K+1}=0.
\end{align*}
Thus $\mathcal{E}_{B_K}[g]=0$ and $\kappa_{B_K}=0$. Hence
$\mathcal S_K[f]=\mathcal W_K[f]$ on $P_KV$. For every nonzero
$f\in P_KV$,
\begin{align*}
    \mathcal S_K[f]
    =\sum_{j\leq K}(\lambda_j-\lambda_{K+1})
      \|\Delta_jf\|_{L^2(\mathbb Q_p^n)}^2\leq-(\lambda_{K+1}-\lambda_K)
      \|f\|_{L^2(\mathbb Q_p^n)}^2<0.
\end{align*}
Consequently,
\begin{align*}
    n_-(\boldsymbol S_K)
    =\dim P_KV
    =\sum_{i\leq K}n_i+\#\mathcal{B}_K.
\end{align*}
The Schur complement identity gives
\begin{align*}
    n_-(\boldsymbol R_K)=\#\mathcal{B}_K.
\end{align*}

On the other hand, the ball inclusions give
\begin{align*}
    \Omega_{>K}\subseteq\mathcal{U}_K
    =\bigsqcup_{B_K\in\mathcal{B}_K}B_K,
\end{align*}
and therefore
\begin{align*}
    |\Omega_{>K}|p^{nK}\leq\#\mathcal{B}_K.
\end{align*}
Thus the inequality
\begin{align*}
    \#\mathcal{B}_K
    \leq\left[|\Omega_{>K}|p^{nK}\right]
\end{align*}
holds if and only if
\begin{align*}
    \#\mathcal{B}_K=|\Omega_{>K}|p^{nK}.
\end{align*}
\end{proof}

\begin{remark}
If the hypothesis in the preceding fixed-scale lemma were imposed
simultaneously for every $K\in\mathbb{Z}$, no nonempty bounded
$\Omega$ could satisfy it. Indeed, choose a ball
$B_L(\boldsymbol a)$ containing $\Omega$. If $K<L-1$, let $B_K$ be the
ball of level $K$ containing $B_L(\boldsymbol a)$. Then
\begin{align*}
    \Omega\subseteq B_L(\boldsymbol a)\subsetneq B_K,
\end{align*}
so $B_K\in\mathcal{B}_K$. If a maximal ball $B_{K+1}$ of $\Omega$
satisfied $B_{K+1}\subseteq B_K\cap\Omega$, then it would meet
$B_L(\boldsymbol a)$. Since $K+1<L$, the inclusion property of
$p$-adic balls would give
\begin{align*}
    B_L(\boldsymbol a)\subsetneq B_{K+1}
    \subseteq\Omega\subseteq B_L(\boldsymbol a),
\end{align*}
a contradiction. Thus the geometric hypothesis is necessarily a
condition on sufficiently fine scales only.
\end{remark}

\Needspace{17\baselineskip}
The following corollary is the global formulation used below.

\begin{corollary}[Fine scales and finitely many coarse scales]
Suppose that there exists $K_0\in\mathbb{Z}$ such that, for every
$K\geq K_0$ and every $B_K\in\mathcal{B}_K$, there is a maximal ball
$B_{K+1}\subseteq B_K\cap\Omega$. Then Pólya's conjecture holds if
and only if the following two conditions are satisfied:
\par\smallskip
\begin{enumerate}
    \item $\#\mathcal{B}_K=|\Omega_{>K}|p^{nK}$ for every $K\geq K_0$;
    \item $n_-(\boldsymbol{R}_K)
    \leq\left[|\Omega_{>K}|p^{nK}\right]$ for every $K<K_0$.
\end{enumerate}
The second condition is automatic for all sufficiently negative $K$.
Hence only finitely many coarse scales require a separate verification.
\end{corollary}
\begin{proof}
For $K\geq K_0$, the general criterion at this scale is equivalent to the
first condition by the preceding lemma.
For $K<K_0$, the second condition is precisely the general Schur
criterion.

It remains to prove the last assertion. Choose a ball
$B_L(\boldsymbol{a})$ containing $\Omega$. The singular-integral
representation of the Dirichlet form and the zero extension give
\begin{align*}
    \mathcal{E}_{\Omega,\mathcal{D}}[u]
    &\geq
    \frac{p^{\alpha}-1}{1-p^{-\alpha-n}}
    \int_{\Omega}|u(\boldsymbol{x})|^2
    \left(\int_{B_L(\boldsymbol{a})^c}
    \frac{d\boldsymbol{y}}
    {\|\boldsymbol{x}-\boldsymbol{y}\|_p^{n+\alpha}}
    \right)d\boldsymbol{x}\\
    &=p^{\alpha L}
    \frac{1-p^{-n}}{1-p^{-\alpha-n}}
    \|u\|_{L^2(\Omega)}^2.
\end{align*}
Hence the bottom of the spectrum of $\boldsymbol{H}$ is at least
\begin{align*}
    p^{\alpha L}\frac{1-p^{-n}}{1-p^{-\alpha-n}}>0.
\end{align*}
For every sufficiently negative $K$, we therefore have
\begin{align*}
    N_{\boldsymbol{H}}^{-}(\lambda_{K+1})=0.
\end{align*}
Condition \eqref{eq:condition_Polya_conjecture_holds}, and hence the
equivalent Schur condition, is automatic at all such scales. There are
only finitely many remaining integers $K<K_0$.
\end{proof}
\begin{example}[The single-missing-digit self-similar domain]
Assume that
\begin{align*}
    \#I=p^n-1,\qquad \#J_{\mathrm{in}}=1,
    \qquad J_{\mathrm{ex}}=\emptyset,
\end{align*}
and let $\Omega$ be the domain constructed in Section~5. Then
\begin{align*}
    \Omega=\mathbb Z_p^n\setminus\partial\Omega,\qquad
    d=\dim_M(\partial\Omega)=\log_p(p^n-1),\qquad
    |\partial\Omega|=0.
\end{align*}
For the Dirichlet realization associated with
$H_0^{\alpha/2}(\Omega)$, Pólya's conjecture holds if and only if
\begin{equation}
\label{eq:Polya-self-similar-strict}
    \sum_{m=1}^{\infty}
    \frac{1}{p^n(1-p^{-n})^m(p^{\alpha m}-1)}
    \leq
    \frac{p^{n+\alpha}-1}{p^\alpha-1}.
\end{equation}
A divergent series is understood not to satisfy
\eqref{eq:Polya-self-similar-strict}. In particular, Pólya's
conjecture holds whenever
\begin{align*}
    p^\alpha\geq\frac{2p^n}{p^n-1}.
\end{align*}
\end{example}

\begin{proof}
The boundary consists of the points
\begin{align*}
    \sum_{k=0}^{\infty}\boldsymbol a_kp^k,
    \qquad \boldsymbol a_k\in I\quad\text{for every }k.
\end{align*}
At level $k$, exactly $(p^n-1)^k$ balls meet $\partial\Omega$, so
$|\partial\Omega|=0$ and
$d=\log_p(p^n-1)$. Since
$\Omega=\mathbb Z_p^n\setminus\partial\Omega$, it also follows that
$|\Omega|=1$.

We first verify all scales other than $K=-1$. For $K\geq0$,
\begin{align*}
    \mathcal{B}_K
    =\{B_K(\boldsymbol a):\boldsymbol a\in I^K\}.
\end{align*}
Every $B_K(\boldsymbol a)\in\mathcal{B}_K$ contains the unique maximal
ball of level $K+1$ obtained by appending the unique digit in
$J_{\mathrm{in}}$. Consequently,
\begin{equation}
\begin{aligned}
\label{eq:self-similar-fine-scale-counts}
    &\#\mathcal{B}_K=(p^n-1)^K,\\
    &|\Omega_{>K}|=(1-p^{-n})^K,\\
    &|\Omega_{>K}|p^{nK}=\#\mathcal{B}_K.
\end{aligned}  
\end{equation}
The fixed-scale geometric lemma therefore verifies the strict Schur
criterion for every $K\geq0$. Notice that the first two identities in
\eqref{eq:self-similar-fine-scale-counts} are asserted only for
$K\geq0$; similarly, $n_K=(p^n-1)^{K-1}$ only for $K\geq1$.

For $K\leq-2$, domain monotonicity and the spectrum of the unit ball
give
\begin{align*}
    \inf\sigma(D_{\Omega,\mathcal D}^{\alpha})
    \geq\frac{1-p^{-n}}{1-p^{-\alpha-n}}
    >p^{-\alpha}\geq p^{\alpha(K+1)}.
\end{align*}
Indeed,
\begin{align*}
    \frac{1-p^{-n}}{1-p^{-\alpha-n}}-p^{-\alpha}
    =\frac{(p^\alpha-1)((p^n-1)p^\alpha-1)}
    {p^\alpha(p^{n+\alpha}-1)}>0.
\end{align*}
Thus the criterion is automatic for every $K\leq-2$.

It remains to consider $K=-1$. The only ball in
$\mathcal{B}_{-1}$ contains $\mathbb Z_p^n$ as its only level-zero
sub-ball meeting $\Omega$, and $\mathbb Z_p^n$ still meets both
$\Omega$ and $\Omega^c$. Hence this ball contains no maximal ball of
level zero, and the fine-scale lemma does not apply. Moreover,
\begin{align*}
    |\Omega_{>-1}|=|\Omega|=1,
    \qquad
    \left[|\Omega_{>-1}|p^{-n}\right]
    =[p^{-n}]=0.
\end{align*}
The strict Schur criterion at this scale is therefore equivalent to
the nonnegativity of $\mathcal Q_{-1}$.

Finite linear combinations of maximal-ball indicators form a core of
$V\cap H_0^{\alpha/2}(\Omega)$. Average such a finite combination over
all permutations of the $p^n-1$ continuing sub-balls contained in each
mixed ball. These permutations preserve Haar measure, all ball
distances, and $\Omega$; the averaging is therefore an orthogonal
projection for both the $L^2$ norm and the Dirichlet form. All
components removed by this averaging occur with the positive
coefficients $p^{\alpha(k+1)}-1$. More precisely, the original function
is the orthogonal sum of the averaged function and the components
removed by the successive averages, and the corresponding value of
$\mathcal Q_{-1}$ is the sum of its value on the averaged function and
the nonnegative quantities
\begin{align*}
    (p^{\alpha(k+1)}-1)
    \|\Delta_{k+1}u\|_{L^2(\mathbb Q_p^n)}^2
\end{align*}
for the removed components. Hence, for a prescribed average at level
zero, the infimum of $\mathcal Q_{-1}$ is obtained from functions having
the same average on all mixed balls of a fixed level.
The averaging is first performed on each finite form core. Since these
averaging projections are contractions in both norms, they extend through
the form-norm closure to the strict form domain.

Denote this common average by $a_k$ and put
$d_k=a_k-a_{k+1}$. For a finite linear combination one has
$a_k=0$ for all sufficiently large $k$, and hence
$a_0=\sum\limits_{k\geq0}d_k$. In each mixed ball of level $k$, the average on each of the $p^n-1$ mixed sub-balls of level $k+1$ is $a_{k+1}$, whereas the average on the remaining sub-ball, which is contained in $\Omega$, is $p^na_k-(p^n-1)a_{k+1}$. Since the level-$k$ mixed balls have total
measure $(1-p^{-n})^k$, it follows that
\begin{align*}
    \|\Delta_{k+1}u\|_{L^2(\mathbb Q_p^n)}^2
    =(p^n-1)(1-p^{-n})^k|d_k|^2.
\end{align*}
The part with levels not exceeding zero is
$a_0\boldsymbol 1_{\mathbb Z_p^n}$, and the constant eigenvalue of the
Dirichlet realization on $\mathbb Z_p^n$ gives
\begin{align*}
    \sum_{j\leq0}(p^{\alpha j}-1)
    \|\Delta_ju\|_{L^2(\mathbb Q_p^n)}^2
    =\left(\frac{1-p^{-n}}{1-p^{-\alpha-n}}-1\right)|a_0|^2.
\end{align*}
Combining these orthogonal decompositions, according to the ball
inclusions $B_{k+1}\subseteq B_k$, gives
\begin{equation}
\label{eq:Q-minus-one-radial}
\begin{aligned}
    \mathcal Q_{-1}[u]
    =&\left(\frac{1-p^{-n}}{1-p^{-\alpha-n}}-1\right)|a_0|^2+\sum_{k=0}^{\infty}(p^n-1)(1-p^{-n})^k(p^{\alpha(k+1)}-1)|d_k|^2.
\end{aligned}
\end{equation}
This identity first holds on the averaged finite core and then on its closure in the strict form norm. The orthogonal components removed by the averaging contribute precisely the nonnegative terms described above. Consequently, nonnegativity on the averaged closure is equivalent to nonnegativity on the
full strict form domain.

For finitely supported sequences satisfying $\sum d_k=a_0$, the
weighted Cauchy--Schwarz inequality gives
\begin{align*}
\sum_{k=0}^{\infty}(p^n-1)(1-p^{-n})^k
(p^{\alpha(k+1)}-1)|d_k|^2
\geq
\frac{|a_0|^2}{\displaystyle
\sum_{k=0}^{\infty}
\frac{1}{(p^n-1)(1-p^{-n})^k(p^{\alpha(k+1)}-1)}}
\end{align*}
when the denominator is finite. Equality is approached by taking $d_k$
proportional to the reciprocal weights and then truncating and normalizing.
If the denominator diverges, the same finite truncations show that the
infimum is zero. Since the finite sequences correspond to functions in the
strict form core, no almost-everywhere support relaxation is used in this
minimization.

It follows from \eqref{eq:Q-minus-one-radial} that
$\mathcal Q_{-1}$ is nonnegative if and only if
\begin{align*}
\sum_{k=0}^{\infty}
\frac{1}{(p^n-1)(1-p^{-n})^k(p^{\alpha(k+1)}-1)}
\leq\frac{p^{n+\alpha}-1}{p^\alpha-1}.
\end{align*}
Using
\begin{align*}
    (p^n-1)(1-p^{-n})^{m-1}
    =p^n(1-p^{-n})^m
\end{align*}
and changing the index by $m=k+1$ gives precisely
\eqref{eq:Polya-self-similar-strict}. The verification of every other
scale proves the asserted equivalence. The series converges exactly
when $p^\alpha(1-p^{-n})>1$, thus the convergence threshold is already contained in \eqref{eq:Polya-self-similar-strict}.

Finally,
\begin{align*}
    p^{\alpha m}-1
    \geq p^{\alpha(m-1)}(p^\alpha-1)
\end{align*}
implies
\begin{align*}
    \sum_{m=1}^{\infty}
    \frac{1}{p^n(1-p^{-n})^m(p^{\alpha m}-1)}
    \leq
    \frac{p^\alpha}
    {p^n(p^\alpha-1)((1-p^{-n})p^\alpha-1)}.
\end{align*}
If $p^\alpha\geq2p^n/(p^n-1)$, the last expression is at most one.
Indeed, with $x=p^\alpha$,
\begin{align*}
    p^n(x-1)((1-p^{-n})x-1)-x
    =(p^n-1)x^2-2p^nx+p^n\geq0,
\end{align*}
because the quadratic polynomial on the right is increasing for
$x\geq2p^n/(p^n-1)$ and has value $p^n$ at the left endpoint. Moreover,
$(p^{n+\alpha}-1)/(p^\alpha-1)>p^n\geq2$.
This proves the stated sufficient condition.
\end{proof}
\section*{Acknowledgements}
The author is grateful to his supervisor, Bobo Hua, for his invaluable
guidance and constant support throughout this research.
\section*{Data availability}
Data sharing is not applicable to this article as no datasets were generated or analysed during the current study.
\section*{Competing interests}
The author has no relevant interests to disclose.

\bibliographystyle{unsrt}
\bibliography{references}

@article {MR1511670,
    AUTHOR = {Weyl, Hermann},
     TITLE = {Das asymptotische {V}erteilungsgesetz der {E}igenwerte
              linearer partieller {D}ifferentialgleichungen (mit einer
              {A}nwendung auf die {T}heorie der {H}ohlraumstrahlung)},
   JOURNAL = {Math. Ann.},
  FJOURNAL = {Mathematische Annalen},
    VOLUME = {71},
      YEAR = {1912},
    NUMBER = {4},
     PAGES = {441--479},
      ISSN = {0025-5831,1432-1807},
   MRCLASS = {99-04},
  MRNUMBER = {1511670},
       DOI = {10.1007/BF01456804},
       URL = {https://doi.org/10.1007/BF01456804},
}

@article {MR1580880,
    AUTHOR = {Weyl, H.},
     TITLE = {\"Uber die {R}andwertaufgabe der {S}trahlungstheorie und
              asymptotische {S}pektralgesetze},
   JOURNAL = {J. Reine Angew. Math.},
  FJOURNAL = {Journal f\"ur die Reine und Angewandte Mathematik. [Crelle's
              Journal]},
    VOLUME = {143},
      YEAR = {1913},
     PAGES = {177--202},
      ISSN = {0075-4102,1435-5345},
   MRCLASS = {99-04},
  MRNUMBER = {1580880},
       DOI = {10.1515/crll.1913.143.177},
       URL = {https://doi.org/10.1515/crll.1913.143.177},
}

@article {MR564330,
    AUTHOR = {Ivrii, V. Ja.},
     TITLE = {The second term of the spectral asymptotics for the
              {L}aplace-{B}eltrami operator on manifolds with boundary and
              for elliptic operators acting in vector bundles},
   JOURNAL = {Dokl. Akad. Nauk SSSR},
  FJOURNAL = {Doklady Akademii Nauk SSSR},
    VOLUME = {250},
      YEAR = {1980},
    NUMBER = {6},
     PAGES = {1300--1302},
      ISSN = {0002-3264},
   MRCLASS = {58G25 (35P20)},
  MRNUMBER = {564330},
MRREVIEWER = {P.\ G\"unther},
}

@incollection {MR573438,
    AUTHOR = {Melrose, R. B.},
     TITLE = {Weyl's conjecture for manifolds with concave boundary},
 BOOKTITLE = {Geometry of the {L}aplace operator ({P}roc. {S}ympos. {P}ure
              {M}ath., {U}niv. {H}awaii, {H}onolulu, {H}awaii, 1979)},
    SERIES = {Proc. Sympos. Pure Math.},
    VOLUME = {XXXVI},
     PAGES = {257--274},
 PUBLISHER = {Amer. Math. Soc., Providence, RI},
      YEAR = {1980},
      ISBN = {0-8218-1439-7},
   MRCLASS = {58G25 (35P20)},
  MRNUMBER = {573438},
MRREVIEWER = {P.\ G\"unther},
}

@incollection {MR556688,
    AUTHOR = {Berry, M. V.},
     TITLE = {Distribution of modes in fractal resonators},
 BOOKTITLE = {Structural stability in physics ({P}roc. {I}nternat.
              {S}ymposia {A}ppl. {C}atastrophe {T}heory and {T}opological
              {C}oncepts in {P}hys., {I}nst. {I}nform. {S}ci., {U}niv.
              {T}\"ubingen, {T}\"ubingen, 1978)},
    SERIES = {Springer Ser. Synergetics},
    VOLUME = {4},
     PAGES = {51--53},
 PUBLISHER = {Springer, Berlin},
      YEAR = {1979},
   MRCLASS = {70K50},
  MRNUMBER = {556688},
       DOI = {10.1007/978-3-642-67363-4\_7},
       URL = {https://doi.org/10.1007/978-3-642-67363-4_7},
}

@article {MR834484,
    AUTHOR = {Brossard, Jean and Carmona, Ren\'e},
     TITLE = {Can one hear the dimension of a fractal?},
   JOURNAL = {Comm. Math. Phys.},
  FJOURNAL = {Communications in Mathematical Physics},
    VOLUME = {104},
      YEAR = {1986},
    NUMBER = {1},
     PAGES = {103--122},
      ISSN = {0010-3616,1432-0916},
   MRCLASS = {58G25 (35P20 58G32)},
  MRNUMBER = {834484},
MRREVIEWER = {Pierre\ B\'erard},
       URL = {http://projecteuclid.org/euclid.cmp/1104114935},
}

@article {MR994168,
    AUTHOR = {Lapidus, Michel L.},
     TITLE = {Fractal drum, inverse spectral problems for elliptic operators
              and a partial resolution of the {W}eyl-{B}erry conjecture},
   JOURNAL = {Trans. Amer. Math. Soc.},
  FJOURNAL = {Transactions of the American Mathematical Society},
    VOLUME = {325},
      YEAR = {1991},
    NUMBER = {2},
     PAGES = {465--529},
      ISSN = {0002-9947,1088-6850},
   MRCLASS = {58G25 (28A75 35J25 35P20)},
  MRNUMBER = {994168},
MRREVIEWER = {Harold\ Donnelly},
       DOI = {10.2307/2001638},
       URL = {https://doi.org/10.2307/2001638},
}

@article {MR1189091,
    AUTHOR = {Lapidus, Michel L. and Pomerance, Carl},
     TITLE = {The {R}iemann zeta-function and the one-dimensional
              {W}eyl-{B}erry conjecture for fractal drums},
   JOURNAL = {Proc. London Math. Soc. (3)},
  FJOURNAL = {Proceedings of the London Mathematical Society. Third Series},
    VOLUME = {66},
      YEAR = {1993},
    NUMBER = {1},
     PAGES = {41--69},
      ISSN = {0024-6115,1460-244X},
   MRCLASS = {58G18 (11M06 35P05 47F05 58G25)},
  MRNUMBER = {1189091},
MRREVIEWER = {Robert\ Brooks},
       DOI = {10.1112/plms/s3-66.1.41},
       URL = {https://doi.org/10.1112/plms/s3-66.1.41},
}

@article {MR1356166,
    AUTHOR = {Lapidus, Michel L. and Pomerance, Carl},
     TITLE = {Counterexamples to the modified {W}eyl-{B}erry conjecture on
              fractal drums},
   JOURNAL = {Math. Proc. Cambridge Philos. Soc.},
  FJOURNAL = {Mathematical Proceedings of the Cambridge Philosophical
              Society},
    VOLUME = {119},
      YEAR = {1996},
    NUMBER = {1},
     PAGES = {167--178},
      ISSN = {0305-0041,1469-8064},
   MRCLASS = {58G25 (35P20)},
  MRNUMBER = {1356166},
MRREVIEWER = {Vadim\ A.\ Ka\u imanovich},
       DOI = {10.1017/S0305004100074053},
       URL = {https://doi.org/10.1017/S0305004100074053},
}

@book {MR66321,
    AUTHOR = {Polya, G.},
     TITLE = {Induction and analogy in mathematics. {M}athematics and
              plausible reasoning, vol. {I}},
 PUBLISHER = {Princeton University Press, Princeton, NJ},
      YEAR = {1954},
     PAGES = {xvi+280},
   MRCLASS = {02.0X},
  MRNUMBER = {66321},
MRREVIEWER = {E.\ W.\ Beth},
}

@article {MR129219,
    AUTHOR = {P\'olya, G.},
     TITLE = {On the eigenvalues of vibrating membranes},
   JOURNAL = {Proc. London Math. Soc. (3)},
  FJOURNAL = {Proceedings of the London Mathematical Society. Third Series},
    VOLUME = {11},
      YEAR = {1961},
     PAGES = {419--433},
      ISSN = {0024-6115,1460-244X},
   MRCLASS = {73.35 (35.80)},
  MRNUMBER = {129219},
MRREVIEWER = {H.\ F.\ Weinberger},
       DOI = {10.1112/plms/s3-11.1.419},
       URL = {https://doi.org/10.1112/plms/s3-11.1.419},
}

@article {MR4635832,
    AUTHOR = {Filonov, Nikolay and Levitin, Michael and Polterovich, Iosif
              and Sher, David A.},
     TITLE = {P\'olya's conjecture for {E}uclidean balls},
   JOURNAL = {Invent. Math.},
  FJOURNAL = {Inventiones Mathematicae},
    VOLUME = {234},
      YEAR = {2023},
    NUMBER = {1},
     PAGES = {129--169},
      ISSN = {0020-9910,1432-1297},
   MRCLASS = {35P15 (11P21 35P20)},
  MRNUMBER = {4635832},
MRREVIEWER = {Gianpaolo\ Piscitelli},
       DOI = {10.1007/s00222-023-01198-1},
       URL = {https://doi.org/10.1007/s00222-023-01198-1},
}

@article {MR5028853,
    AUTHOR = {Filonov, Nikolay and Levitin, Michael and Polterovich, Iosif
              and Sher, David A.},
     TITLE = {P\'olya's conjecture for {D}irichlet eigenvalues of annuli},
   JOURNAL = {J. Lond. Math. Soc. (2)},
  FJOURNAL = {Journal of the London Mathematical Society. Second Series},
    VOLUME = {113},
      YEAR = {2026},
    NUMBER = {2},
     PAGES = {Paper No. e70425, 37},
      ISSN = {0024-6107,1469-7750},
   MRCLASS = {35P15 (11P21 33C10 35P20 65N25)},
  MRNUMBER = {5028853},
       DOI = {10.1112/jlms.70425},
       URL = {https://doi.org/10.1112/jlms.70425},
}

@article {MR3900781,
    AUTHOR = {Kwa\'snicki, Mateusz and Laugesen, Richard S. and Siudeja,
              Bart\l omiej A.},
     TITLE = {P\'olya's conjecture fails for the fractional {L}aplacian},
   JOURNAL = {J. Spectr. Theory},
  FJOURNAL = {Journal of Spectral Theory},
    VOLUME = {9},
      YEAR = {2019},
    NUMBER = {1},
     PAGES = {127--135},
      ISSN = {1664-039X,1664-0403},
   MRCLASS = {35R11 (35P15)},
  MRNUMBER = {3900781},
       DOI = {10.4171/JST/242},
       URL = {https://doi.org/10.4171/JST/242},
}

@article {MR226394,
    AUTHOR = {Taibleson, Mitchell},
     TITLE = {Harmonic analysis on {$n$}-dimensional vector spaces over
              local fields. {I}. {B}asic results on fractional integration},
   JOURNAL = {Math. Ann.},
  FJOURNAL = {Mathematische Annalen},
    VOLUME = {176},
      YEAR = {1968},
     PAGES = {191--207},
      ISSN = {0025-5831,1432-1807},
   MRCLASS = {46.35 (42.00)},
  MRNUMBER = {226394},
MRREVIEWER = {M.\ Hasumi},
       DOI = {10.1007/BF02052825},
       URL = {https://doi.org/10.1007/BF02052825},
}

@article {MR971464,
    AUTHOR = {Vladimirov, V. S.},
     TITLE = {Generalized functions over the field of {$p$}-adic numbers},
   JOURNAL = {Uspekhi Mat. Nauk},
  FJOURNAL = {Akademiya Nauk SSSR i Moskovskoe Matematicheskoe Obshchestvo.
              Uspekhi Matematicheskikh Nauk},
    VOLUME = {43},
      YEAR = {1988},
    NUMBER = {5(263)},
     PAGES = {17--53, 239},
      ISSN = {0042-1316},
   MRCLASS = {46P05 (22E50 46Fxx 81E99)},
  MRNUMBER = {971464},
MRREVIEWER = {Witold\ Wi\polhk es\l aw},
       DOI = {10.1070/RM1988v043n05ABEH001924},
       URL = {https://doi.org/10.1070/RM1988v043n05ABEH001924},
}

@article {MR1092528,
    AUTHOR = {Vladimirov, V. S.},
     TITLE = {On the spectrum of some pseudodifferential operators over the
              field of {$p$}-adic numbers},
   JOURNAL = {Algebra i Analiz},
  FJOURNAL = {Algebra i Analiz},
    VOLUME = {2},
      YEAR = {1990},
    NUMBER = {6},
     PAGES = {107--124},
      ISSN = {0234-0852},
   MRCLASS = {47S10 (11S80 47A10 47G30 47N50 81Q10)},
  MRNUMBER = {1092528},
MRREVIEWER = {Eugen\ Belokolos},
}

@article {MR1209031,
    AUTHOR = {Kochubei, A. N.},
     TITLE = {The differentiation operator on subsets of the field of
              {$p$}-adic numbers},
   JOURNAL = {Izv. Ross. Akad. Nauk Ser. Mat.},
  FJOURNAL = {Izvestiya Rossiiskoi Akademii Nauk. Seriya Matematicheskaya},
    VOLUME = {56},
      YEAR = {1992},
    NUMBER = {5},
     PAGES = {1021--1039},
      ISSN = {1607-0046,2587-5906},
   MRCLASS = {11S80 (11Z50 47S10 81Q99)},
  MRNUMBER = {1209031},
       DOI = {10.1070/IM1993v041n02ABEH002262},
       URL = {https://doi.org/10.1070/IM1993v041n02ABEH002262},
}

@article {MR1252936,
    AUTHOR = {Haran, Shai},
     TITLE = {Quantizations and symbolic calculus over the {$p$}-adic
              numbers},
   JOURNAL = {Ann. Inst. Fourier (Grenoble)},
  FJOURNAL = {Universit\'e{} de Grenoble. Annales de l'Institut Fourier},
    VOLUME = {43},
      YEAR = {1993},
    NUMBER = {4},
     PAGES = {997--1053},
      ISSN = {0373-0956,1777-5310},
   MRCLASS = {22E35 (11S80 22E50)},
  MRNUMBER = {1252936},
MRREVIEWER = {David\ Manderscheid},
       DOI = {10.5802/aif.1363},
       URL = {https://doi.org/10.5802/aif.1363},
}

@article {MR3708861,
    AUTHOR = {Chac\'on-Cort\'es, L. F. and Z\'u\~niga-Galindo, W. A.},
     TITLE = {Heat traces and spectral zeta functions for {$p$}-adic
              {L}aplacians},
   JOURNAL = {Algebra i Analiz},
  FJOURNAL = {Rossi\u iskaya Akademiya Nauk. Algebra i Analiz},
    VOLUME = {29},
      YEAR = {2017},
    NUMBER = {3},
     PAGES = {144--166},
      ISSN = {0234-0852},
   MRCLASS = {35S05 (35R11)},
  MRNUMBER = {3708861},
MRREVIEWER = {Luigi\ Rodino},
       DOI = {10.1090/spmj/1505},
       URL = {https://doi.org/10.1090/spmj/1505},
}

@book {MR3793137,
    AUTHOR = {Khrennikov, Andrei Yu. and Kozyrev, Sergei V. and
              Z\'u\~niga-Galindo, W. A.},
     TITLE = {Ultrametric pseudodifferential equations and applications},
    SERIES = {Encyclopedia of Mathematics and its Applications},
    VOLUME = {168},
 PUBLISHER = {Cambridge University Press, Cambridge},
      YEAR = {2018},
     PAGES = {xv+237},
      ISBN = {978-1-107-18882-2},
   MRCLASS = {35-02 (35S05 37P99 47-02 47G30 60H99 60J99)},
  MRNUMBER = {3793137},
MRREVIEWER = {Anatoly\ N.\ Kochubei},
       DOI = {10.1017/9781316986707},
       URL = {https://doi.org/10.1017/9781316986707},
}

@article {MR1555153,
    AUTHOR = {Ostrowski, Alexander},
     TITLE = {\"Uber einige {L}\"osungen der {F}unktionalgleichung
              {$\psi(x)\cdot\psi(x)=\psi(xy)$}},
   JOURNAL = {Acta Math.},
  FJOURNAL = {Acta Mathematica},
    VOLUME = {41},
      YEAR = {1916},
    NUMBER = {1},
     PAGES = {271--284},
      ISSN = {0001-5962,1871-2509},
   MRCLASS = {99-04},
  MRNUMBER = {1555153},
       DOI = {10.1007/BF02422947},
       URL = {https://doi.org/10.1007/BF02422947},
}

@article {MR1918846,
    AUTHOR = {Kozyrev, S. V.},
     TITLE = {Wavelet theory as {$p$}-adic spectral analysis},
   JOURNAL = {Izv. Ross. Akad. Nauk Ser. Mat.},
  FJOURNAL = {Izvestiya Rossiiskoi Akademii Nauk. Seriya Matematicheskaya},
    VOLUME = {66},
      YEAR = {2002},
    NUMBER = {2},
     PAGES = {149--158},
      ISSN = {1607-0046,2587-5906},
   MRCLASS = {42C40 (47A10 47S10)},
  MRNUMBER = {1918846},
MRREVIEWER = {B.\ S.\ Rubin},
       DOI = {10.1070/IM2002v066n02ABEH000381},
       URL = {https://doi.org/10.1070/IM2002v066n02ABEH000381},
}
\noindent Yaojia Sun, 261101800043@m.fudan.edu.cn\\
\emph{School of Mathematical Sciences, Fudan University, Shanghai, 200433, P.R. China}\\[-8pt]
\end{document}